\documentclass[AMS,Times1COL]{WileyNJDv5} 

\articletype{Article Type}%

\received{Date Month Year}
\revised{Date Month Year}
\accepted{Date Month Year}
\startpage{1}

\hypersetup{
    colorlinks = true,
    linkcolor = {blue},
    citecolor = {blue},
    linkbordercolor = {blue}
}

\usepackage{mathtools} 

\usepackage{comment}

\usepackage{bm} 
\usepackage[nameinlink]{cleveref} 

\allowdisplaybreaks 

\Crefname{figure}{Figure}{Figures}
\crefname{figure}{Figure}{Figures}
\crefname{lemma}{Lemma}{Lemmas}
\crefname{section}{Section}{Sections}
\crefname{appendix}{Appendix}{Appendices}
\crefname{definition}{Definition}{Definitions}
\crefname{theorem}{Theorem}{Theorems}
\crefname{corollary}{Corollary}{Corollaries}
\crefname{footnote}{Footnote}{Footnotes}
\crefname{assumption}{Assumption}{Assumptions}
\crefname{remark}{Remark}{Remarks}
\Crefname{remark}{Remark}{Remarks}

\newcommand{\wh}[1]{\widehat{#1}} 					
\newcommand{\wt}[1]{\widetilde{#1}}

\DeclareFontFamily{U}{mathx}{}
\DeclareFontShape{U}{mathx}{m}{n}{<-> mathx10}{}
\DeclareSymbolFont{mathx}{U}{mathx}{m}{n}
\DeclareMathAccent{\wc}{0}{mathx}{"71} 
\DeclareMathAccent{\wbar}{0}{mathx}{"73} 

\DeclareMathOperator{\diag}{diag}
\renewcommand{\i}[0]{\ensuremath{\mathrm{i}}}

\newcommand{\snw}[0]{\ensuremath{s_{\textrm{nw}}}}
\newcommand{\sn}[0]{\ensuremath{s_{\textrm{n}}}}
\newcommand{\sne}[0]{\ensuremath{s_{\textrm{ne}}}}
\newcommand{\sw}[0]{\ensuremath{s_{\textrm{w}}}}
\newcommand{\scc}[0]{\ensuremath{{s_{\textrm{c}}}}}
\newcommand{\se}[0]{\ensuremath{s_{\textrm{e}}}}
\newcommand{\ssw}[0]{\ensuremath{s_{\textrm{sw}}}}
\renewcommand{\ss}[0]{\ensuremath{s_{\textrm{s}}}}
\newcommand{\sse}[0]{\ensuremath{s_{\textrm{se}}}}

\newcommand{\csnw}[0]{\ensuremath{\wc{s}_{\textrm{nw}}}}
\newcommand{\csn}[0]{\ensuremath{\wc{s}_{\textrm{n}}}}
\newcommand{\csne}[0]{\ensuremath{\wc{s}_{\textrm{ne}}}}
\newcommand{\csw}[0]{\ensuremath{\wc{s}_{\textrm{w}}}}
\newcommand{\cscc}[0]{\ensuremath{{\wc{s}_{\textrm{c}}}}}
\newcommand{\cse}[0]{\ensuremath{\wc{s}_{\textrm{e}}}}
\newcommand{\cssw}[0]{\ensuremath{\wc{s}_{\textrm{sw}}}}
\newcommand{\css}[0]{\ensuremath{\wc{s}_{\textrm{s}}}}
\newcommand{\csse}[0]{\ensuremath{\wc{s}_{\textrm{se}}}}

\begin{document}

\title{Overlapping Schwarz methods are not anisotropy-robust multigrid smoothers}

\author[1]{Oliver A. Krzysik}

\author[1]{Ben S. Southworth}

\author[2]{Bobby Philip}

\authormark{KRZYSIK \textsc{et al.}}
\titlemark{Overlapping Schwarz methods are not anisotropy-robust multigrid smoothers}

\address[1]{\orgdiv{Theoretical Division}, \orgname{Los Alamos National Laboratory}, \orgaddress{\state{NM}, \country{USA}}}

\address[2]{\orgdiv{X Computational Physics}, \orgname{Los Alamos National Laboratory}, \orgaddress{\state{NM}, \country{USA}}}

\corres{Corresponding author Oliver A. Krzysik. \email{okrzysik@lanl.gov}}



\abstract[Abstract]{We analyze overlapping multiplicative Schwarz methods as smoothers in the geometric multigrid solution of two-dimensional anisotropic diffusion problems.
For diffusion equations, it is well known that the smoothing properties of point-wise smoothers, such as Gauss--Seidel, rapidly deteriorate as the strength of anisotropy increases.
On the other hand, global smoothers based on line smoothing are known to generally provide good smoothing for diffusion problems, independent of the anisotropy strength.
A natural question is whether global methods are really necessary to achieve good smoothing in such problems, or whether it can be obtained with locally overlapping block smoothers using sufficiently large blocks and overlap.
Through local Fourier analysis and careful numerical experimentation, we show that global methods are indeed necessary to achieve anisotropy-robust smoothing.
Specifically, for any fixed block size bounded sufficiently far away from the global domain size, we find that the smoothing properties of overlapping multiplicative Schwarz rapidly deteriorate with increasing anisotropy, irrespective of the amount of overlap between blocks.
Moreover, our results indicate that anisotropy-robust smoothing requires blocks of diameter ${\cal O}(\epsilon^{-1/2})$ for anisotropy ratio $\epsilon \in (0,1]$.
}


\maketitle

\renewcommand\thefootnote{\fnsymbol{footnote}}

\renewcommand{\thefootnote}{\arabic{footnote}} 

\setcounter{footnote}{1}

%
\section{Introduction}

In this paper we are interested in the scalar anisotropic diffusion equation with source term $f$ and symmetric positive semi-definite (SPSD) diffusion tensor $\mathcal{T}$:
\begin{subequations} \label{eq:rot}
\begin{align} \label{eq:rot-pde}
- \nabla \cdot \mathcal{T} \nabla u
&= f(\bm{x}),
&&\bm{x} \in \Omega,
\\
u &= 0, &&\bm{x} \in \partial \Omega,
\end{align}
\end{subequations}
where it is assumed that $\mathcal{T}$ diffuses strongly in certain directions and weakly in others. Such problems occur in many different fields. Here we are particularly motivated by heat conduction in magnetized plasmas and magnetic confinement fusion applications, where the dominant diffusion in $\mathcal{T}$ aligns with the magnetic field, and anisotropy ratios aligned with vs. orthogonal to the magnetic field can reach ratios of $10^{12}$ and potentially be non-grid aligned, e.g. \cite{hoelzl2021jorek,vogl2023mesh,chacon2024asymptotic,Green-etal-2024,Wimmer-etal-2024}.

The solution of such anisotropic equations remains a major challenge within larger physics simulations. Multigrid methods are known to be ideal fast, parallel solvers for discretizations of many elliptic and parabolic partial differential equations (PDEs) \cite{Briggs-etal-2000}. Unfortunately, efficient and parallel multigrid solvers (or any solvers for that matter) for realistic levels of anisotropy as seen in magnetized plasmas remains a largely open question to our knowledge. Multigrid methods consist of two complementary processes, smoothing and coarse-grid correction. For a successful multilevel method, smoothing is typically point-wise or local in nature, and coarse-grid correction must attenuate error not attenuated by smoothing. Unfortunately, for high levels of anisotropy it has proven difficult to use local smoothing and represent the remaining error on a coarse grid. This is because \eqref{eq:rot} has modes that are smooth in the direction of strong diffusion and highly oscillatory in the direction of weak diffusion, which have very small eigenvalues and are typically not attenuated by local smoothing, but are also difficult to represent on a coarse grid \cite{Yavneh_1998}.

In the context of geometric multigrid solvers, it is well-known that point-wise smoothing methods like Jacobi and Gauss--Seidel break down for anisotropic diffusion problems because they only smooth error in the direction of strong connections.
As a result, the error after smoothing cannot be accurately represented on a grid that has been coarsened uniformly in all directions, rendering the associated coarse-grid correction ineffective.
One potential remedy for this problem is to use a semi-coarsened grid, coarsening only in the direction of anisotropy, e.g. \cite{oliveira2012optimized,prieto2001parallel,mulder1989new}. Such coarsening assumes a grid-aligned anisotropy however, and can inhibit parallel performance. Algebraic multigrid (AMG) methods can automatically semi-coarsen in non-grid-aligned directions on the first level (e.g., see \cite[Fig. 3]{sivas2021air}), and specialized methods have been developed for anisotropic problems, e.g. \cite{schroder2012smoothed,manteuffel2017root,gee2009new,Brandt-etal-2015}, but these methods still fail on realistic non-grid-aligned anisotropies \cite{Wimmer-etal-2024}. In \cite{Wimmer-etal-2024} we develop a specialized mixed discretization and block preconditioner for highly anisotropic diffusion  built on AMG for hyperbolic transport operators \cite{manteuffel2018nonsymmetric,manteuffel2019nonsymmetric}, but the solver studies therein are limited to anisotropic fields with open field lines, which is also not realistic in practical fusion settings.

Another popular remedy for anisotropic problems, more closely aligned with the scope of this paper, is to introduce line or plane smoothing, e.g. \cite{hofhaus1996alternating,oosterlee1995convergence,alcouffe1981multi,prieto2001parallel,mulder1990note,Llorente-Melson-1998}, where subsets of degrees of freedom (DOFs) aligned with a given grid dimension are inverted together. Such an approach is robust with respect to anisotropy strength for grid-aligned anisotropy, but has several pit falls. The resulting lines/planes are typically global, stretching from one side of the domain to the other \cite{Trottenberg_etal_2001,Briggs-etal-2000,Wesseling-1992}, and identifying appropriate lines on non-regular meshes is not straightforward, although it is possible in certain situations \cite{philip-chartier-2012,Green-etal-2024}. By nature of this global  coupling, line and (particularly) plane smoothing is also computationally expensive due to inverting large sets of DOFs, and poor in parallel (although recent work has tried to address the parallel performance issues \cite{reisner2020scalable}). 
For non-grid-aligned anisotropy, common practice is to utilize alternating line smoothing \cite{Wesseling-1992}, but this is not always robust as a smoother.

Recently, there has been a resurgent interest in patch- or block-based smoothers for multigrid methods, where overlapping local blocks of DOFs (i.e., blocks that do not globally couple DOFs across the domain) are inverted, e.g. \cite{farrell2019augmented,adler2021monolithic,laakmann2022augmented,abu2023monolithic,Riva-etal-2019,Riva-etal-2021,Riva-etal-2024}. Such smoothers fall into the class of overlapping Schwarz methods, and can be either additive or multiplicative in the way the unknowns are updated. The class of overlapping Schwarz methods is broad, encompassing many classical multigrid smoothing methods. For example, included in the special case of zero overlap are the aforementioned line smoothing, and point-wise Jacobi and Gauss--Seidel, while methods with non-zero overlap such as Vanka are commonplace for PDE systems like the Stokes and Navier--Stokes equations \cite{MacLachlan-Oosterlee-2011,He-etal-2021,Vanka1986, guy-philip-griffith-2015}.
Here we are particularly motivated by recent results from \cite{farrell2019augmented,laakmann2022augmented}, where local patch-based smoothing is able to provide robust geometric multigrid convergence on high-Reynolds (low viscosity) flow problems, which are notoriously difficult. Because advection is related to highly anisotropic diffusion in its directional nature and need for line-smoothing and/or semi-coarsening techniques, a natural question is whether anisotropy-robust smoothing can be achieved for anisotropic diffusion using local overlapping block smoothers.
To the best of our knowledge this question has not been addressed previously, and doing so is the objective of this paper.

Our primary methodology for studying this problem is local Fourier analysis (LFA), utilizing the general overlapping Schwarz framework developed by MacLachlan and Oosterlee \cite{MacLachlan-Oosterlee-2011}.
We do note that while overlapping Schwarz smoothers are most commonly used in the context of PDE systems, they have been analyzed for scalar PDEs using LFA, particularly for isotropic Poisson problems \cite{MacLachlan-Oosterlee-2011,Riva-etal-2024,Riva-etal-2019,Riva-etal-2021,Greif-He-2023}.
Our primary findings are that overlapping multiplicative Schwarz methods using local subdomains are \emph{not} robust smoothers for anisotropies in \eqref{eq:rot}. That is, for any fixed local subdomain size, we find that the smoothing properties deteriorate without bound as the level of anisotropy increases, regardless of the amount of overlap between neighboring subdomains.
This result is somewhat counter intuitive because the application of such a smoother requires (depending on the subdomain size and overlap) a significantly larger amount of work relative to point-wise and line smoothers, and, moreover, it is surprising because block overlapping Schwarz methods have a reputation for being ``strong smoothers'' due to their effectiveness for PDE systems.
On the other hand, our results show that anisotropy-robust smoothing can be obtained for grid-aligned anisotropies, provided the subdomains grow as ${\cal O}(\epsilon^{-1/2})$ for an anisotropy ratio of $\epsilon \in (0, 1]$ in \eqref{eq:rot}.
Such an approach may prove practical for mild to moderate anisotropies in developing efficient multigrid methods with local relaxation. For large anisotropies as seen in fusion applications, blocks of size ${\cal O}(\epsilon^{-1/2})$ are effectively not local, and the use of overlapping Schwarz smoothers is neither robust nor practical.

The remainder of this paper is structured as follows.
\Cref{sec:prelims} presents preliminaries, including the two-dimensional model problems we consider, the multigrid method used to solve them, and an overview of overlapping multiplicative Schwarz methods.
\Cref{sec:LFA} uses LFA to study the efficacy of maximally overlapping Schwarz smoothing on two-dimensional subdomains of size $2 \times 2$ and $\ell \times 1$ for $\ell \in \mathbb{N}$. Both grid-aligned and non-grid-aligned problems are considered, with theoretical results focusing on grid-aligned problems.
\Cref{sec:num-res} presents numerical results to complement and extend the LFA findings from the previous section, including considering subdomains of size $\ell \times \ell$ for $\ell > 2$, as well as the impacts of using less overlap.
Conclusions are drawn in \cref{sec:con}.
Note that to improve readability, derivations for theoretical results in \cref{sec:LFA} can be found in the appendices.

%
\section{Preliminaries}
\label{sec:prelims}

%
\subsection{Model problems}
\label{sec:model-problems}

Returning to \eqref{eq:rot}, we restrict ourselves to a two-dimensional problem with domain $(x, y) \in \Omega = [0,1]^2$ given by the unit square, and diffusion tensor $\mathcal{T} \coloneqq Q {\cal D} Q^\top$. Given anisotropy ratio $\epsilon \in [0, 1]$, we consider the anisotropy matrix ${\cal D} = [1,\, 0;\, 0,\, \epsilon]$, and the orthogonal rotation matrix $Q = [\cos \theta, \, - \sin \theta;\, \sin \theta,\, \cos \theta]$, with rotation angle $\theta \in [0, \pi/2]$. This yields a two-dimensional rotated anisotropic diffusion equation in \eqref{eq:rot}.  Multiplying out the derivative terms yields the equivalent form
\begin{align} \label{eq:rot-expanded}
\nabla \cdot (Q {\cal D} Q^\top \nabla u)
=
\alpha u_{xx} + \beta u_{yy} + \gamma u_{xy},
\quad
\alpha = \cos^2 \theta + \epsilon \sin^2 \theta,
\quad
\beta = \epsilon \cos^2 \theta + \sin^2 \theta,
\quad
\gamma = 2(1 - \epsilon) \cos \theta \sin \theta.
\end{align}
Hence, when $\theta = 0$ \eqref{eq:rot-pde} reduces to the standard grid-aligned anisotropic diffusion equation
\begin{align} \label{eq:aligned}
    -u_{xx} - \epsilon u_{yy} = f(x,y), \quad (x, y) \in \Omega.
\end{align}

More generally, the PDE operator \eqref{eq:rot-expanded} corresponds to the anisotropic diffusion operator $-u_{\xi \xi } - \epsilon u_{\eta \eta}$, with coordinate system $(\xi, \eta)$ obtained by rotating the $(x, y)$-coordinate system about an angle $\theta$ from the positive $x$-axis; see, e.g., \cite{Brandt-etal-2015}.
As such, for any $\epsilon \in [0, 1)$, connections in the $\xi$ direction are strong relative to those in the $\eta$ direction.

For PDE \eqref{eq:rot-expanded}, we consider two discretizations.
In each case, the domain $\Omega \in [0,1]^2$ is discretized with $n+1$ nodes in both the $x$ and $y$ directions, equispaced by a distance of $h = 1/n$, so that the resulting mesh has a total of $(n+1) \times (n+1)$ nodes.
The first discretization is based on standard second-order finite differences (FDs) for the Laplacian term in \eqref{eq:rot-expanded}, and upwind FDs for the mixed $u_{xy}$ term. This upwinding is similar to recent finite-element work, where transport based-stabilization of mixed derivatives is shown to improve stability and accuracy in high anisotropic regimes \cite{Wimmer-etal-2024,Wimmer-etal-2025}.
Letting $A_h$ denote the discretization matrix on a mesh with $h$, and $[A_h]$ its stencil, the FD discretization is (see, e.g., \cite[Appendix 8.5.1]{Trottenberg_etal_2001}, \cite[p. 116]{Wesseling-1992}, \cite[Section 4.1]{Brandt-etal-2015})
\begin{align}  \label{eq:rot-fd}
[A_h] 
&= 
\frac{\alpha}{h^2} [-1 \quad \hphantom{-}2 \quad -1] 
+ 
\frac{\beta}{h^2}
\begin{bmatrix}
-1 \\
\hphantom{-} 2 \\
-1
\end{bmatrix}
-
\frac{\gamma}{2 h^2}
\begin{bmatrix}
\hphantom{-}0 & -1 & \hphantom{-}1 \\
-1 & \hphantom{-}2 & -1 \\
\hphantom{-}1 & -1 & \hphantom{-}0
\end{bmatrix}.
\end{align}
Recall that we limit the rotation angle in \eqref{eq:rot-expanded} to the interval $\theta \in [0, \pi/2]$ rather than $[0, 2\pi)$; this restriction is to ensure the upwindedness of this discretization, but, by symmetry, is without loss of generality to our results. 
The second discretization we consider is that based on bilinear finite elements (FEs), and has the stencil (see e.g., \cite[Section 4.1]{Brandt-etal-2015})
\begin{align} \label{eq:rot-fe}
[A_h]
&=
\alpha
\begin{bmatrix}
-\frac{1}{6} & \hphantom{-}\frac{1}{3} & -\frac{1}{6} \\[1.5ex]
-\frac{2}{3} & \hphantom{-}\frac{4}{3} & -\frac{2}{3} \\[1.5ex]
-\frac{1}{6} & \hphantom{-}\frac{1}{3} & -\frac{1}{6} \\
\end{bmatrix}
+ \beta
\begin{bmatrix}
-\frac{1}{6} & -\frac{2}{3} & -\frac{1}{6} \\[1.5ex]
\hphantom{-}\frac{1}{3} & \hphantom{-}\frac{4}{3} & \hphantom{-}\frac{1}{3} \\[1.5ex]
-\frac{1}{6} & -\frac{2}{3} & -\frac{1}{6} \\
\end{bmatrix}
-
\frac{\gamma}{4}
\begin{bmatrix}
-1 & \hphantom{-}0 & \hphantom{-}1 \\
\hphantom{-}0 & \hphantom{-}0 & \hphantom{-}0 \\
\hphantom{-}1 & \hphantom{-}0 & -1
\end{bmatrix}.
\end{align}
%

%
\subsection{Geometric multigrid method}
\label{sec:mg}

Suppose that discretizing the PDE \eqref{eq:rot} on an $(n_0 + 1) \times (n_0+1)$ mesh yields the linear system $A_0 \bm{x} = \bm{b}$, with $N_0 = (n_0-1)^2$ DOFs, after boundary conditions have been eliminated. 
For simplicity, we suppose that $n_0$ is a power of two.
We solve $A_0 \bm{x} = \bm{b}$ using a geometric multigrid method based on full coarsening; that is, the number of mesh points in each direction is coarsened by a factor of two, so that the coarse grid consists of $(n_1 + 1) \times (n_1 + 1) = (n_0/2+1) \times (n_0/2 + 1)$ points, which, after eliminating boundary conditions results in $N_1 = (n_1-1)^2$ coarse-grid DOFs.

We use the standard grid-transfer operators of bi-linear interpolation $P \colon \mathbb{R}^{N_1} \to \mathbb{R}^{N_0}$ and its adjoint (up to scaling), i.e., full weighting,  $P^\top \colon \mathbb{R}^{N_0} \to \mathbb{R}^{N_1}$ for restriction; see \cite{Wesseling-1992,Trottenberg_etal_2001,Briggs-etal-2000}.
We use the Galerkin coarse-grid operator $A_1 = P^{\top} A_0 P \in \mathbb{R}^{N_1}$.
For the grid-aligned problem \eqref{eq:aligned} this is a perfectly good coarse-grid operator.
However, this coarse-grid operator does not result in a $\theta$-robust coarse-grid correction for the rotated problem \eqref{eq:rot}, particularly when used recursively in the multilevel setting, because it is not accurate enough on smooth characteristic components \cite{Yavneh_1998,Wienands_Joppich_2005}.
Nonetheless, we elect to use this coarse operator since we mostly focus on the grid-aligned case, and because we are primarily interested in analyzing the smoother, which can be done  largely independently of the coarse-grid correction.
The solver employs one smoothing step pre and post the coarse-grid correction using some form of multiplicative Schwarz, with specific details to be discussed later.

In our experiments we use a multilevel solver based on recursively applying the above two-level solver.
Numerically we implement this multigrid method using various computational kernels provided by PyAMG \cite{Bell-etal-2023}, including its implementation of overlapping multiplicative Schwarz.
%

%
\subsection{Overlapping multiplicative Schwarz methods}

Consider covering the discretized two-dimensional domain $\Omega$ with potentially overlapping subdomains $S_{ij}$, so that $\Omega = \cup_{ij} S_{ij}$.
In this work we assume these subdomains are rectangular, with constant size and constant overlap across the domain. A subdomain is denoted to have size $\ell \times m$ if it covers $\ell \geq 1$ and $m \geq 1$ DOFs in the $x$- and $y$-directions, respectively; see examples in \cref{fig:subdomain-2x2,fig:subdomain-ellx1}.
A multiplicative Schwarz iteration sequentially iterates over all of the subdomains $S_{ij}$, updating the DOFs in each such subdomain so that their residual is zero immediately after being updated.

Suppose we are at the stage of the smoothing sweep where we are about to  update DOFs in subdomain $S_{ij}$, then let $\bm{u}^{\rm old} \in \mathbb{R}^N$ be the current approximation of all DOFs on the mesh, and $\bm{r}^{\rm old} = A \bm{e}^{\rm old} \in \mathbb{R}^N$ the associated residual for error $\bm{e}^{\rm old} = \bm{u} - \bm{u}^{\rm old}$.
Let us introduce the boolean restriction operator $V_{ij}^\top \colon \mathbb{R}^N \to \mathbb{R}^{\ell m}$ that selects DOFs in $S_{ij}$ from global vectors such as $\bm{u}$, and $\bm{e}$, and let us denote variables on a given subdomain using calligraphic letters subscripted by the subdomain's indices. For example, ${\cal U}_{ij} = V_{ij}^\top \bm{u}$, and ${\cal E}_{ij} = V_{ij}^\top \bm{e} \in \mathbb{R}^{\ell m}$ are the restrictions of $\bm{u}$ and $\bm{e}$ onto $S_{ij}$, respectively.
Immediately after updating DOFs in $S_{ij}$, denote the new global error and solution vector as $\bm{e}^{\rm new}$ and $\bm{u}^{\rm new}$, respectively. 
Now, on $S_{ij}$, the updated solution vector is obtained by employing the correction that sets its residual is zero:
\begin{align} \label{eq:SCH-update}
    {\cal U}^{\rm new}_{ij} 
    =
    {\cal U}^{\rm old}_{ij}
    +
    A_{i,j}^{-1}
    {\cal R}^{\rm old}_{ij}
    \quad
    \Longrightarrow
    \quad
    A_{i,j} 
    \left(
    {\cal E}^{\rm old}_{ij} - {\cal E}^{\rm new}_{ij}
    \right)
    =
    (A {\cal E}^{\rm old})_{ij}.
\end{align}
In \eqref{eq:SCH-update} the matrix $A_{ij}$ is the principle submatrix of $A$ corresponding to subdomain $S_{ij}$, which can be written as $A_{ij} = V_{ij}^\top A V_{ij}$.
As such, the new error in subdomain $S_{ij}$ may be written as
\begin{align}
    {\cal E}^{\rm new}_{ij}
    =
    {\cal E}^{\rm old}_{ij}
    -
    A_{i,j}^{-1} 
    (A {\cal E}^{\rm old})_{ij}
    =
    {\cal E}^{\rm old}_{ij} - \left( (V_{ij}^\top A V_{ij})^{-1} V_{ij}^\top A \right)
    \bm{e}^{\rm old}.
\end{align}
The resulting new global error $\bm{e}^{\rm new}$ is equal to $\bm{e}^{\rm old}$ everywhere except for locations of DOFs that are in $S_{ij}$. Hence, post updating $S_{ij}$, the global error vector is
\begin{align}\label{eq:err-1subdomain}
    \bm{e}^{\rm new} 
    = 
    (\bm{e}^{\rm old} - V_{ij} {\cal E}^{\rm old}_{ij}) + V_{ij} {\cal E}^{\rm new}_{ij} 
    =
    \left[ I - V_{ij} (V_{ij}^\top A V_{ij})^{-1} V_{ij}^\top A \right] \bm{e}^{\rm old}.
\end{align}
One may then obtain the global error update corresponding to solving sequentially over each subdomain by taking the product of error propagation \eqref{eq:err-1subdomain} over all subdomains:
\begin{align} \label{eq:SCH-E-prop}
    \bm{e}^{\rm new}
    =
    \left(
    \prod \limits_{ij}
    \left[ I - V_{ij} (V_{ij}^\top A V_{ij})^{-1} V_{ij}^\top A \right]
    \right)
    \bm{e}^{\rm old}.
\end{align}

In this work we order the subdomains block-row-wise lexicographically, so that we begin with the south-west most subdomain on $\Omega$, and then proceed by sweeping to the east and then to the north.

%
\section{LFA for maximally overlapping Schwarz smoothing}
\label{sec:LFA}

In this section we study certain maximally overlapping multiplicative Schwarz smoothers using LFA.
The error propagators \eqref{eq:SCH-E-prop} for these smoothers have one-dimensional invariant subspaces given by the Fourier modes, making their LFA relatively tractable. 
In contrast, invariant subspaces of \eqref{eq:SCH-E-prop} for smoothers using less overlap are more complicated, and for this reason we elect to numerically investigate those smoothers later in \cref{sec:num-res}.
This decision is further justified because the ``strength'' of a smoother with a given subdomain size increases with overlap (see \cref{sec:num-res}), and our results in this section show that, in a certain sense, even maximally overlapping smoothers are not robust with respect to anisotropy strength.  

The remainder of this section is organized as follows.
\cref{sec:LFA-basics} provides a brief overview of LFA, \cref{sec:LFA:2x2} outlines the LFA procedure for Schwarz with $2 \times 2$ subdomains, and \cref{sec:LFA:2x2-align,sec:LFA:2x2-rot} present accompanying results for anisotropic diffusion in the grid-aligned and rotated cases, respectively.
\cref{sec:LFA:ellx1} discusses the LFA procedure for $\ell \times 1$ subdomains, and \cref{sec:LFA:ellx1-theory} presents accompanying theoretical results.
%

%
\subsection{LFA basics}
\label{sec:LFA-basics}

LFA, originally introduced by Brandt \cite{Brandt_1977}, is a widely used technique for both quantitative and qualitative analysis of multigrid methods.
Here we briefly summarize the basics of the methodology, and refer the reader to the textbooks \cite{Wesseling-1992,Trottenberg_etal_2001,Wienands_Joppich_2005} for thorough introductions to the subject.

The error propagator of a two-grid method with one pre- and post-smoothing step (as considered in this work) with smoother $S$ takes the form
\begin{align} \label{eq:E}
    E = S (I - P A_1^{-1} P^\top A_0) S.
\end{align}
Here $P$ is the interpolation operator that transfers corrections from the coarse to fine grid, its adjoint $P^\top$ is responsible for transferring residuals from the fine to coarse grid, and $A_{0}, A_{1}$ represent the fine- and coarse-grid matrices, respectively.
For PDE problems with constant coefficients discretized on structured uniform grids, it is typically the case that each of the components in \eqref{eq:E} is locally (block) Toeplitz, at least away from boundaries, and this observation forms the motivation for LFA.
In LFA, each of the (assumed) constant-stencil components in \eqref{eq:E} is represented on a periodic infinite grid in order to make use of the fact that any infinite-dimensional (block) Toeplitz operator is formally diagonalized by the Fourier modes. This allows one to characterize the action of the components in \eqref{eq:E} on individual Fourier modes or small subspaces of them.
To this end, we introduce infinite grids for the fine and coarse representation of the problem, respectively,
\begin{align}
    \bm{G}_0 = \{ (x,y)=(i h, jh) \, \colon \, (i,j) \in \mathbb{Z}^2 \},
    \quad
    \bm{G}_1 = \{ (x,y)=(2i h, 2j h) \, \colon \, (i,j) \in \mathbb{Z}^2 \}.
\end{align}
Then, on these infinite grids consider the Fourier modes
\begin{align} \label{eq:modes}
    \bm{\varphi}_k = e^{\i \bm{\omega} \cdot \bm{x} / h}, 
    \quad
    \bm{x} \in \bm{G}_k,
    \quad \bm{\omega} = (\omega_1, \omega_2) \in [-\pi/2, 3\pi/2)^2,
\end{align}
for $k\in\{0,1\}$ with two-dimensional frequency $\bm{\omega}$ varying continuously. 
The Fourier modes are $2 \pi$ periodic with respect to $\bm{\omega}$, so that $\bm{\omega}$ may span any domain of length $2 \pi$, but the one we choose here provides notational simplifications.
The frequency domain is partitioned into two disjoint subsets, low and high frequencies, motivated by the fact that high-frequency modes are not representable on a factor-two coarsened grid.
The low frequencies are $\bm{\omega} \in  [-\pi/2, \pi/2)^2$, and the high frequencies are $\bm{\omega} \in [-\pi/2, 3\pi/2)^2 \setminus [-\pi/2, \pi/2)^2$.

In this work we only apply LFA to smoothers $S$ that are invariant with respect to the Fourier modes \eqref{eq:modes}, meaning that their action on these modes is characterized by
\begin{align} \label{eq:S-symbol-def}
    S \bm{\varphi}_0(\bm{\omega}) = \wt{s}(\bm{\omega}) \bm{\varphi}_0(\bm{\omega}),
\end{align}
with scalar $\wt{s}(\bm{\omega}) \in \mathbb{C}$ the so-called symbol of $S$.
As discussed in \cref{sec:mg}, in this paper we consider multigrid methods using full geometric coarsening, which means that we are concerned with the behavior of the smoother over the high frequency Fourier modes, since these are not representable on the coarse grid, so that error in their direction should be handled by the smoother. 
As such, the natural quality measure of the smoother is the so-called smoothing factor, defined as follows.
\begin{definition}[Smoothing factor] \label{def:mu}
The \textit{smoothing factor} of any smoother $S$ with scalar Fourier symbol $\wt{s}(\omega_1, \omega_2)$ satisfying \eqref{eq:S-symbol-def} is
\begin{align} \label{eq:mu-def}
    \mu := \max_{(\omega_1, \omega_2) \in [-\pi/2, 3\pi/2)^2 \setminus [-\pi/2, \pi/2)^2} | \wt{s}(\omega_1, \omega_2) |. 
\end{align}
\end{definition}
The smoothing factor for any convergent smoother cannot exceed unity; however, \textit{good} smoothing requires a smoothing factor that is reasonably bounded above by unity.
For example, point-wise lexicographic Gauss--Seidel with $\mu = 0.5$  for isotropic Poisson discretized with second-order FDs (see, e.g., \cite[p. 105]{Trottenberg_etal_2001}) is considered a good smoother.

In this work we assess the quality of smoothers with respect to the anisotropy ratio $\epsilon$ in \eqref{eq:rot}, giving rise to the following definition of robustness.
\begin{definition}[$\epsilon$-robustness] \label{def:robust}
    A class of smoothers is robust with respect to anisotropy ratio $\epsilon \in [0,1]$ if its smoothing factor satisfies $\mu(\epsilon) \leq \eta < 1$ for all $\epsilon \in [0,1]$ and constant $\eta$.
\end{definition}

We now consider the two-grid error propagator \eqref{eq:E}.
Recall from \cref{sec:mg} that we use bilinear interpolation for $P$ in \eqref{eq:E} for interpolating functions from $\bm{G}_1$ to $\bm{G}_0$. Bilinear interpolation is well known to interpolate $\bm{\varphi}_1$ to a linear combination of four fine-grid modes, known as the harmonics.  
Specifically, for any low frequency $\bm{\omega}_{\rm low} \in [-\pi/2,\pi/2)^2$ the associated space of harmonics is ${\cal H}(\bm{\omega}_{\rm low}) = 
\textrm{span}\big( 
\bm{\varphi}_0( \bm{\omega}_{\rm low} ),
\bm{\varphi}_0( \bm{\omega}_{\rm low} + (\pi,\pi) ),
\bm{\varphi}_0( \bm{\omega}_{\rm low}  + (\pi,0) ),
\bm{\varphi}_0( \bm{\omega}_{\rm low}  + (0,\pi) ) \big)$.
The consequence of this fact is that for any $\bm{\omega} \in [-\pi/2,\pi/2)^2$, ${\cal H}(\bm{\omega})$ is an invariant subspace of the error propagator $E$ in \eqref{eq:E}, 
with associated matrix-valued Fourier symbol\footnote{Technically, the domain of $\wh{E}(\bm{\omega})$ excludes frequencies $\bm{\omega}$ for which either the fine- or coarse-grid symbols vanish, but we abuse notation by not writing this out explicitly. In our numerical implementation of the two-grid convergence factor \cref{def:E-tg-def} we exclude any $\bm{\omega}$ from the optimization for which $|\wh{A}_1(2\bm{\omega})|$ or $\Vert \wh{A}_0(\bm{\omega}) \Vert$ are smaller than $10^{-14}$.}
\begin{align} \label{eq:Ewh}
    \wh{E}(\bm{\omega})
    =
    \wh{S}(\bm{\omega})
    [I - \wh{P}(\bm{\omega}) 
    \wh{A}_1^{-1}(2\bm{\omega})
    \wh{P}^\top(\bm{\omega})\wh{A}_0(\bm{\omega})]
    \wh{S}(\bm{\omega})
    \in \mathbb{C}^{4 \times 4}
    ,
    \quad
    \bm{\omega} \in [-\pi/2, \pi/2)^2.
\end{align}
Here each of the hatted matrices is a matrix-valued Fourier symbol of the associated operator characterizing its transformation either from ${\cal H}(\bm{\omega})$ to ${\cal H}(\bm{\omega})$ (i.e., $A_0$ and $S$), from $\textrm{span}( 
\bm{\varphi}_1(2 \bm{\omega} ))$ to ${\cal H}(\bm{\omega})$ (i.e., $P$), or from $\textrm{span}( 
\bm{\varphi}_1( 2\bm{\omega} ))$ to $\textrm{span}( 
\bm{\varphi}_1( 2\bm{\omega} ))$ (i.e., $A_1^{-1}$).
See \cite[Theorem 4.4.1]{Trottenberg_etal_2001} for details.
The overall convergence picture of the two-grid method can then be ascertained by examining the worst-case convergence of $E$ over all of the harmonic subspaces, leading to the definition of the two-grid convergence factor. 
\begin{definition}[Two-grid convergence] \label{def:E-tg-def}
    The two-grid convergence factor of a two-grid method of the form \eqref{eq:E} with symbol $\wh{E}$ given by \eqref{eq:Ewh} is
    \begin{align} \label{eq:E-tg-def}
        \rho^{\rm TG}
        :=
        \max_{\bm{\omega} \in [-\pi/2, \pi/2)^2 } \rho(\wh{E}(\omega_1, \omega_2)), 
    \end{align}
    with $\rho( \cdot )$ denoting the spectral radius.
\end{definition}
To numerically compute either the smoothing factor \eqref{eq:mu-def} or the two-grid convergence factor \eqref{eq:E-tg-def}, we recast them as minimization problems which we solve with Scipy's  \texttt{optimize.shgo} \cite{Scipy-2020}.

The remainder of this section considers LFA for certain maximally overlapping multiplicative Schwarz smoothers applied to general 9-point discretization stencils of the form
\begin{align} \label{eq:9-point}
    [A] = 
    \begin{bmatrix}
        \snw & \sn & \sne \\
        \sw & \scc & \se \\
        \ssw & \ss & \sse
    \end{bmatrix},
\end{align}
noting that both FD \eqref{eq:rot-fd} and FE \eqref{eq:rot-fe}  discretizations are of this form.
Furthermore, we introduce the notation 
\begin{align} \label{eq:9-point-wc}
    [\wc{A}] = 
    \begin{bmatrix}
        \csnw & \csn & \csne \\
        \csw & \cscc & \cse \\
        \cssw & \css & \csse
    \end{bmatrix}
    =
    \begin{bmatrix}
        \snw e^{\i(-\omega_1 + \omega_2)} & \sn e^{+ \i\omega_2} & \sne e^{\i(+\omega_1 + \omega_2)} \\
        \sw e^{-\i \omega_1} & \scc & \se e^{+ \i\omega_1} \\
        \ssw e^{\i(-\omega_1 - \omega_2)} & \ss e^{-\i \omega_2} & \sse e^{\i(+\omega_1 - \omega_2)}
    \end{bmatrix},
\end{align}
motivated as follows. Consider an error vector $\bm{e}$ with $(\bm{e})_{ij} = e^{\i \bm{\omega} \cdot \bm{x}_{ij}/h}$, then the evaluation of the associated residual appearing on the right-hand side of the Schwarz update \eqref{eq:SCH-update} is readily expressed in stencil form as
$
(\bm{r})_{ij} = (A \bm{e})_{ij}
=
e^{\i \bm{\omega} \cdot \bm{x}_{ij}/h} \sum_{p,q} [\wc{A}]_{p,q}
$.
%

%
\subsection{LFA for $2 \times 2$ subdomains} 
\label{sec:LFA:2x2}

Here we outline the procedure for computing the Fourier symbol for  multiplicative Schwarz on maximally overlapped $2 \times 2$ subdomains as it applies to the 9-point stencil \eqref{eq:9-point}.
The procedure we outline can be seen as a straightforward extension of that developed by MacLachlan and Oosterlee \cite{MacLachlan-Oosterlee-2011} for the bilinear FE discretization of Poisson's equation.
In \cite{MacLachlan-Oosterlee-2011} it was proved that the associated error propagator \eqref{eq:SCH-E-prop} is invariant with respect to the Fourier modes \eqref{eq:modes}.

The Fourier symbol is based on the correction equation \eqref{eq:SCH-update} used to update DOFs in subdomain $S_{ij}$.
In the $2 \times 2$ case, $S_{ij}$ contains four grid points, which, suppose for arguments sake, are, $(\bm{x}_{ij}, \bm{x}_{i+1,j}, \bm{x}_{i,j+1}, \bm{x}_{i+1,j+1})$, with the associated DOFs in the same ordering (see the dashed blue square in right of \cref{fig:subdomain-2x2}).
Since the Fourier modes are eigenvectors of this equation, one introduces the ansatz that the error vector ${\cal E}^{\rm x}_{ij} \in \mathbb{C}^4$ in \eqref{eq:SCH-update} consists of a single Fourier mode with frequency $\bm{\omega}$.
Importantly, we must account for the fact that the error at different points in $S_{ij}$ has been updated a different number of times; see \cref{fig:subdomain-2x2}.
Specifically, recalling that subdomains are swept west to east and then south to north, prior to setting the residual on $S_{ij}$ to zero, the DOFs at mesh points $\bm{x}_{ij}$, $\bm{x}_{i+1,j}$, $\bm{x}_{i,j+1}$, and $\bm{x}_{i+1,j+1}$ have been updated three, two, one, and zero times, respectively.
Immediately after the residual on $S_{ij}$ is set to zero, the aforementioned DOFs will each have been updated once more.
As such, the error at a given point on the grid will also be a function of the number of times that point has been updated; hence, the error at some point $\bm{x}_{s,q}$ assumes the form $\alpha_p e^{\i \bm{\omega} \cdot \bm{x}_{s,q}/h}$, with subscript $p$ reflecting the number of times the error at this point has been updated.   
Applying the ansatz that the error vectors in \eqref{eq:SCH-update} consist of a single Fourier mode with frequency $\bm{\omega}$ therefore allows us to write them as
\begin{align} \label{eq:SCH-2x2-Eold-Enew}
{\cal E}^{\rm old}_{ij}
=
e^{\i \bm{\omega} \cdot \bm{x}_{ij} / h}
D
\begin{bmatrix}
    \alpha_3 \\
    \alpha_2 \\
    \alpha_1 \\
    \alpha_0 
\end{bmatrix},
\quad
{\cal E}^{\rm new}_{ij}
=
e^{\i \bm{\omega} \cdot \bm{x}_{ij} / h}
D
\begin{bmatrix}
    \alpha_4 \\
    \alpha_3 \\
    \alpha_2 \\
    \alpha_1 \\
\end{bmatrix}.
\end{align}
Here the coefficients $\alpha_p \in \mathbb{C}$, $p = 1,2,3,4,$ are  unknowns, and are a function of the known coefficient $\alpha_0$, describing the initial error.
For notational simplicity, here we have introduced the matrix $D$ to account for the position of each point in $S_{ij}$ relative to that of $\bm{x}_{ij}$ (called the relative Fourier matrix in \cite{He-etal-2021}):
\begin{align} \label{eq:SCH-2x2-D}
D = \diag
     \left(
        1,
        e^{\i \bm{\omega} \cdot (1, 0) / h},
         e^{\i \bm{\omega} \cdot (0, 1) / h},
       e^{\i \bm{\omega} \cdot (1, 1) / h} 
    \right).
\end{align}
Recalling the ordering of DOFs in $S_{ij}$ and the stencil in \eqref{eq:9-point}, the projected matrix $A_{ij}$ in \eqref{eq:SCH-update} is
\begin{align} \label{eq:SCH-2x2-Aij}
    A_{ij}
    =
    \begin{bmatrix}
        \scc & \se & \sn & \sne \\
        \sw & \scc & \snw & \sn \\
        \ss & \sse & \scc & \se \\
        \ssw & \ss & \sw & \scc 
    \end{bmatrix}.
\end{align}

\begin{figure}[b!]
    \centering
    \includegraphics[scale=1.0]{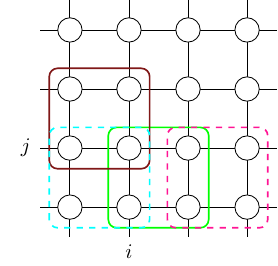}
    \hspace{8ex}
    \includegraphics[scale=1.0]{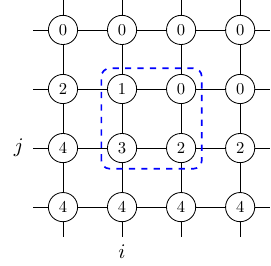}
    \caption{Diagram for maximally overlapping Schwarz on $2 \times 2$ subdomains.
    \textbf{Left:} All $2 \times 2$ subdomains that have been updated immediately prior to updating subdomain $S_{ij}$ and who share DOFs with those in $S_{ij}$.
    \textbf{Right:} Number of times DOFs in the 9-point stencil of DOFs in subdomain $S_{ij}$ (dashed blue rectangle) have been updated immediately prior to updating  $S_{ij}$.
    \label{fig:subdomain-2x2}
    }
\end{figure}

At this stage it is useful to introduce the notation $\bm{\alpha} = (\alpha_4, \alpha_3, \alpha_2, \alpha_1) \in \mathbb{C}^4$ for the unknown coefficients.
To solve for these coefficients using \eqref{eq:SCH-update} we must also consider the residual ${\cal R}^{\rm old}_{ij} = (A {\cal E}^{\rm old})_{ij}$ in \eqref{eq:SCH-update}. 
This quantity is more complicated to express than the errors in \eqref{eq:SCH-2x2-Eold-Enew} because each of its four components are generated by applying the 9-point stencil \eqref{eq:9-point} to the vector ${\cal E}^{\rm old}$ centered at the corresponding corner of $S_{ij}$.
As an example, let us consider the bottom left point in $S_{ij}$.
The right plot in \cref{fig:subdomain-2x2} shows how many times each DOF in the 9-point stencil of this DOF have been updated. Recalling that the error at any grid point $\bm{x}_{s,q}$ that has been updated $p$ times is $\alpha_p e^{\i \bm{\omega} \cdot \bm{x}_{p,q}/h}$, we may write the residual at this point as
\begin{align}\label{eq:r-derivation}
\begin{split}
    r_{ij}^{\rm old} 
    = 
    (A \bm{e}^{\rm old})_{ij} 
    &=
    \ssw e_{i-1,j-1}^{\rm old}
    +
    \ss e_{i,j-1}^{\rm old}
    +
    \ssw e_{i+1,j-1}^{\rm old}
    +
    \sw e_{i-1,j}^{\rm old}
    +
    \scc e_{i,j}^{\rm old}
    +
    \se e_{i+1,j}^{\rm old}
    +
    \snw e_{i-1,j+1}^{\rm old}
    +
    \sn e_{i,j+1}^{\rm old}
    +
    \sse e_{i+1,j+1}^{\rm old}
    \\
    \begin{split}
    &=
    \ssw \alpha_4 e^{\i \bm{\omega} \cdot \bm{x}_{i-1,j-1}/h }
    +
    \ss \alpha_4 e^{\i \bm{\omega} \cdot \bm{x}_{i,j-1}/h }
    +
    \ssw \alpha_4 e^{\i \bm{\omega} \cdot \bm{x}_{i+1,j-1}/h }
    +
    \sw \alpha_4 e^{\i \bm{\omega} \cdot \bm{x}_{i-1,j}/h }
    \\
    &\quad +\scc \alpha_3 e^{\i \bm{\omega} \cdot \bm{x}_{i,j}/h }
    +
    \se \alpha_2 e^{\i \bm{\omega} \cdot \bm{x}_{i+1,j}/h }
    +
    \snw \alpha_2 e^{\i \bm{\omega} \cdot \bm{x}_{i-1,j+1}/h }
    +
    \sn \alpha_1 e^{\i \bm{\omega} \cdot \bm{x}_{i,j+1}/h }
    +
    \sse \alpha_0 e^{\i \bm{\omega} \cdot \bm{x}_{i+1,j+1}/h },
    \end{split}
    \\
    &=
    \left(
    \cssw \alpha_4 
    +
    \css \alpha_4 
    +
    \cssw \alpha_4
    +
    \csw \alpha_4 
    +
    \cscc \alpha_3 
    +
    \cse \alpha_2
    +
    \csnw \alpha_2
    +
    \csn \alpha_1
    +
    \csse \alpha_0 \right) e^{\i \bm{\omega} \cdot \bm{x}_{i,j}/h },
\end{split}
\end{align}
where we have used the stencil notation from \eqref{eq:9-point-wc}.
Using the matrix $D$ in \eqref{eq:SCH-2x2-D} to account for the position of the other three points in $S_{ij}$ relative to that of $\bm{x}_{ij}$, the residuals at these other points can be calculated analogously. 
Ultimately, we may write the residual vector in $S_{ij}$ as
\begin{align}
\nonumber
(A {\cal E}^{\rm old})_{ij}
&=
e^{\i \bm{\omega} \cdot \bm{x}_{ij} / h} D
\begin{bmatrix}
    (\cssw + \css + \csse + \csw) \alpha_4 + (\cscc) \alpha_3 + (\cse + \csnw)  \alpha_2 + (\csn)  \alpha_1 + (\csne)        \alpha_0 \\ 
    (\cssw + \css + \csse)        \alpha_4 + ( \csw) \alpha_3 + (\cscc + \cse)  \alpha_2 + (\csnw) \alpha_1 + (\csn + \csne) \alpha_0 \\ 
    (\cssw)                       \alpha_4 + (\css)  \alpha_3 + (\csse + \csw ) \alpha_2 + (\cscc) \alpha_1 + (\cse + \csnw + \csn + \csne) \alpha_0 \\ 
                                             (\cssw)  \alpha_3 +(\css + \csse)  \alpha_2 + (\csw)  \alpha_1 + (\cscc + \cse + \csnw + \csn + \csne) \alpha_0 
\end{bmatrix}
,
\\
\nonumber
&= 
e^{\i \bm{\omega} \cdot \bm{x}_{ij} / h} D
\left(
\begin{bmatrix}
    \cssw + \css + \csse + \csw & \cscc & \cse + \csnw  &  \csn  \\ 
    \cssw + \css + \csse        &  \csw & \cscc + \cse  & \csnw  \\ 
    \cssw                       & \css  & \csse + \csw & \cscc  \\ 
         0 & \cssw & \css + \csse  & \csw  
\end{bmatrix}
\begin{bmatrix}
    \alpha_4
    \\
    \alpha_3
    \\
    \alpha_2 
    \\
    \alpha_1
\end{bmatrix} 
+
\begin{bmatrix}
    \csne \\
     \csn + \csne \\
       \cse + \csnw + \csn + \csne \\
       \cscc + \cse + \csnw + \csn + \csne
\end{bmatrix}
\alpha_0
\right)
,
\\
\label{eq:SCH-2x2-Rold}
&\equiv
e^{\i \bm{\omega} \cdot \bm{x}_{ij} / h}
D( C \bm{\alpha} + \bm{c} \alpha_0).
\end{align}

Plugging \eqref{eq:SCH-2x2-Eold-Enew} and \eqref{eq:SCH-2x2-Rold} into \eqref{eq:SCH-update}, cancelling the common factor of $e^{\i \bm{\omega} \cdot \bm{x}_{ij} / h} \neq 0$ and rearranging yields the linear system
\begin{align} \label{eq:SCH-2x2-system}
    \underbrace{[A_{ij} D (U - I) - D C]}_{{\cal A}}
    \bm{\alpha}
    =
    \underbrace{[D \bm{c} - A_{ij} D \bm{e}_4]}_{\bm{b}} \alpha_0,
    \quad
    U - I = 
    \begin{bmatrix}
        -1 & \hphantom{-}1 \\
        & -1 & \hphantom{-}1 \\
        & & -1 & \hphantom{-}1 \\
        & & & -1
    \end{bmatrix}.
\end{align}
This system can be solved as 
$\bm{\alpha}
=
{\cal A}^{-1} \bm{b}
{\alpha_0}$,
where each component of ${\cal A}^{-1} \bm{b}$ is the amplification factor for the four, three, two, and once updated DOFs, i.e., $\alpha_4/\alpha_0$, $\alpha_3/\alpha_0$, $\alpha_2/\alpha_0$, and $\alpha_1/\alpha_0$, respectively.
After one full smoothing iteration has been completed, all DOFs will have been updated four times, so that ultimately we care about the first component of this solution vector, and hence the relevant symbol is 
\begin{align} \label{eq:SCH-2x2-symbol-def}
    \wt{s}(\omega_1, \omega_2) = ({\cal A}^{-1} \bm{b})_1,
\end{align}
where $(\bm{z})_1$ denotes the first component of vector $\bm{z}$.
We remark that for additive versions of overlapping Schwarz methods, all DOFs are just updated once only, and for this reason their LFA is relatively simpler than the multiplicative case, but that in the additive case one typically has to contend with finding (via optimization) an appropriate damping factor to get good smoothing \cite{He-etal-2021,Greif-He-2023}.
%

%
\subsection{Grid-aligned anisotropic diffusion on $2 \times 2$ subdomains}
\label{sec:LFA:2x2-align}

Here we use the LFA outlined in the previous section to analyze the smoothing properties of maximally overlapping multiplicative Schwarz on $2 \times 2$ subdomains.
Given the complicated linear system underlying the symbol of this smoother \eqref{eq:SCH-2x2-symbol-def}, theoretically quantifying the associated smoothing properties for general rotation angle $\theta$ and anisotropy ratio $\epsilon$ in \eqref{eq:rot} is going to be difficult.
To this end, in this section we consider only the grid-aligned case (i.e., $\theta = 0$) for which we present theoretical and computational results. Following this, we numerically investigate the rotated case next in \cref{sec:LFA:2x2-rot}.

\begin{lemma} \label{lem:mu-2x2}
    Consider the anisotropic diffusion equation \eqref{eq:aligned} with anisotropy ratio ${\epsilon \in [0,1]}$ discretized with either FDs \eqref{eq:rot-fd} or FEs \eqref{eq:rot-fe}.
    Let $\mu_{2,2}(\epsilon)$ be the smoothing factor \eqref{eq:mu-def} of maximally overlapping multiplicative Schwarz with $2 \times 2$ subdomains (see \cref{sec:LFA:2x2}).
    Then, 
    \begin{align} \label{eq:mu-2x2}
        \mu_{2,2}(\epsilon)
        \geq
        1 - c \epsilon + {\cal O}(\epsilon^2),
    \end{align}
    with $c = 12$ and $c = 19.2$ for the FD and FE discretizations, respectively.
    As such, this smoother is not $\epsilon$-robust in the sense of \cref{def:robust} for either discretization, since ${\lim_{\epsilon \to 0^+} \mu_{2,2}(\epsilon) \geq 1}$.
\end{lemma}

\begin{proof}
    See \cref{app:2x2}.
\end{proof}

The main takeaway from \cref{lem:mu-2x2} is that the smoother is not $\epsilon$-robust in the sense of \cref{def:robust}.
\Cref{fig:LFA-2x2-theory} shows plots of the lower bounds in \cref{lem:mu-2x2} as a function of $\epsilon$, along with the true smoothing factors computed by numerically solving the optimization problem in \eqref{eq:mu-def}.\footnote{Rather than plotting the smoothing factor directly, we plot the difference between it and unity on a logarithmic scale so as to better highlight its asymptotic behavior for small $\epsilon$ and that of the bounds in \cref{lem:mu-2x2}. Convergence factors on a linear scale can be seen from the $\theta = 0$ cross-sections in the right column of \cref{fig:LFA-rot}.
The domain of interest for quantities $X$ shown in \cref{fig:LFA-2x2-theory} is $[0,1]$; note that $1 - X \to 0^+$ corresponds to $X \to 1^-$, and $1 - X \to 1^{-}$ corresponds to $X \to 0^+$. 
} 
For sufficiently small $\epsilon$, it appears that the lower bounds in \cref{lem:mu-2x2} overlap with the true smoothing factor, indicating that the lower bounds in \eqref{eq:mu-2x2} hold with equality, but we do not provide a proof for this.
Shown also in these plots are the associated two-grid convergence factors (see \cref{def:E-tg-def}), which appear to deteriorate qualitatively similar to the smoothing factors. 
For this grid-aligned problem, it is well known that the coarse-grid correction we use is $\epsilon$-robust, so that the poor two-grid convergence shown is purely a consequence of the poor smoothing.
In fact, for sufficiently small $\epsilon$, the two-grid convergence factors appear to be given by $\rho^{\rm TG}_{2,2} = \mu_{2,2}^2(\epsilon) + {\cal O}(\epsilon^2) = 1 - 2c \epsilon + {\cal O}(\epsilon^2)$, with constants $c$ as in \cref{lem:mu-2x2}, although we do not prove this; note that the square occurs here since we use two smoothing steps per two-grid iteration.
%

\begin{figure}[t!]
    \centering
    \includegraphics[width=0.4\linewidth]{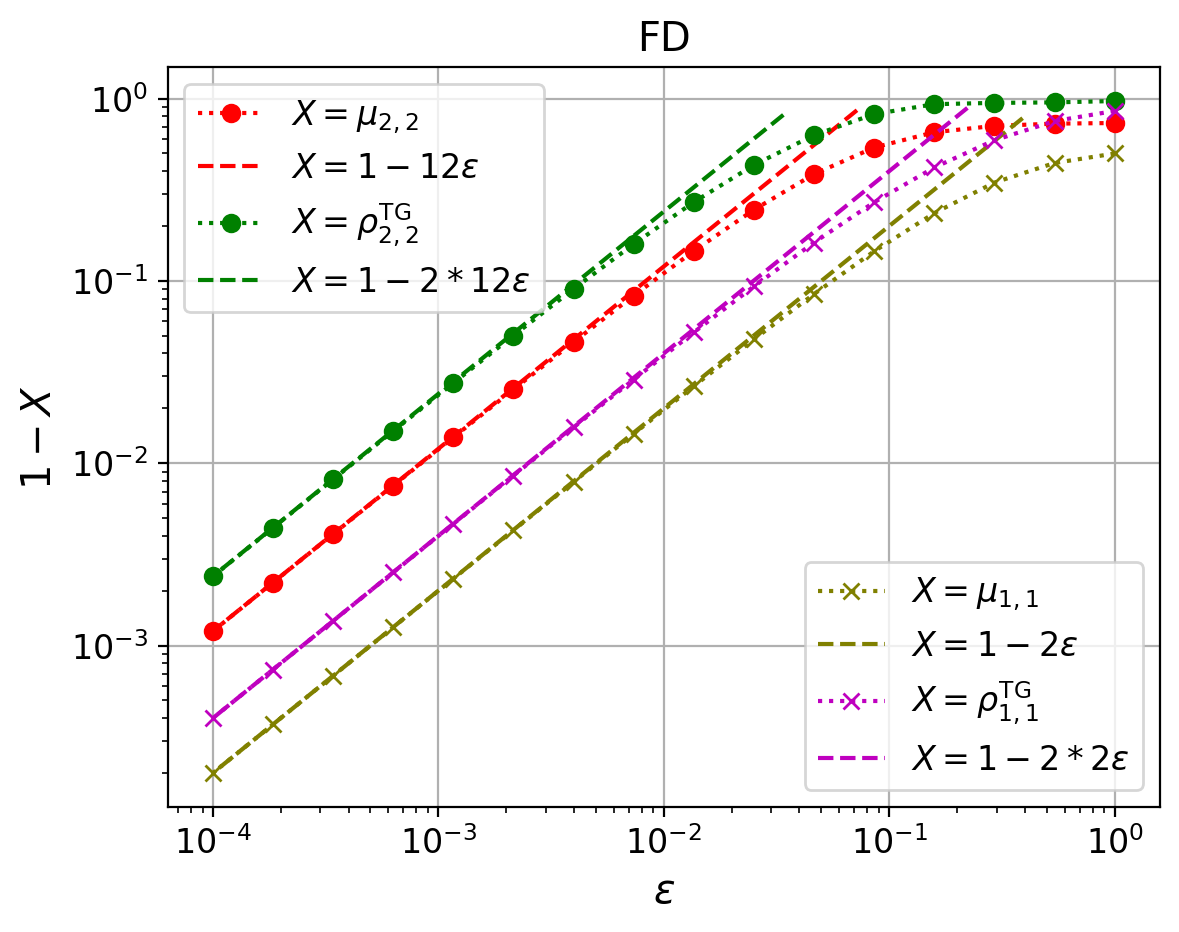}
    \hspace{0ex}
    \includegraphics[width=0.4\linewidth]{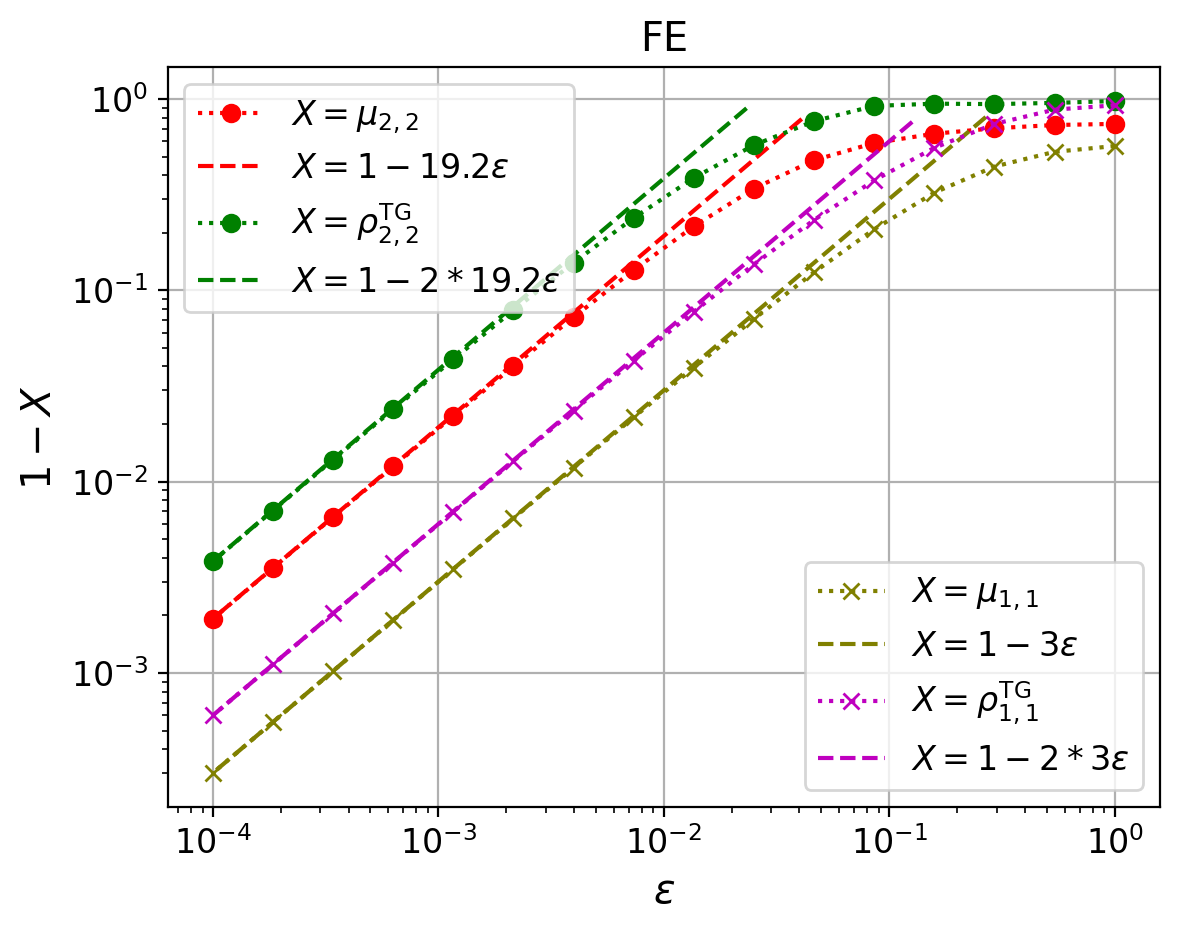}
    \caption{Grid-aligned anisotropic diffusion \eqref{eq:aligned} with anisotropy $\epsilon$. LFA results for maximally overlapped multiplicative Schwarz smoothing applied to anisotropic diffusion with $1 \times 1$ and $2 \times 2$ subdomains. \textbf{Left:} FD discretization \eqref{eq:rot-fd}. \textbf{Right:} FE discretization \eqref{eq:rot-fe}.
    \label{fig:LFA-2x2-theory}
    }
\end{figure}

It is interesting to contrast the smoothing properties of this smoother with those of point-wise Gauss--Seidel, since point-wise Gauss--Seidel is nothing but multiplicative Schwarz on maximally overlapped $1 \times 1$ subdomains.
The smoothing factor for Schwarz on $1 \times 1$ subdomains applied to grid-aligned anisotropic diffusion is
\begin{align} \label{eq:mu-1x1}
    \mu_{1,1}(\epsilon) = 1 - d \epsilon + {\cal O}(\epsilon^2),
\end{align}
with constants $d = 2$ and $d = 3$ for the FD and FE discretizations, respectively (this result is the $\ell = 1$ special case of \cref{thm:1d-smooth}, and we note the FD case also appears in \cite[p. 202]{Hackbusch2003}).
Notice that the relatively smaller ${\cal O}(\epsilon)$ constants in \eqref{eq:mu-1x1} compared to those in \eqref{eq:mu-2x2} indicate that the $1 \times 1$ smoother deteriorates faster than the $2 \times 2$ smoother as $\epsilon \to 0^+$. 
This is indeed confirmed in \cref{fig:LFA-2x2-theory}, which plots the smoothing factors \eqref{eq:mu-1x1} and the associated two-grid convergence factors.
It is quite interesting then to remark that while the $2 \times 2$ method has better smoothing properties than the $1 \times 1$ method, these two methods behave qualitatively similarly for sufficiently small $\epsilon$, in the sense that each has a smoothing and associated two-grid convergence factor behaving as $1 - \delta \epsilon$ for some constants $\delta$.
Based on these observations, one might anticipate that a maximally overlapped smoother extended to $\ell \times \ell$, $\ell > 2$, subdomains may deteriorate analogously with increasing anisotropy, which is something that our numerical investigations in \cref{sec:num-res} support.
%

\begin{figure}[t!]
    \centering
    \includegraphics[width=0.4\linewidth]{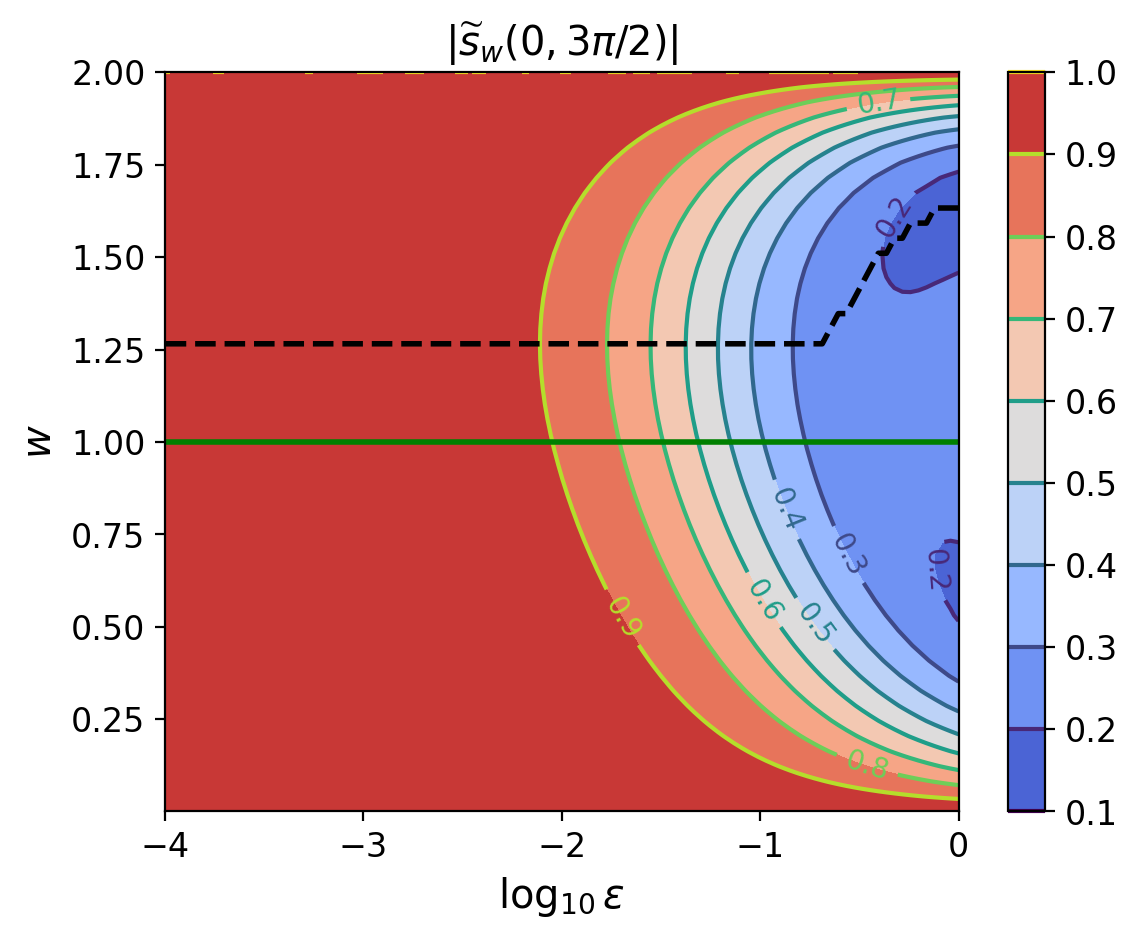}
    \includegraphics[width=0.4\linewidth]{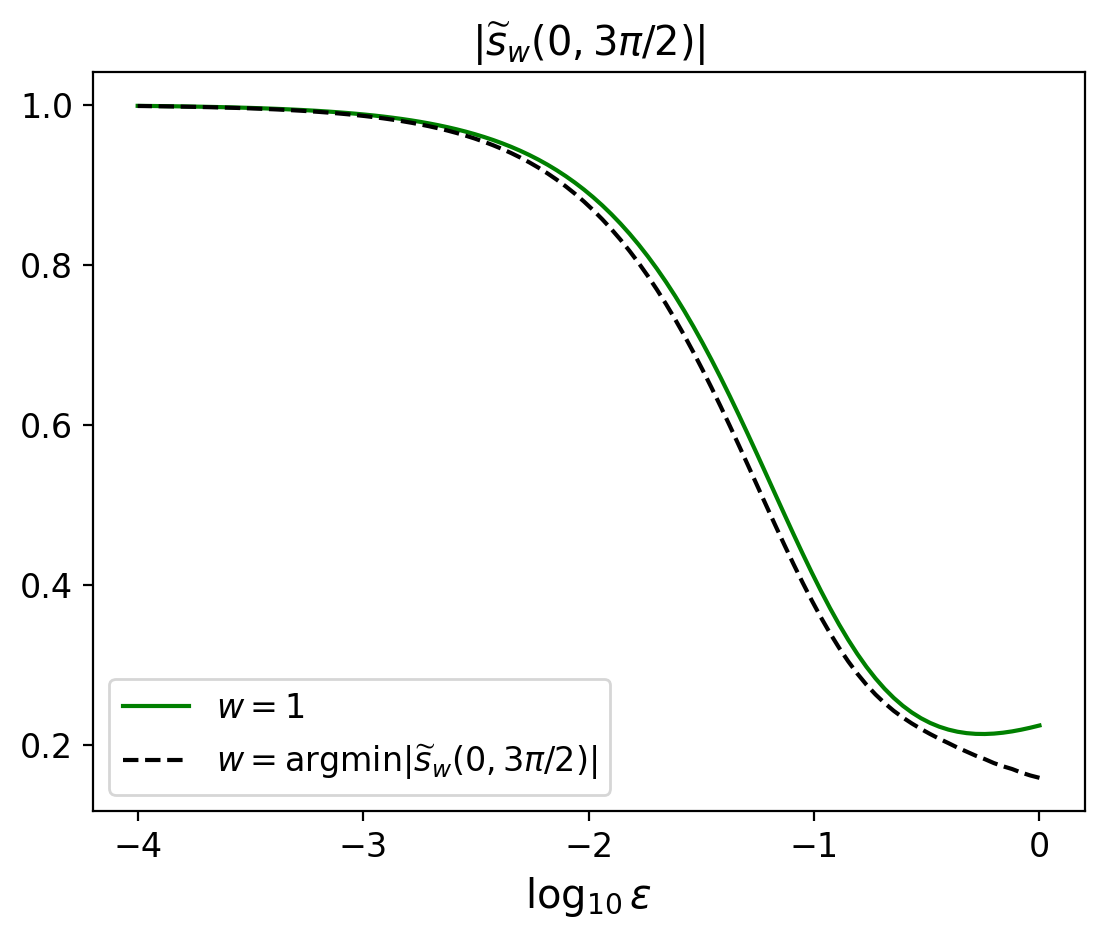}
    \caption{
        Fourier symbol $\wt{s}_w(\omega_1, \omega_2)$ of a weighted version of the Schwarz update \eqref{eq:SCH-update} on maximally overlapped $2 \times 2$ subdomains evaluated at $(\omega_1, \omega_2) = (0, 3\pi/2)$ for the FD discretization \eqref{eq:rot-fd} of grid-aligned anisotropic diffusion \eqref{eq:aligned}.
        \textbf{Left:} Contour of the absolute value of symbol as a function of weight $w$.
        \textbf{Right:} Cross-sections of the absolute value of the symbol at $w = 1$, and at the weight that minimizes the absolute value of the symbol for each $\epsilon$. 
        \label{fig:LFA-overweight}
    }
\end{figure}

\begin{remark}[Coloring]
    For isotropic problems, colored Gauss--Seidel (e.g., red-black) is known to be a more efficient smoother than its lexicographic counterpart \cite{Trottenberg_etal_2001}.
    As such, a more complete comparison might also weigh colored Gauss--Seidel against a colored variant of multiplicative Schwarz, but this is not something we consider here. We do refer the interested reader to \cite{Riva-etal-2021} for an example of where colored overlapping multiplicative Schwarz is used.
\end{remark}

For certain mildly anisotropic diffusion problems Yavneh \cite{Yavneh-1996} showed that successive over-relaxation in red-black ordering with an appropriately optimized weight is a more efficient smoother than its unweighted counterpart, i.e., red-black Gauss--Seidel, although it is still not $\epsilon$-robust.
To this end, we now show that overweighting of Schwarz on maximally overlapped $2 \times 2$ subdomains does not improve its smoothing properties for mild anisotropies.\footnote{The conclusions we draw could differ slightly if red-black ordering was used instead of lexicographic ordering, because this ordering seems to have be an important factor in the improvements produced in \cite{Yavneh-1996}.}
A weighted multiplicative Schwarz update takes the same form as in \eqref{eq:SCH-update}, just with the scaling $A_{ij} \gets A_{ij}/w$ for weight $w > 0$.
Evidently, for fixed $w$, the associated Fourier symbol on maximally overlapped $2 \times 2$ subdomains may be calculated analogously to that described in \cref{sec:LFA:2x2}.
From \cref{def:mu} we obtain a lower bound on the smoothing factor associated with this weighted update by evaluating the corresponding symbol at any high-frequency pair $(\omega_1, \omega_2) \in [-\pi/2, 3\pi/2)^2 \setminus [-\pi/2, \pi/2)^2$.
Results are shown in \cref{fig:LFA-overweight} for this symbol evaluated at $(\omega_1, \omega_2) = (0, 3\pi/2)$, revealing that its minimum for almost all $\epsilon$ is essentially no different from the $w = 1$ case.  
In other words, overweighting cannot mitigate the poor smoothing properties and lack of $\epsilon$-robustness of multiplicative Schwarz on maximally overlapped $2 \times 2$ subdomains.
%

%
\subsection{Rotated anisotropic diffusion results on $2 \times 2$ subdomains}
\label{sec:LFA:2x2-rot}

Having analyzed overlapping Schwarz on maximally overlapped $1 \times 1$ and $2 \times 2$ subdomains for the grid-aligned problem \eqref{eq:aligned}, we now study them for the more general rotated problem \eqref{eq:rot}.
In \cref{fig:LFA-rot} we show smoothing factors (\cref{def:mu}) and two-grid convergence factors (\cref{def:E-tg-def}) for both FD and FE discretizations as a function of anisotropy ratio $\epsilon$ and  rotation angle $\theta$.
These contours are created by discretizing $\theta \in [0, \pi/2]$ and $\log_{10} \epsilon \in [-4, 0]$ using 17 equispaced points.

\begin{figure}[t!]
    \centering
    \includegraphics[width=0.345\textwidth]{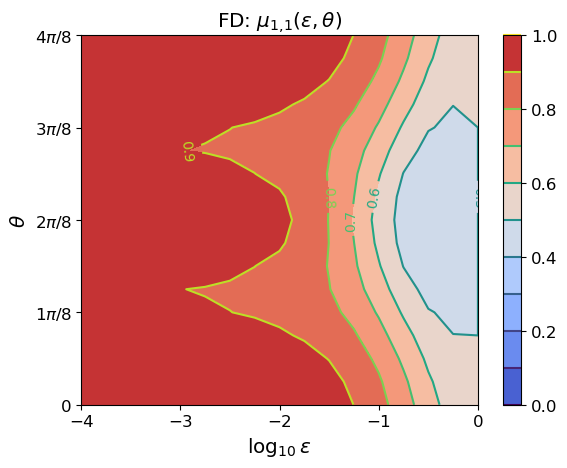}
    \includegraphics[width=0.345\textwidth]{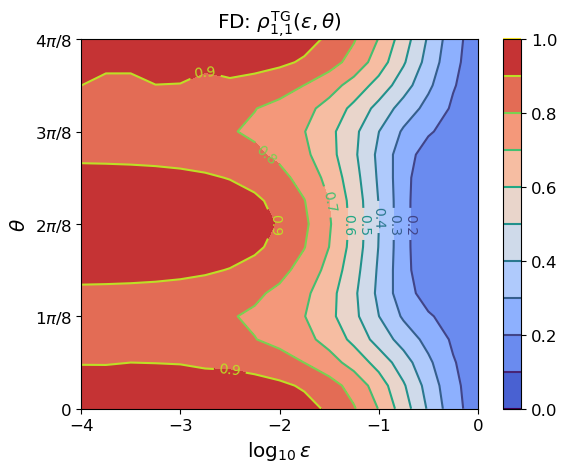}
    \includegraphics[width=0.345\textwidth]{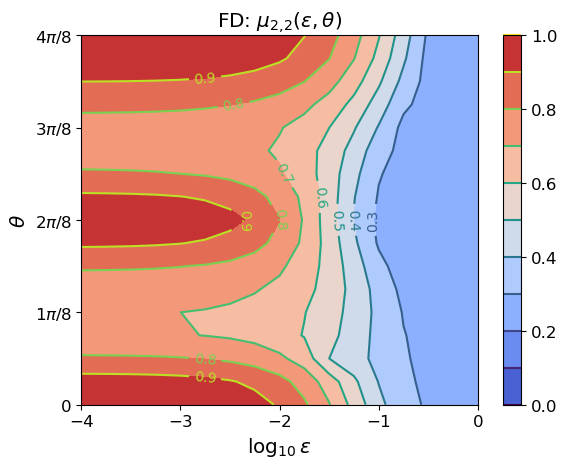}
    \includegraphics[width=0.345\textwidth]{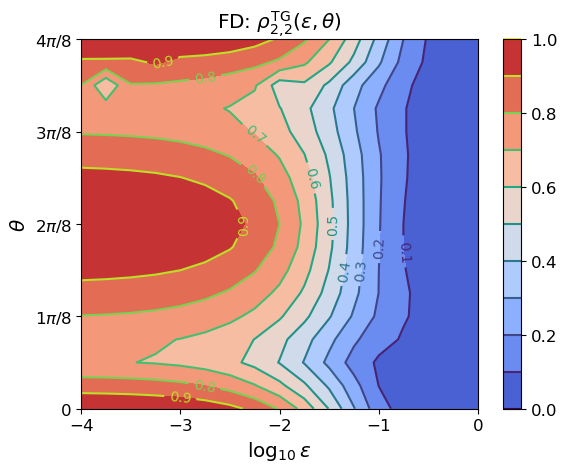}
    \includegraphics[width=0.345\textwidth]{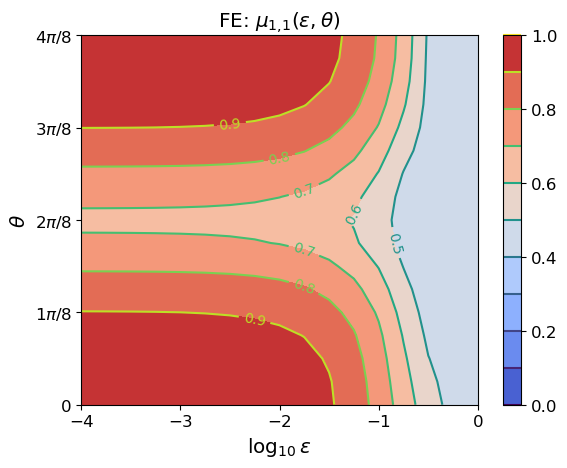}
    \includegraphics[width=0.345\textwidth]{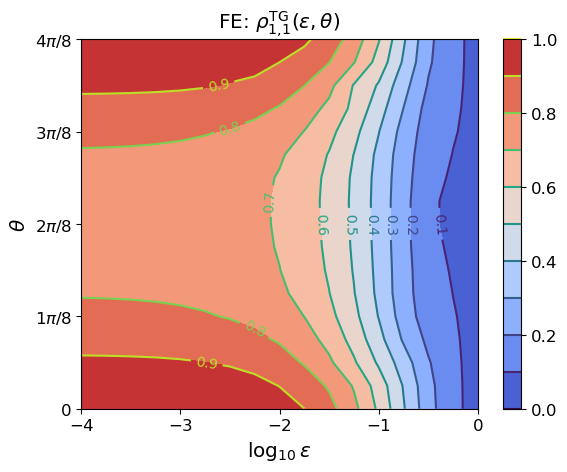}
    \includegraphics[width=0.345\textwidth]{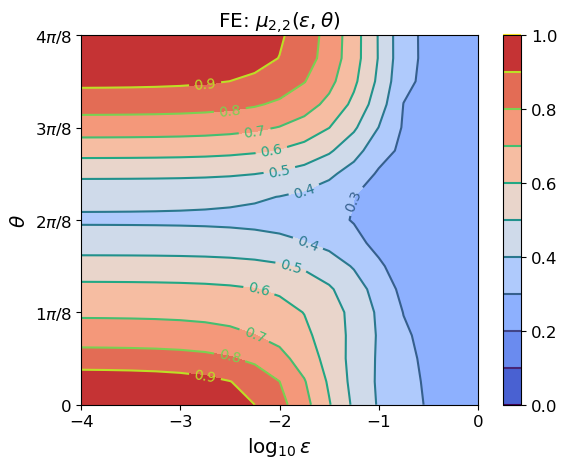}
    \includegraphics[width=0.345\textwidth]{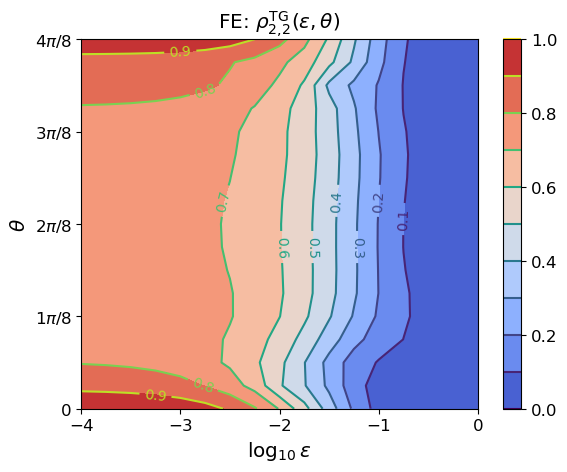}
    \caption{
        Maximally overlapping Schwarz smoothing on $1 \times 1$ and $2 \times 2$ subdomains.
        Smoothing factors $\mu$ and two-grid convergence factors $\rho^{\rm TG}$ for FD \eqref{eq:rot-fd} and FE \eqref{eq:rot-fe} discretizations of the rotated anisotropic diffusion equation \eqref{eq:rot} as a function of anisotropy ratio $\epsilon$ and rotation angle $\theta$.
        \label{fig:LFA-rot}
    }
\end{figure}

Considering the smoothing factors (left column), observe for each discretization that the $2 \times 2$ smoother is stronger than the $1 \times 1$ smoother for all $\epsilon, \theta$.
However, the $2 \times 2$ smoother does not ``magically'' fix the $1 \times 1$ smoother in regions where it has poor smoothing.\footnote{For the $1 \times 1$ smoothing factor in the top left plot, this is not obvious since there are no contour levels between 0.9 and 1. However, closer inspection reveals that indeed the smoothing factor is bounded away from unity for $\theta$ far enough away from $\theta \approx 0$, $\theta \approx \pi/4$, and $\theta \approx \pi/2$.}
For example, for both FD and FE discretizations, the $1 \times 1$ smoother is not $\epsilon$-robust around $\theta \approx 0$ and $\theta \approx \pi/2$, and the same is true for the $2 \times 2$ smoother, even though it produces adequate smoothing over a larger range of $\epsilon$.
Another example is along the $\theta \approx 2 \pi/8 = \pi/4$ axis for the FD discretization, where both the $1 \times 1$ and $2 \times 2$ smoothers fail. 
It is unsurprising that both smoothers fail in this region, however, because even alternating line smoothing is not $\epsilon$-robust for this discretization at $\theta \approx \pi/4$ \cite[Table 7.7.1]{Wesseling-1992}.
Somewhat surprisingly, both smoothers are $\epsilon$-robust for the FE discretization at $\theta \approx \pi/4$.
Overall, the $2\times 2$ smoothers perform acceptably for $\theta$ sufficiently far from grid alignment, and also away from diagonal alignment for the FD discretization.   

The two-grid convergence factors in the right column of \cref{fig:LFA-rot} tell largely the same story as the corresponding smoothing factors in the left column. 
To objectively compare two-grid factors based on $1\times1$ and $2 \times 2$ smoothers we must consider the ratio $\log_{10}( \rho^{\rm TG}_{2,2} )/\log_{10}( \rho^{\rm TG}_{1,1} )$, which tells us how many iterations of the solver using the $1 \times 1$ smoother are required to get the same residual reduction as one iteration of the $2 \times 2$ solver.
We do not show contours of this function, but note that, for the $\epsilon$, $\theta$-domains in \cref{fig:LFA-rot}, it seems bounded between $\approx (1.1, 6)$ and $\approx (1.1, 6.4)$ for the FD and FE discretizations, respectively.
For both discretizations, this ratio  achieves its maximum around $\theta \approx 0$, notably where neither smoother is $\epsilon$-robust.
Given the increased cost of the $2 \times 2$ smoother, its faster two-grid convergence speed likely does not justify its use over the $1 \times 1$ smoother; see related discussion in \cite{MacLachlan-Oosterlee-2011} for the isotropic case.

While two-grid convergence factors do tell largely the same story as the smoothing factors, they are not in total agreement, as can be seen from bottom row of plots where smoothing is relatively good around $\theta \approx \pi/4$ yet two-grid convergence is relatively poor.
This poor convergence is purely due to the inaccuracy of the coarse-grid correction on smooth characteristic components, as discussed in \cref{sec:mg}, and can be rectified by using a higher-order interpolation in the construction of the Galerkin coarse-grid operator \cite{Yavneh_1998}.

Finally, it is interesting to note that neither the smoothing factor nor two-grid convergence factor is symmetric about $\theta = \pi/4$, due to the ordering of the smoother being not symmetric, recalling that subdomains are updated by sweeping west to east and then south to north. 
%

%
\subsection{LFA for $\ell \times 1$ subdomains}
\label{sec:LFA:ellx1}

In this section we outline the LFA procedure for maximally overlapping multiplicative Schwarz with subdomains of size $\ell \times 1$, $\ell \geq 1$ applied to the 9-point stencil \eqref{eq:9-point}.
Motivations for considering this smoother will be discussed in the next section.
We note that a closely related maximally overlapping smoother was considered in \cite{Riva-etal-2019} albeit for a genuinely one-dimensional problem. 
To apply LFA to this smoother we must of course verify that the associated error propagator is invariant with respect to the Fourier modes \eqref{eq:modes}.
For the moment, let us assume this is the case and then verify this assumption shortly. 

\begin{figure}[b!]
    \centering
    \includegraphics[scale=0.85]{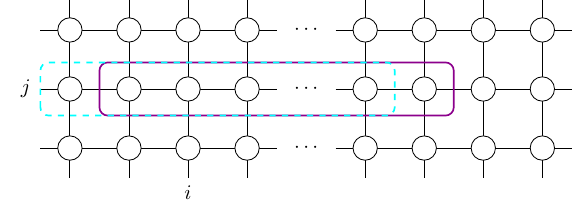}
    \hspace{1ex}
    \includegraphics[scale=0.85]{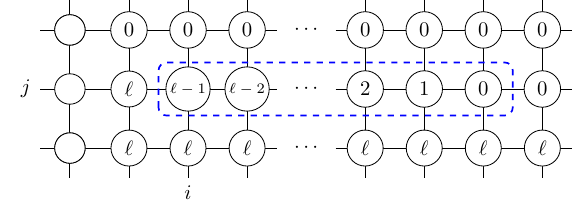}
    \caption{Diagram for maximally overlapping Schwarz on $\ell \times 1$ subdomains.
    \textbf{Left:} The final two $\ell \times 1$ subdomains that have been updated immediately prior to updating subdomain $S_{ij}$ and who share DOFs with those in $S_{ij}$.
    \textbf{Right:} Number of times DOFs in the 9-point stencil of DOFs in subdomain $S_{ij}$ (dashed blue rectangle) have been updated immediately prior to updating $S_{ij}$.
    \label{fig:subdomain-ellx1}
    }
\end{figure}

The derivation of the symbol proceeds analogously to that described in \cref{sec:LFA:2x2} for $2 \times 2$ subdomains.
Suppose that the subdomain $S_{ij}$ consists of the $\ell$ grid points $(\bm{x}_{i,j}, \bm{x}_{i+1,j}, \ldots, \   \bm{x}_{i+\ell-1,j})$, with associated DOFs in the same ordering (see \cref{fig:subdomain-ellx1}).
Immediately prior to updating all DOFs in $S_{ij}$, each of the aforementioned DOFs in $S_{ij}$ have been updated $\ell-1$, $\ell - 2$, $\ldots$, and zero times, respectively (see \cref{fig:subdomain-ellx1}).
As such, let us introduce the ansatz that the error associated vectors in \eqref{eq:SCH-update} assume the form
\begin{align} \label{eq:SCH-ellx1-Eold-Enew}
{\cal E}^{\rm old}_{ij}
=
e^{\i \bm{\omega} \cdot \bm{x}_{ij} / h}
D
\begin{bmatrix}
    \alpha_{\ell-1} \\
    \alpha_{\ell-2} \\
    \vdots \\
    \alpha_0 
\end{bmatrix},
\quad
{\cal E}^{\rm new}_{ij}
=
e^{\i \bm{\omega} \cdot \bm{x}_{ij} / h}
D
\begin{bmatrix}
    \alpha_{\ell} \\
    \alpha_{\ell-1} \\
    \vdots \\
    \alpha_1 \\
\end{bmatrix}.
\end{align}
Here the coefficients $\alpha_p \in \mathbb{C}$, $p = 1,\ldots,\ell$ are unknowns, and are a function of the known coefficient $\alpha_0$, describing the initial error,
and $D$ is again a relative Fourier matrix relating the position of each point in $S_{ij}$ to that of $\bm{x}_{ij}$.
Let us introduce the vector $\bm{\alpha} = ( \alpha_{\ell}, \ldots, \alpha_1 )^\top \in \mathbb{C}^{\ell}$ holding the unknown coefficients.
Considering the one-dimensional nature of the subdomains, both $D$ and the principle submatrix matrix $A_{ij}$ in \eqref{eq:SCH-update} have the following relatively simple structure:
\begin{align} \label{eq:SCH-1d-Aij-D}
    D = \diag
    \left( (e^{\i \omega_1})^0, (e^{\i \omega_1})^1, \ldots, (e^{\i \omega_1})^{\ell - 1}   
    \right)
    \in \mathbb{C}^{\ell \times \ell},
    \quad
    A_{ij}
    =
    \begin{bmatrix}
        \scc & \se \\
        \sw & \scc & \se \\
        & \ddots & \ddots & \ddots \\
        & & \sw & \scc
    \end{bmatrix}
    \in \mathbb{R}^{\ell \times \ell}.
\end{align}
As for the $2 \times 2$ case in \cref{sec:LFA:2x2}, the residual on the right-hand side of \eqref{eq:SCH-update} is complicated to express because each component involves applying the 9-point stencil \eqref{eq:9-point} centered at the corresponding location, and accounting for the fact that different DOFs in the stencil have been updated a different number of times (see \cref{fig:subdomain-ellx1}).
As such, using the notation from \eqref{eq:9-point-wc} and proceeding analogously to \eqref{eq:r-derivation}, the residual on the right-hand side of \eqref{eq:SCH-update} for general $\ell \in \mathbb{N}$ may be written
\begin{subequations}
\begin{align} 
    \nonumber
    & (A{\cal E}^{\textrm{old}})_{ij}
    =
    e^{\i \bm{\omega} \cdot \bm{x}_{ij} / h}
    D
    \left(
    \begin{bmatrix}
        (\cssw + \css + \csse) \alpha_{\ell} + \csw \alpha_{\ell} + {\cscc} \alpha_{\ell-1} + {\cse} \alpha_{\ell-2} \\ 
        (\cssw + \css + \csse) \alpha_{\ell} + {\csw} \alpha_{\ell-1} + {\cscc} \alpha_{\ell-2} + {\cse} \alpha_{\ell-3} \\
        \vdots \\
        (\cssw + \css + \csse) \alpha_{\ell} + {\csw} \alpha_{3} + {\cscc} \alpha_{2} + {\cse} \alpha_{1} \\
        (\cssw + \css + \csse) \alpha_{\ell} + {\csw} \alpha_{2} + {\cscc} \alpha_{1} \\
        (\cssw + \css + \csse) \alpha_{\ell} + {\csw} \alpha_{1}
    \end{bmatrix}
    +
    \begin{bmatrix}
        \csnw + \csn + \csne \\
        \csnw + \csn + \csne \\
        \vdots \\
        \csnw + \csn + \csne \\
        \csnw + \csn + \csne + \cse \\
        \csnw + \csn + \csne + \cscc + \cse \\
    \end{bmatrix}
    \alpha_0
    \right)
    \\
    \label{eq:SCH-1d-Rold-temp}
    &=
    e^{\i \bm{\omega} \cdot \bm{x}_{ij} / h}
    D
    \left[
    \left(
    \begin{bmatrix}
        \csw & {\cscc} & {\cse} \\
        & {\csw} & {\cscc} & {\cse} \\
        & & \ddots & \ddots & \ddots \\
        & & & {\csw} & {\cscc} & {\cse} \\
        & & & & {\csw} & {\cscc} \\
        & & & & & {\csw}
    \end{bmatrix}
    +
    \left( \cssw + \css + \csse \right) \bm{1} \bm{e}_1^\top
    \right) 
    \begin{bmatrix}
    \alpha_{\ell} \\
    \alpha_{\ell-1} \\
    \vdots \\
    \alpha_3 \\
    \alpha_2 \\
    \alpha_1 \\
    \end{bmatrix}
    +
    \left[  
    (\csnw + \csn + \csne) \bm{1} 
    + \cse \bm{e}_{\ell-1} 
    + (\cscc + \cse) \bm{e}_{\ell} 
    \right]
    \alpha_0
    \right]
    \\
    \label{eq:SCH-1d-Rold}
    &\equiv
    e^{\i \bm{\omega} \cdot \bm{x}_{ij} / h}D (C \bm{\alpha} + \bm{c} \alpha_0).
\end{align}
\end{subequations}
Here $\bm{e}_j \in \mathbb{R}^{\ell}$ denotes the canonical basis vector in direction $j$, and $\bm{1} \in \mathbb{R}^{\ell}$ the vector of ones.

Plugging \eqref{eq:SCH-1d-Aij-D} and \eqref{eq:SCH-1d-Rold} into \eqref{eq:SCH-update}, cancelling the common factor of $e^{\i \bm{\omega} \cdot \bm{x}_{ij} / h} \neq 0$ and rearranging we arrive at the linear system\footnote{While this formula holds for general $\ell \in \mathbb{N}$, its interpretation for $\ell = 1$ and $\ell = 2$ is somewhat ambiguous. 
For $\ell = 1$ we have: $A_{ij} = \scc$ and $U - I = -1$. 
Moreover, the residual $(A{\cal E}^{\textrm{old}})_{ij}$ in \eqref{eq:SCH-1d-Rold-temp} should be read from the bottom rather than the top, so that the $\ell = 2$ residual corresponds to the last two rows, and the $\ell = 1$ residual to the last row, where one should interpret $\bm{e}_{\ell - 1} = \bm{e}_{0} = 0$.}
\begin{align} \label{eq:SCH-1d-LFA-sys}
    \underbrace{[A_{ij} D (U - I) - D C]}_{{\cal A}}
    \bm{\alpha}
    =
    \underbrace{[D \bm{c} - A_{ij} D \bm{e}_{\ell}]}_{\bm{b}} \alpha_0,
    \quad
    U - I = 
    \begin{bmatrix}
        -1 & 1 \\
        & \ddots & \ddots \\
         & & -1
    \end{bmatrix}
    \in \mathbb{R}^{\ell \times \ell}.
\end{align}
While the same notation is used to describe the linear system governing the symbol in the case of $2 \times 2$ subdomains \eqref{eq:SCH-2x2-system}, there should be no confusion that the quantities used here take on different values than those in \eqref{eq:SCH-2x2-system}.
This system can be solved as
$\bm{\alpha}
=
{\cal A}^{-1} \bm{b}
{\alpha_0}$,
where each component of ${\cal A}^{-1} \bm{b}$ is the amplification factor for the $\ell$, $\ell-1$, $\ldots$, once updated DOFs, respectively.
After one complete smoothing iteration, all DOFs will have been updated $\ell$ times, so that ultimately we care about the first component of this solution vector, and hence the relevant symbol is
\begin{align} \label{eq:SCH-ellx1-symbol-def}
    \wt{s}(\omega_1, \omega_2) = ({\cal A}^{-1} \bm{b})_1.
\end{align}

Finally, we now verify that the error propagator is indeed invariant with respect to the Fourier modes.
\begin{lemma}[Fourier mode invariance]
    The error propagator for multiplicative Schwarz on maximally overlapping $\ell \times 1$ subdomains applied to the 9-point stencil \eqref{eq:9-point} is invariant with respect to the Fourier modes \eqref{eq:modes}.
\end{lemma}

\begin{proof}
    This follows from a direct application of the arguments used to prove \cite[Theorem 1]{MacLachlan-Oosterlee-2011}, which showed that multiplicative Schwarz on maximally overlapped $2 \times 2$ subdomains is invariant with respect to the Fourier modes.
    Namely, having introducing the ansatz that the error consists of a single Fourier mode, we find that a scalar Fourier symbol exists and is given by the first solution component of the consistent linear system \eqref{eq:SCH-1d-LFA-sys}.
\end{proof}
%

%
\subsection{Grid-aligned anisotropic diffusion results for $\ell \times 1$ subdomains}
\label{sec:LFA:ellx1-theory}

Here we consider application of the LFA procedure discussed in the previous section to grid-aligned anisotropic diffusion.
For grid-aligned anisotropic problems, line smoothing with lines grouping strongly coupled DOFs is generally known to be a good smoother.
In fact, when $\epsilon \in (0, 1]$, it is well-known that the $x$-line smoothing factor for the FD discretization \eqref{eq:rot-fd} is $\mu_{\infty, 1}(\epsilon) = \tfrac{1}{\sqrt{5}} \approx 0.447$. See, e.g., \cite[p. 340]{Brandt_1977}, \cite[pp. 125--126]{Wesseling-1992}, \cite[Sec. 5.1.3]{Trottenberg_etal_2001}.\footnote{These references give the smoothing factor for $y$-line smoothing on $-\epsilon u_{xx} - u_{yy} = 0$. By symmetry the same result applies to $x$-line smoothing on $-u_{xx} - \epsilon u_{yy} = 0$.}
We are not aware of literature on an analogous result the FE discretization \eqref{eq:rot-fe}, but we note that it can be shown that the exact same smoothing factor result applies (to save space we do not provide a proof of this result). 

A natural question is then as follows: Is global smoothing in the form of line smoothing actually required to achieve an $\epsilon$-robust smoother in the sense of \cref{def:robust}? 
Or, is it possible to obtain $\epsilon$-independent smoothing by using overlapping multiplicative Schwarz with sufficiently large subdomains and  sufficiently large overlap? We study these questions in this section.
Of course, one could study this question by analyzing square subdomains of size $\ell \times \ell$; however, our numerical results show (see \cref{sec:num-res}) that $\ell \times \ell$ subdomains offer essentially no advantage over $\ell \times 1$ subdomains, at least for sufficiently small $\epsilon$, presumably because the additional connections they add are in the weak direction.

In a certain sense, Schwarz on maximally overlapped $\ell \times 1$ subdomains approximates $x$-line smoothing, since in the limit that $\ell$ approaches the length of the computational domain these two methods coincide.\footnote{On a finite-sized domain this is true, although in the LFA context of an infinite, periodic domain it is not strictly true. Regardless, we abuse notation and denote $x$-line smoothing as corresponding to $\ell = \infty$.}
Another way this smoother approximates $x$-line smoothing is in a domain decomposition sense. That is, if rather than directly inverting the full $(n-1)$-dimensional $x$-line, it were approximately solved by partitioning it into $n-1$ maximally overlapping subdomains of length $\ell$ and then each subdomain problem solved multiplicativively. 
Of course, the maximally overlapping smoother we analyze here is prohibitively expensive for large $\ell$ and is not recommended in practice, since every DOF in the domain is updated $\ell$ times per smoothing sweep. 
Two more practical alternatives are: 1. a smoother with $\ell \times 1$ subdomains but with much less overlap, and 2. an additive Schwarz smoother using $\ell \times 1$ subdomains but with much less overlap, since this would be more readily parallelized.
However, both of these alternatives are weaker than the ``strongest-case'' smoother we analyze. So, while we do not rigorously prove anything about convergence of these more practical alternatives, it is natural to expect that their smoothing properties will be worse than the smoother we analyze, which is something we confirm numerically in \cref{sec:num-res}.

With motivations out of the way, we are now ready to present our main theoretical results.
All of the proofs for these results can be found in \cref{app:ellx1}, and the results are organized as follows.
\Cref{thm:1d-symbol-FD,thm:1d-symbol-FE} present the symbol for the smoother applied to the FD and FE discretizations, respectively.
\Cref{thm:1d-smooth} presents the associated smoothing factors, and then \cref{cor:robust-1d} discusses $\epsilon$-robustness of the smoother.
We remark that all of these results are expressed in terms of asymptotic expansions in $\epsilon$, so that they can be expected to have a relatively large error when $\epsilon$ is not sufficiently small.

\begin{theorem}[Linearized symbol: FD] \label{thm:1d-symbol-FD}
    Consider the anisotropic diffusion equation \eqref{eq:aligned} with anisotropy ratio $\epsilon \in [0,1]$ discretized with FDs \eqref{eq:rot-fd}.
    Let $\wt{s}_{\ell,1}$ be the symbol for maximally overlapping multiplicative Schwarz with $\ell \times 1$ subdomains, $\ell \in \mathbb{N}$ (see \cref{sec:LFA:ellx1}).
    Then,
    \begin{align}
        \wt{s}_{\ell,1}(\omega_1, \omega_2) = \wt{s}_0(\omega_1) + \wt{s}_{1} (\omega_1, \omega_2) \epsilon + {\cal O}(\epsilon^2),
    \end{align}
    where 
    \begin{subequations}
    \begin{align}
        \wt{s}_0(\omega_1) 
            &= 
            \frac{a^{\ell}}{1 + \ell - \ell \bar{a}}, 
            \\
        \quad
        \wt{s}_1(\omega_1, \omega_2)
            &=
        -
        a^{-\ell} 
        \wt{s}_0
        \left[
            \frac{1}{3} \ell (\ell + 1) [3 \bar{a} + (1 - \bar{a}) (\ell + 2)] \wt{s}_0
            +
            (e^{\i \omega_2} + \wt{s}_0 e^{-\i \omega_2}) z_1(a)
        \right],
    \end{align}
    \end{subequations}
    with $a = e^{\i \omega_1}$ and $z_1(a)$ defined in \cref{lem:Bx=b-solutions}.
\end{theorem}

\begin{theorem}[Linearized symbol: FE] \label{thm:1d-symbol-FE}
    Consider the anisotropic diffusion equation \eqref{eq:aligned} with anisotropy ratio $\epsilon \in [0,1]$ discretized with FEs \eqref{eq:rot-fe}.
    Let $\wt{s}_{\ell,1}$ be the symbol for maximally overlapping multiplicative Schwarz with $\ell \times 1$ subdomains, $\ell \in \mathbb{N}$ (see \cref{sec:LFA:ellx1}).
    Then,
    \begin{align}
        \wt{s}_{\ell,1}(\omega_1, \omega_2) = \wt{s}_0(\omega_1, \omega_2) + \wt{s}_{1}(\omega_1, \omega_2) \epsilon  + {\cal O}(\epsilon^2),
    \end{align}
    where 
    \begin{align}
         \wt{s}_0(\omega_1, \omega_2) = \frac{a^{\ell} - \bar{\delta}_0  z_1(a)}{1 + \ell - \ell \bar{a} + \delta_0 z_1(a)}, 
    \end{align}
    with $a = e^{\i \omega_1}$, $\delta_0 = \frac{1}{2} e^{-\i \omega_2} (\cos \omega_1 - 1)$, $z_1(a)$ defined in \cref{lem:B0-solves},
    and
    $\wt{s}_1(\omega_1, \omega_2)$ a function of $\omega_1$ and $\omega_2$ satisfying $\wt{s}_{1}(0, \omega_2) = - \frac{3}{2} \ell (\ell + 1)(1 - e^{-i \omega_2})$.
\end{theorem}

\begin{theorem}[Linearized smoothing factor] \label{thm:1d-smooth}
    Consider the anisotropic diffusion equation \eqref{eq:aligned} with anisotropy ratio $\epsilon \in [0,1]$ discretized with either FDs \eqref{eq:rot-fd} or FEs \eqref{eq:rot-fe}.
    Let $\mu_{\ell,1}(\epsilon)$ be the smoothing factor (see \cref{def:mu}) for maximally overlapping multiplicative Schwarz with $\ell \times 1$ subdomains, $\ell \in \mathbb{N}$ (see \cref{sec:LFA:ellx1}).
    Further, suppose that the symbols in \cref{thm:1d-symbol-FD,thm:1d-symbol-FE} are sufficiently smooth functions of $\omega_1$ and $\omega_2$.\footnote{We have no reason to believe that these functions do not meet the smoothness conditions, but we do not explicitly verify due the complexity or non-analytic form of the functions.}
    Then, the smoothing factor for the FD discretization is
    \begin{align} \label{eq:1d-smooth-FD}
        \mu_{\ell, 1}(\epsilon) 
        = 
        1 - \ell (\ell + 1) \epsilon
    + {\cal O}(\epsilon^2),
    \end{align}
    and the smoothing factor for the FE discretization is bounded from below as\footnote{We believe that this lower bound in fact holds with equality; that is, $\mu_{\ell, 1}(\epsilon)
        =
        1 - \frac{3}{2} \ell (\ell + 1) \epsilon
    + {\cal O}(\epsilon^2)$. However we have not been able to prove that this is the case. See discussion in \cref{rem:smooth-FE-equality}.}
    \begin{align} \label{eq:1d-smooth-FE}
    \mu_{\ell, 1}(\epsilon)
    \geq
        1 - \frac{3}{2} \ell (\ell + 1) \epsilon
    + {\cal O}(\epsilon^2).
    \end{align} 
\end{theorem}

\begin{corollary}[$\epsilon$-robustness] \label{cor:robust-1d}
    Consider the class of multiplicative Schwarz smoothers with maximally overlapping subdomains of size $\ell \times 1$, $\ell \in \mathbb{N}$, and the notion of $\epsilon$-robustness in \cref{def:robust}.
    Then, the following statements hold for the FD \eqref{eq:rot-fd} and FE \eqref{eq:rot-fe} discretizations of grid-aligned anisotropic diffusion \eqref{eq:aligned}:

    \begin{enumerate}
        \item For any fixed $\ell \in \mathbb{N}$ the smoother is not $\epsilon$-robust, with $\lim_{\epsilon \to 0^+} \mu_{\ell,1} = 1$ in the FD case, and $\lim_{\epsilon \to 0^+} \mu_{\ell,1} \geq 1$ in the FE case.
        \item A necessary condition for $\epsilon$-robustness of this class is that $\ell$ increases at least as fast as ${\cal O}(\epsilon^{-1/2})$ as $\epsilon \to 0^+$.

        \item Consider the FD case only. Let $\mu_* \in (0, 1)$ be a smoothing factor for which the ${\cal O}(\epsilon^2)$ terms in \eqref{eq:1d-smooth-FD-copy} are uniformly small with respect to $\ell$, i.e., $| \mu_* - [1 - \ell(\ell + 1) \epsilon ] | < 1$, for all $\ell \in \mathbb{N}$.
    Then, for the FD discretization this class of smoothers has an $\epsilon$-independent smoothing factor of $\mu_{\ell, 1} = \mu_*$ if the subdomain size satisfies
    \begin{align} \label{eq:ell-star}
        \ell 
        = 
        \ell_*(\epsilon) 
        = 
        \left\lceil
        \sqrt{1 - \mu_*} \epsilon^{-1/2} + {\cal O}(\epsilon^0)
        \right\rceil. 
    \end{align}
    \end{enumerate}
\end{corollary}

Some commentary is now in order.
To help contextualize these results we consider the smoothing factor $\mu_{\ell, 1}$ for the FD case in \cref{fig:LFA-1d-smooth-FD}, omitting analogous results for the FE discretization because they are qualitatively similar.
To make the contour plots in the top row of the figure, we discretize $\log_{10}\epsilon \in [-4, 0]$ with 17 equispaced points, and for each $\epsilon$ we compute the smoothing factor for $\ell \in \{1,2,4,8,12,16,20,24,28,32\}$.
The bottom row of \cref{fig:LFA-1d-smooth-FD} provides supporting numerical verification of the smoothing factor result in \cref{thm:1d-smooth}, since for sufficiently small $\epsilon$ the theoretically predicted smoothing factors perfectly lie over the top of numerically computed smoothing factors.
Considering the bottom left plot, it is interesting to note that for a fixed $\ell \geq 2$, there is a range of $\epsilon$ values for which the smoothing factor $\mu_{\ell,1}(\epsilon)$ is roughly the same as that of $x$-line smoothing.
As we motivated earlier, this smoother can be seen as an approximation to $x$-line smoothing, and this plot clarifies how the smoothing properties of the two methods are related.

The bottom right plot in \cref{fig:LFA-1d-smooth-FD} makes it apparent that, indeed, for any fixed $\ell$, the smoothing factor tends to unity as $\epsilon \to 0^+$, as per claim 1 of \cref{cor:robust-1d}.
The top row of \cref{fig:LFA-1d-smooth-FD} provides supporting numerical evidence for claims 2 and 3 of \cref{cor:robust-1d}.
These plots brings into focus the somewhat surprising and unintuitive outcome that $\epsilon$-robust smoothing requires the subdomain size to grow rapidly as $\epsilon$ decreases. 
In particular, for $\epsilon$ sufficiently small, we do obtain $\epsilon$-independent smoothing factors equal to $\mu_*$ when $\ell = \ell_*(\epsilon)$ is given in \eqref{eq:ell-star} and $\mu_*$ is sufficiently large; e.g., \eqref{eq:ell-star} seems quite accurate for $\mu_* \gtrsim 0.8$.
While \eqref{eq:ell-star} appears inaccurate for smoothing factors $\mu_* \lesssim 0.8$, we nonetheless see that all level curves of $\mu_{\ell,1}$ appear to be of the form $\ell \propto \epsilon^{-1/2}$, indicating the necessary condition for robust smoothing given by claim 2 of \cref{cor:robust-1d} is in fact and a necessary and sufficient condition.\footnote{We cannot prove this is a sufficient condition because it would require greater knowledge of the smoothing factor than provided in \cref{thm:1d-smooth}, i.e., knowledge of $\mu_{\ell, 1}$ where the ${\cal O}(\epsilon^2)$ terms in \cref{thm:1d-smooth} are not small. 
}
Finally, we remark that an obvious implication of the results in \cref{cor:robust-1d} is that the multigrid algorithm described in \cref{sec:mg} cannot be $\epsilon$-robust when using a multiplicative Schwarz smoother on maximally overlapping $\ell \times 1$ subdomains smoother for any fixed $\ell$ because the coarse-grid correction has no mechanism for damping the high-frequency modes that the smoother fails to damp.

\begin{figure}[t!]
    \centering
    \includegraphics[width=0.4\linewidth]{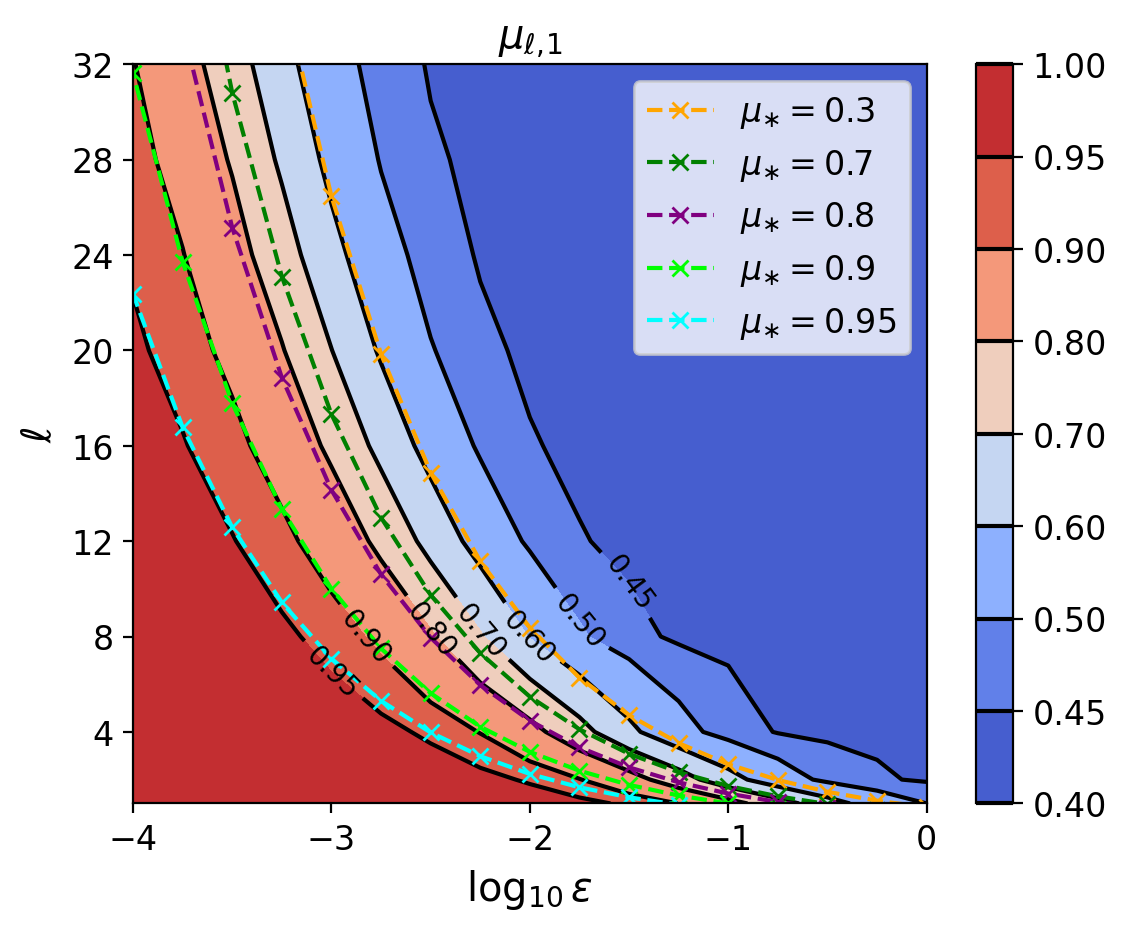}
    \includegraphics[width=0.4\linewidth]{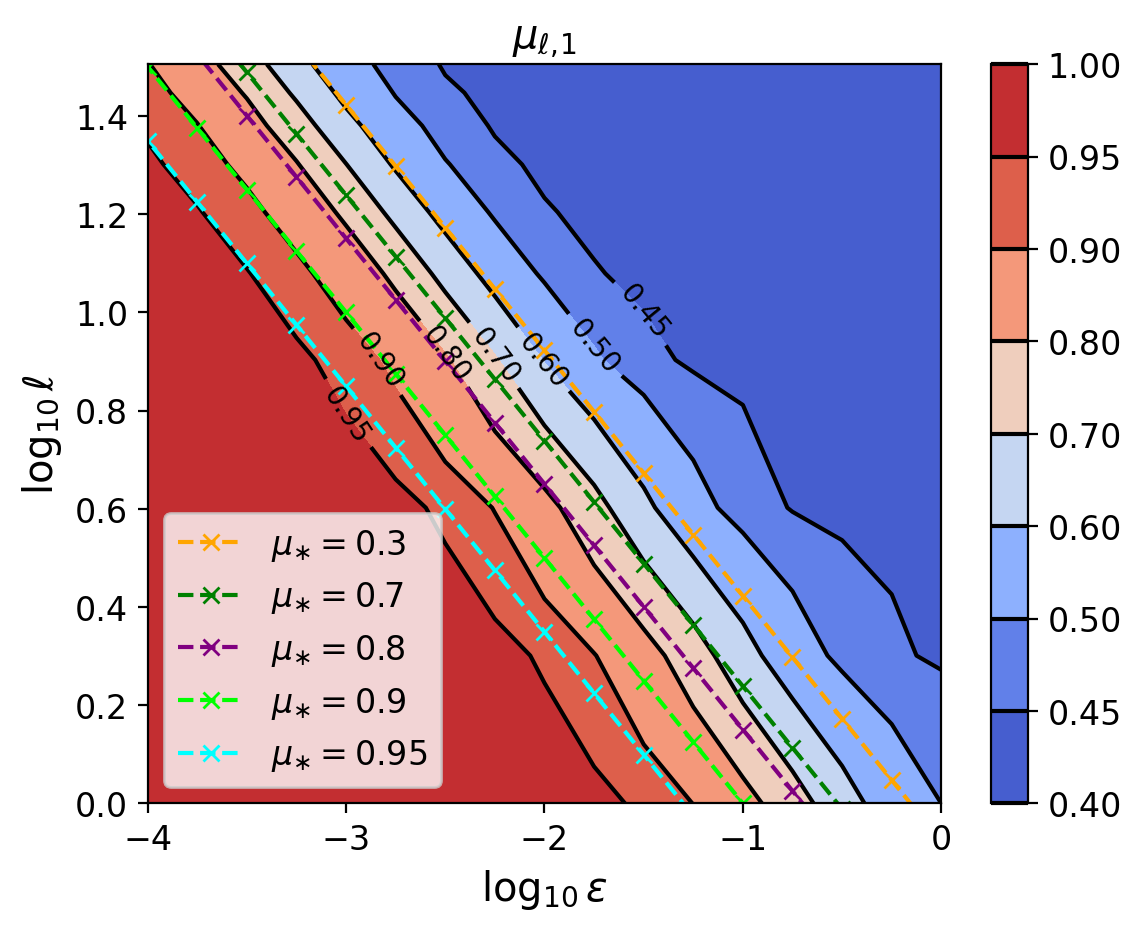}
    \includegraphics[width=0.385\linewidth]{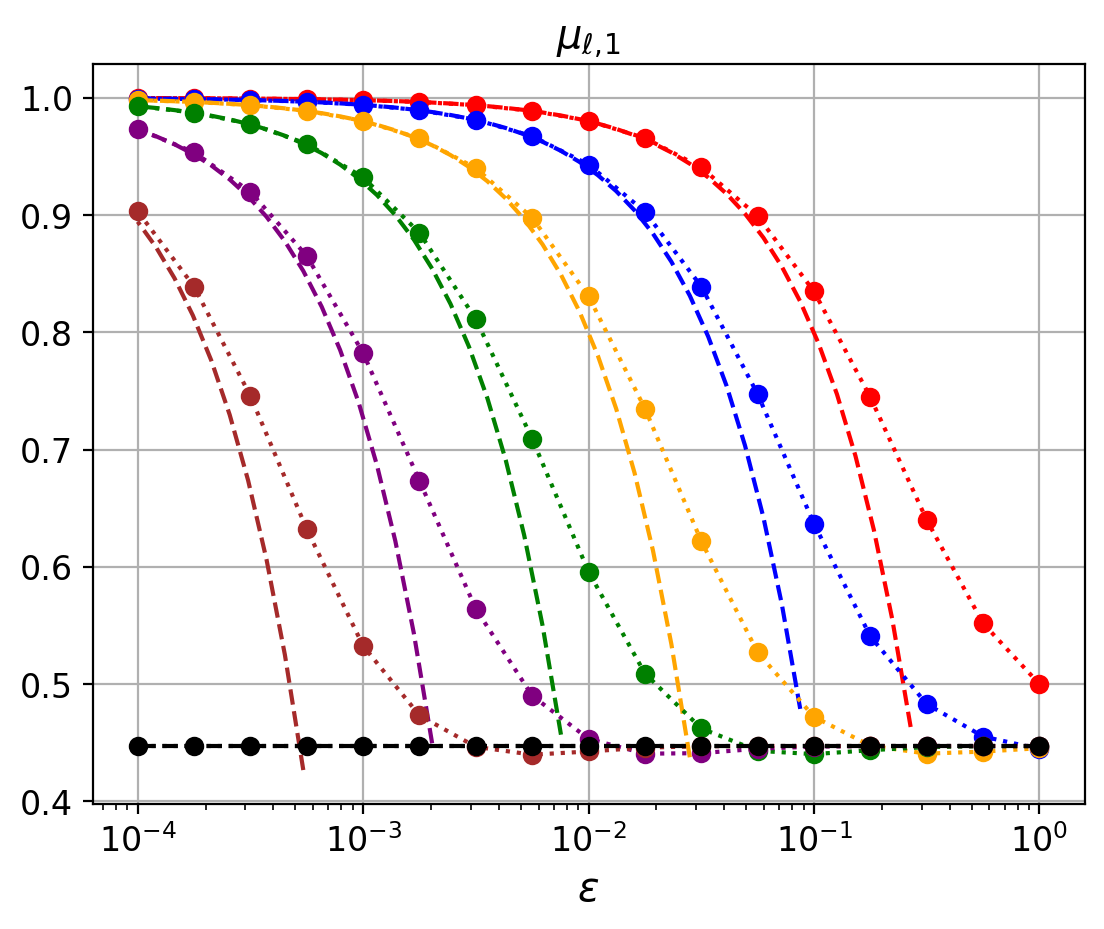}
    \includegraphics[width=0.4\linewidth]{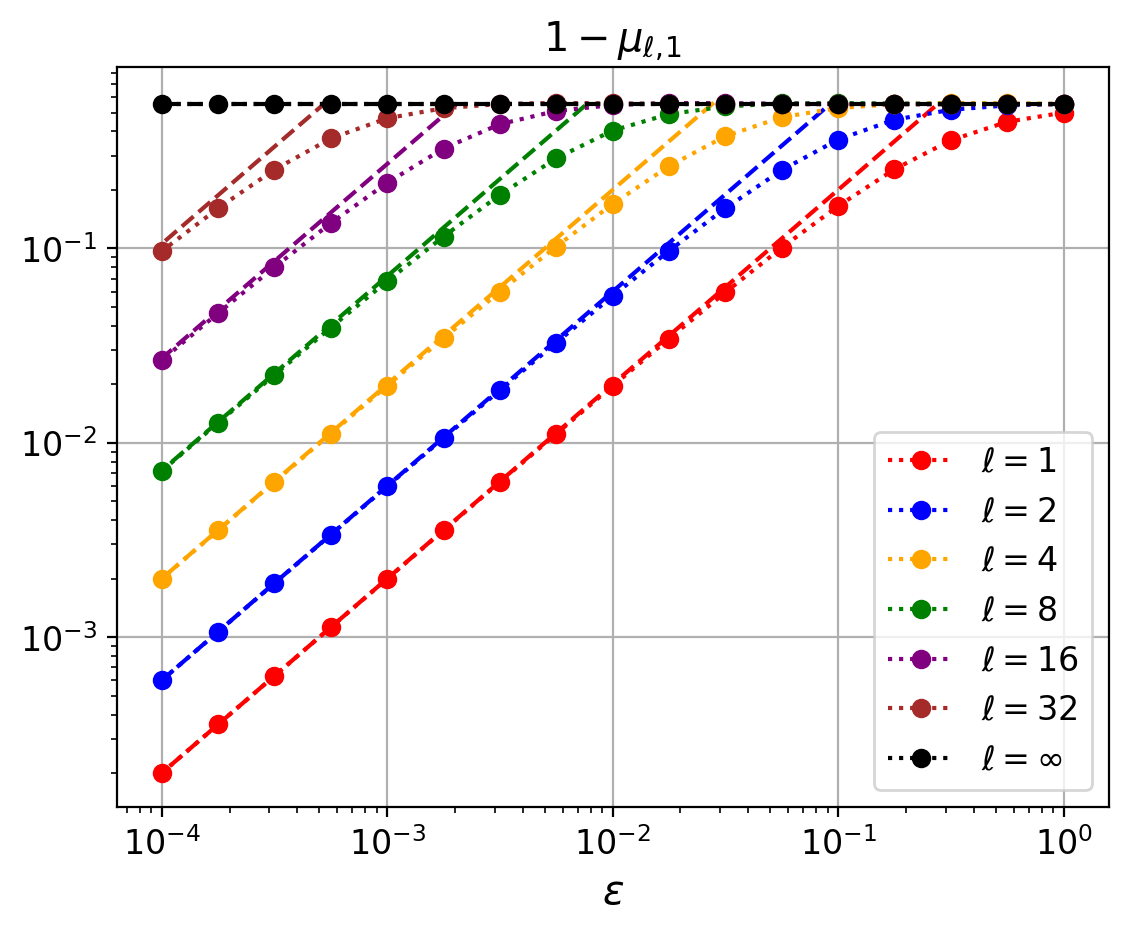}
    \caption{Smoothing factors of multiplicative Schwarz with maximally overlapping $\ell \times 1$ subdomains applied to the FD discretization \eqref{eq:rot-fd} of grid-aligned anisotropic diffusion \eqref{eq:aligned}.
    \textbf{Top:} Contours of $\mu_{\ell,1}(\epsilon)$ are shown in linear $\ell$ space (left) and logged $\ell$ space (right). Curves $\ell = \ell_*(\epsilon) = \sqrt{1 - \mu_*} \epsilon^{-1/2}$ from \cref{cor:robust-1d} are overlaid for values of $\mu_*$ indicated in the legends.
    \textbf{Bottom:} Cross-sections of $\mu_{\ell,1}(\epsilon)$ in linear space and $1 - \mu_{\ell,1}(\epsilon)$ in log space for the values of $\ell$ indicated in the legends. Dotted lines are the numerically computed values, and dashed lines are the linear terms from the \cref{thm:1d-smooth}.
    Black lines are the smoothing factors for $x$-line smoothing of $\mu_{\infty,1} \approx 0.447$.
    \label{fig:LFA-1d-smooth-FD}
    }
\end{figure}

%
\section{Numerical results}
\label{sec:num-res}

In this section we present numerical results illustrating how the LFA predictions from \cref{sec:LFA} extend to the practical setting of Dirichlet boundary conditions, finite-sized domains and multilevel solvers. 
Furthermore, we investigate smoother configurations whose LFA would be substantially more complicated than considered in \cref{sec:LFA}.
In our numerical tests we apply the multigrid algorithm described in \cref{sec:mg} to generate a sequence of iterates $\{ \bm{x}_k \}_{k = 0}$ approximating the solution of the linear system $A_0 \bm{x} = \bm{b}$, corresponding to the discretization of \eqref{eq:rot} on a fine-grid with $(n_0-1)^2$ DOFs.
We set $\bm{b} = \bm{0}$ and take $\bm{x}_0$ to have random entries uniformly drawn from $[0, 1)$.
Numerically reported convergence factors are the ratio of residual norms for the two final residuals produced by the solver, which is iterated until: 1. the absolute residual norm $\Vert \bm{r}_k \Vert = \Vert A_0 \bm{x}_k \Vert$ falls below $10^{-30}$, or 2. the number of iterations reaches 100.
Wherever $\epsilon$ is discretized in our tests, it is done so by discretizing $\log_{10} \epsilon \in [-4, 0]$ using 17 equispaced points.
For most of our tests we consider only the FD discretization \eqref{eq:rot-fd}, omitting results for the FE discretization \eqref{eq:rot-fe} since they are qualitatively similar in most cases.

\begin{figure}[t!]
    \centering
    \includegraphics[width=0.4\linewidth]{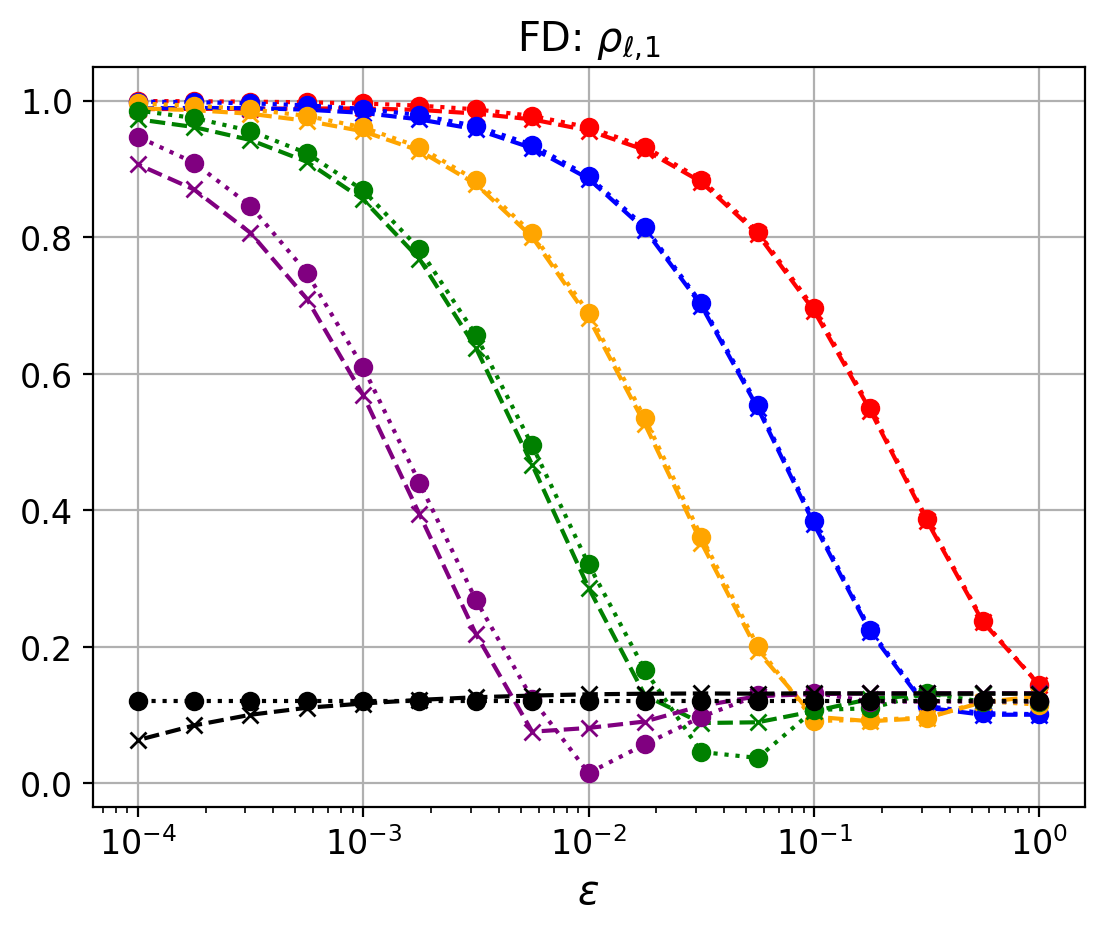}
    \includegraphics[width=0.4\linewidth]{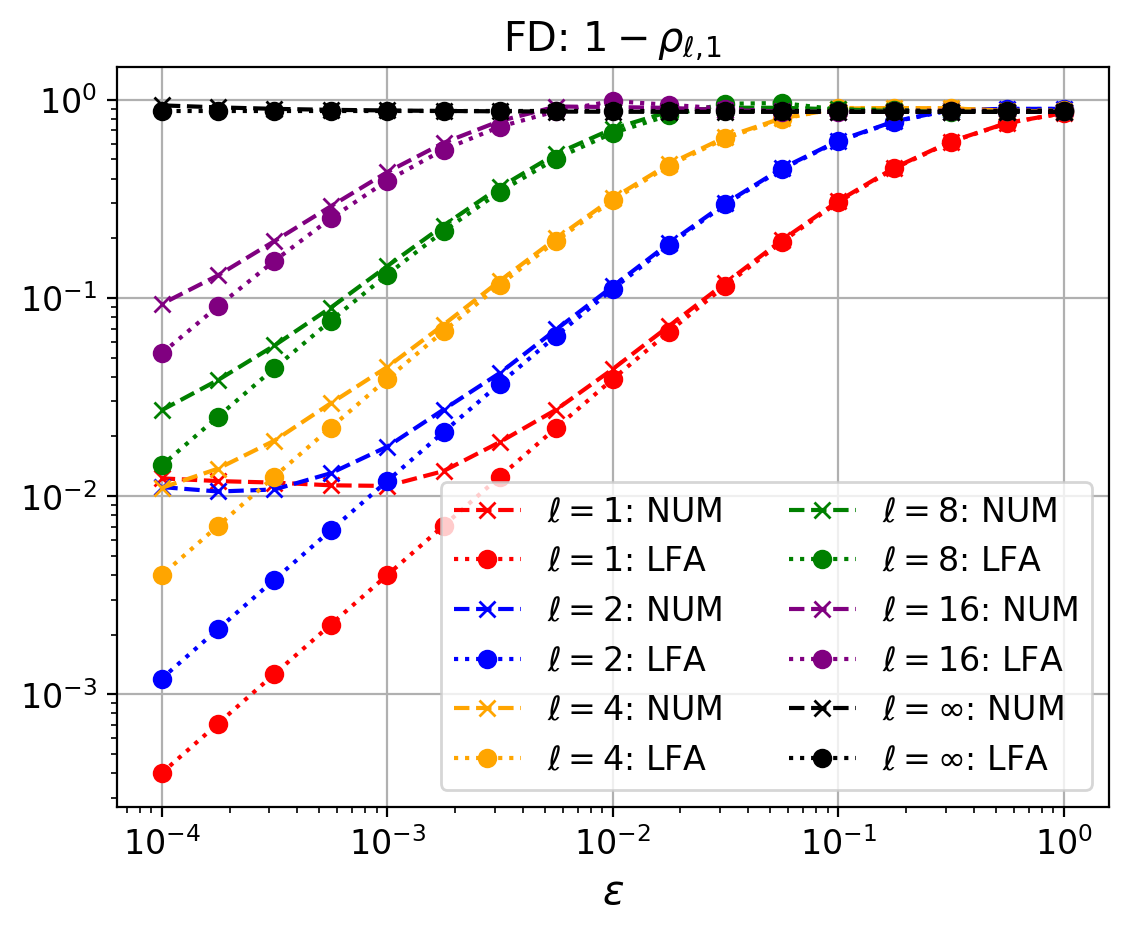}
    \caption{LFA vs. numerically computed convergence factors using maximally overlapping multiplicative Schwarz on $\ell \times 1$ subdomains for the FD discretization \eqref{eq:rot-fd} of the grid-aligned diffusion equation \eqref{eq:aligned}.
    \textbf{Left:} Convergence factor $\rho_{\ell, 1}$. 
    \textbf{Right:} $1 - \rho_{\ell, 1}$ in log-log space.
    \label{fig:NUM-vs-LFA-1d}
    }
\end{figure}

\begin{figure}[b!]
    \centering
    \includegraphics[width=0.4\linewidth]{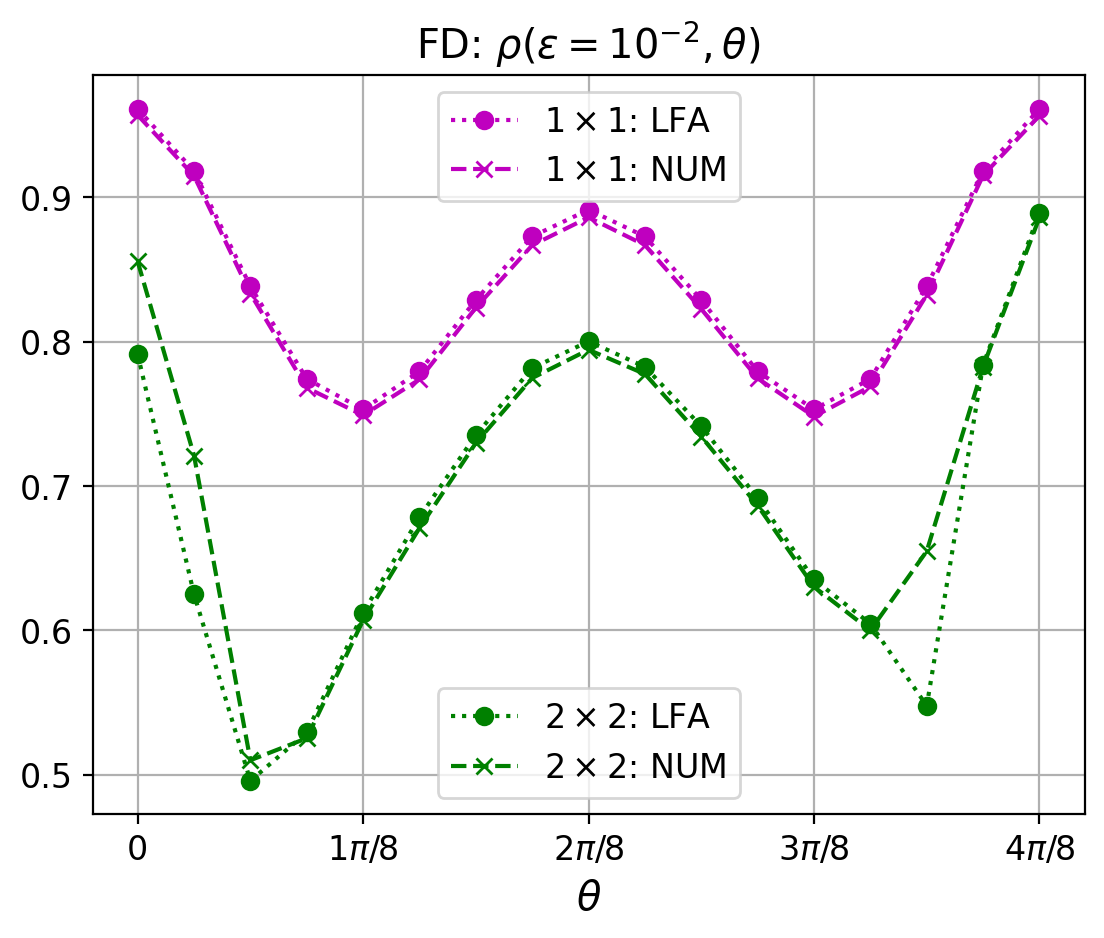}
    \includegraphics[width=0.4\linewidth]{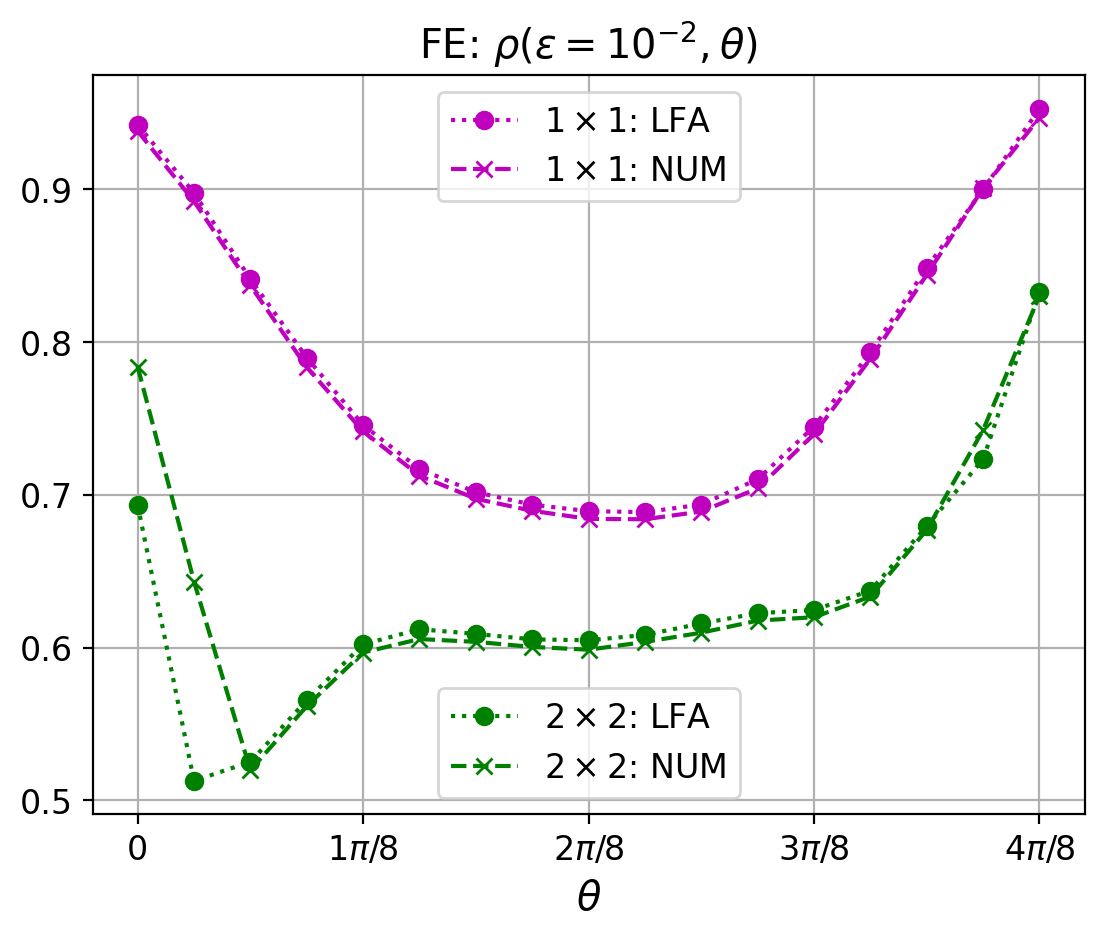}
    \caption{Supporting numerical evidence for the LFA two-grid convergence  predictions in \cref{fig:LFA-rot} for the diffusion problem \eqref{eq:rot} with anisotropy ratio $\epsilon = 10^{-2}$ and rotation angle $\theta$.
    \textbf{Left:} FD discretization.
    \textbf{Right:} FE discretization.
    \label{fig:NUM-vs-LFA-rot}
    }
\end{figure}

First we reconsider the $\ell \times 1$ maximally overlapping subdomains, for which the LFA was presented in \cref{sec:LFA:ellx1,sec:LFA:ellx1-theory}.
Recall that the LFA prediction is that the resulting smoother, and hence two-grid method, is not $\epsilon$-robust for any fixed $\ell$.
In \cref{fig:NUM-vs-LFA-1d} we plot the LFA two-grid convergence factor alongside W-cycle convergence factors from $n_0 = 256$ fine grids.
Overall, there is remarkably close agreement between these two quantities.
Considering the right panel, we see that there is an apparent qualitative difference for sufficiently small $\epsilon$, with LFA predicting a convergence factor tending to zero, while the numerical convergence factors appear to be plateauing.
Further investigation (not shown here) suggests that this is not actually a genuine inconsistency between the two, but rather is an artifact due to the asymptotic convergence factor of the W-cycle still not being reached after 100 iterations. 
That is, if the maximum allowable number of iterations is increased significantly (e.g., more than 1000) we no longer see a plateauing of the numerical convergence factors in the right of \cref{fig:NUM-vs-LFA-1d}.
However, to keep the runtimes of our tests to a reasonable duration we simply cap the maximum number of iterations to 100.

Next we reconsider maximally overlapping square subdomains, for which  LFA results were shown in \cref{sec:LFA:2x2-align,sec:LFA:2x2-rot}.
Recall that LFA two-grid convergence factors are shown in the right-hand column of \cref{fig:LFA-rot} as a function of the rotation angle $\theta$.
Numerically computed and LFA predicted convergence factors are shown in \cref{fig:NUM-vs-LFA-rot} as a function of $\theta$ for anisotropy ratio $\epsilon = 10^{-2}$.
The numerical convergence factors shown in \cref{fig:NUM-vs-LFA-rot} are those of a two-grid method applied to fine grids with $n_0 = 256$.
Considering the results in the figure, we see that agreement between the numerical and LFA convergence factors is quite good for almost all $\theta$ values sampled, and that where there is disagreement it tends to be in regions where the convergence factor is changing quickly with $\theta$.
Finally notice that, just as in \cref{fig:LFA-rot},  convergence of the FD discretization degrades around $\theta \approx \pi/4$ while that of the FE discretization does not.

\begin{figure}[t!]
    \centering
    \includegraphics[width=0.4\linewidth]{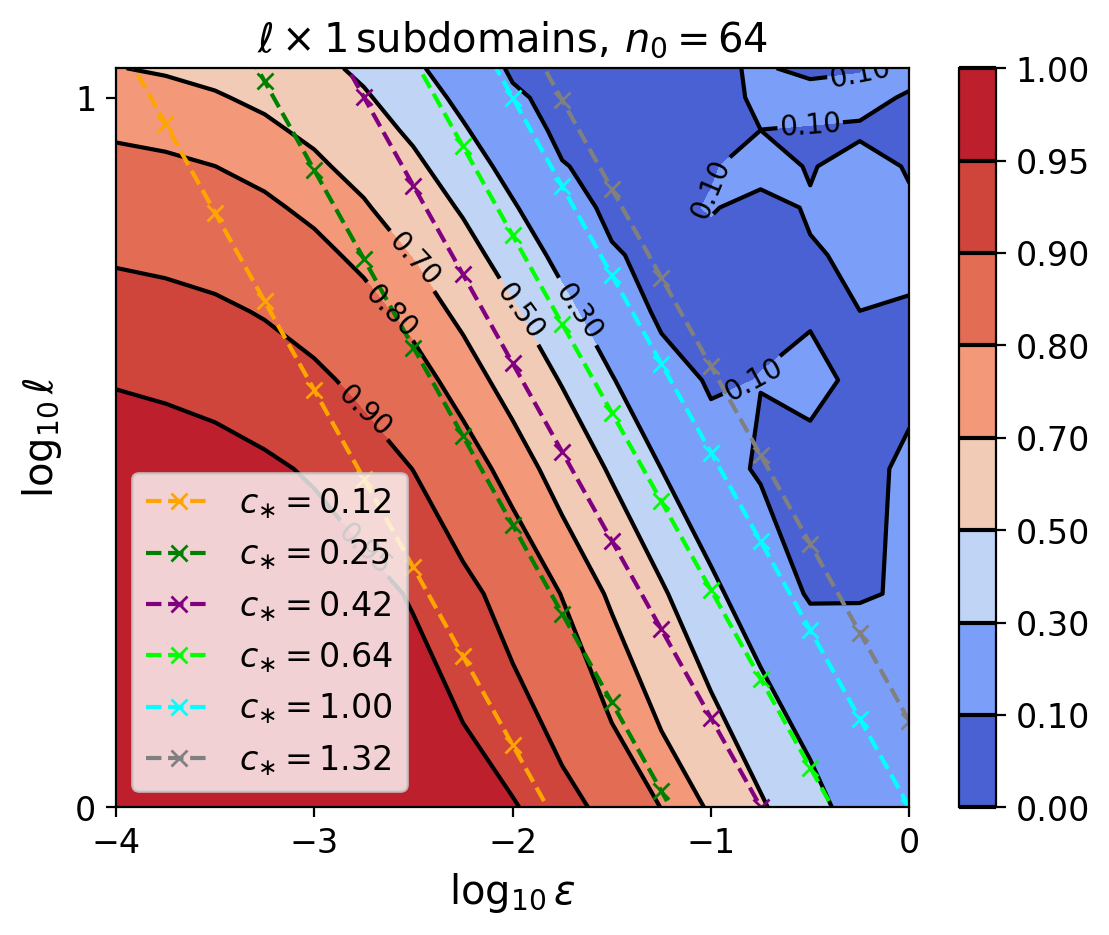}
    \includegraphics[width=0.4\linewidth]{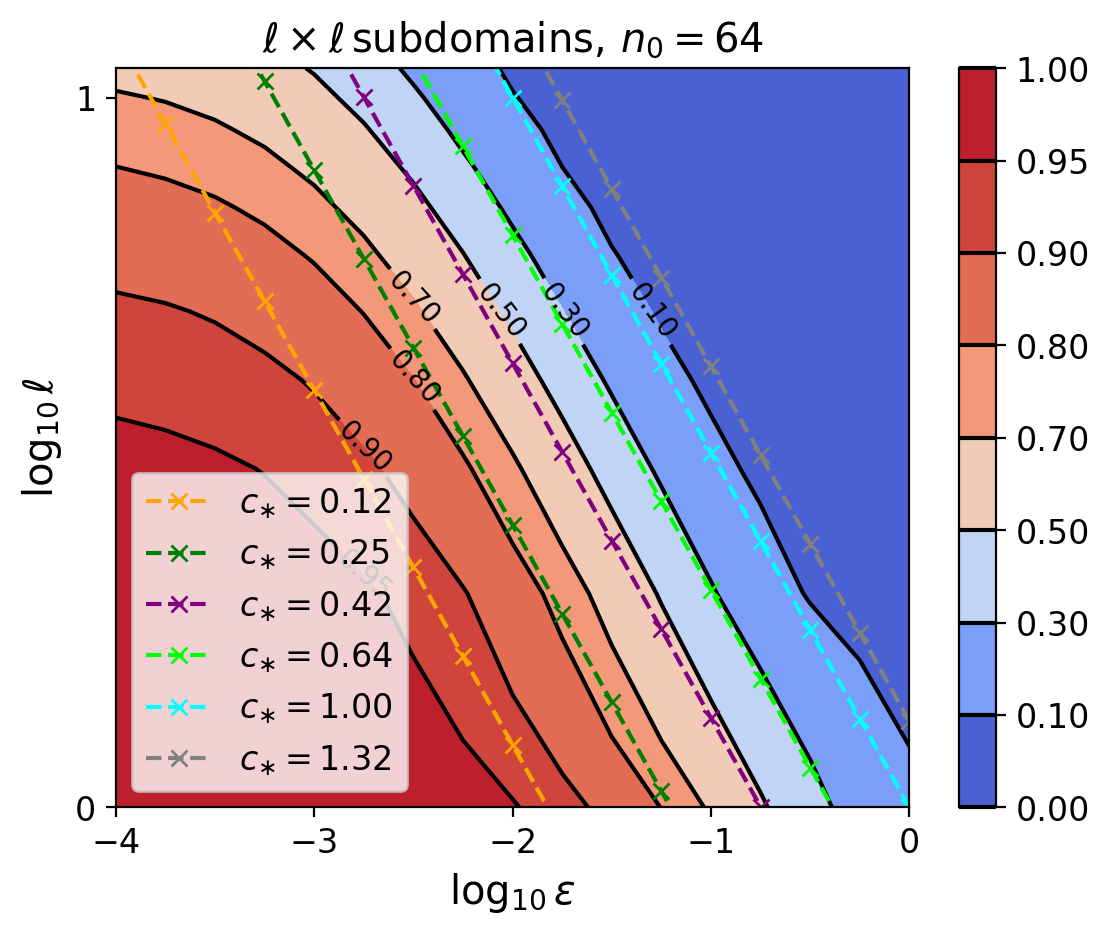}
    \includegraphics[width=0.4\linewidth]{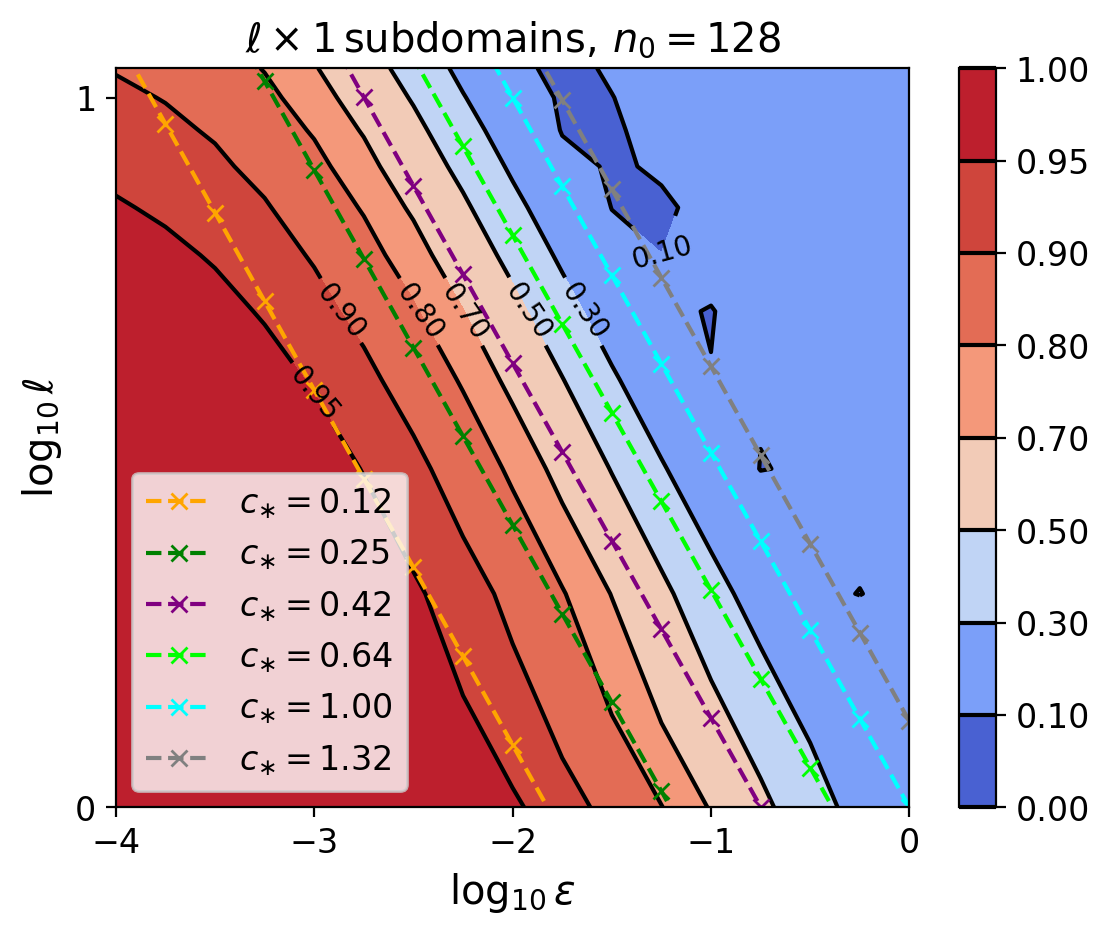}
    \includegraphics[width=0.4\linewidth]{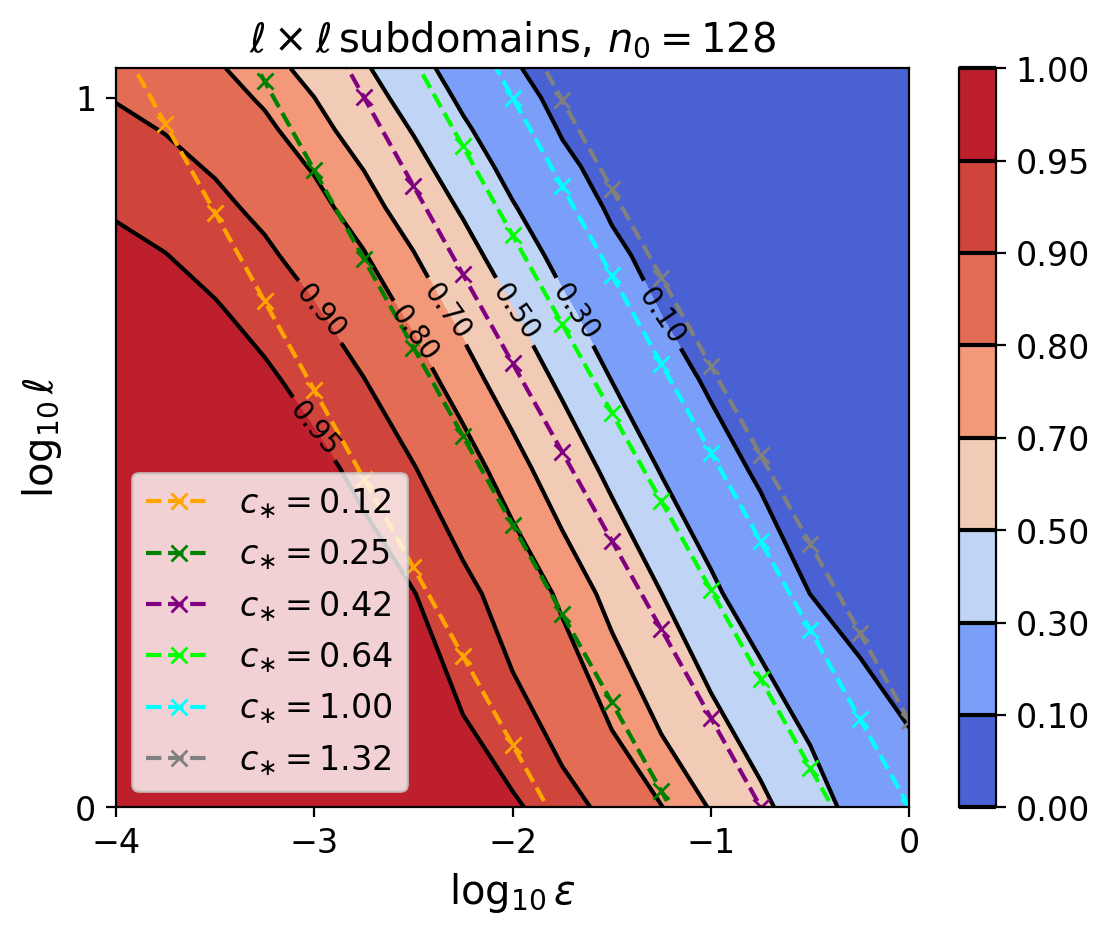}
    \includegraphics[width=0.4\linewidth]{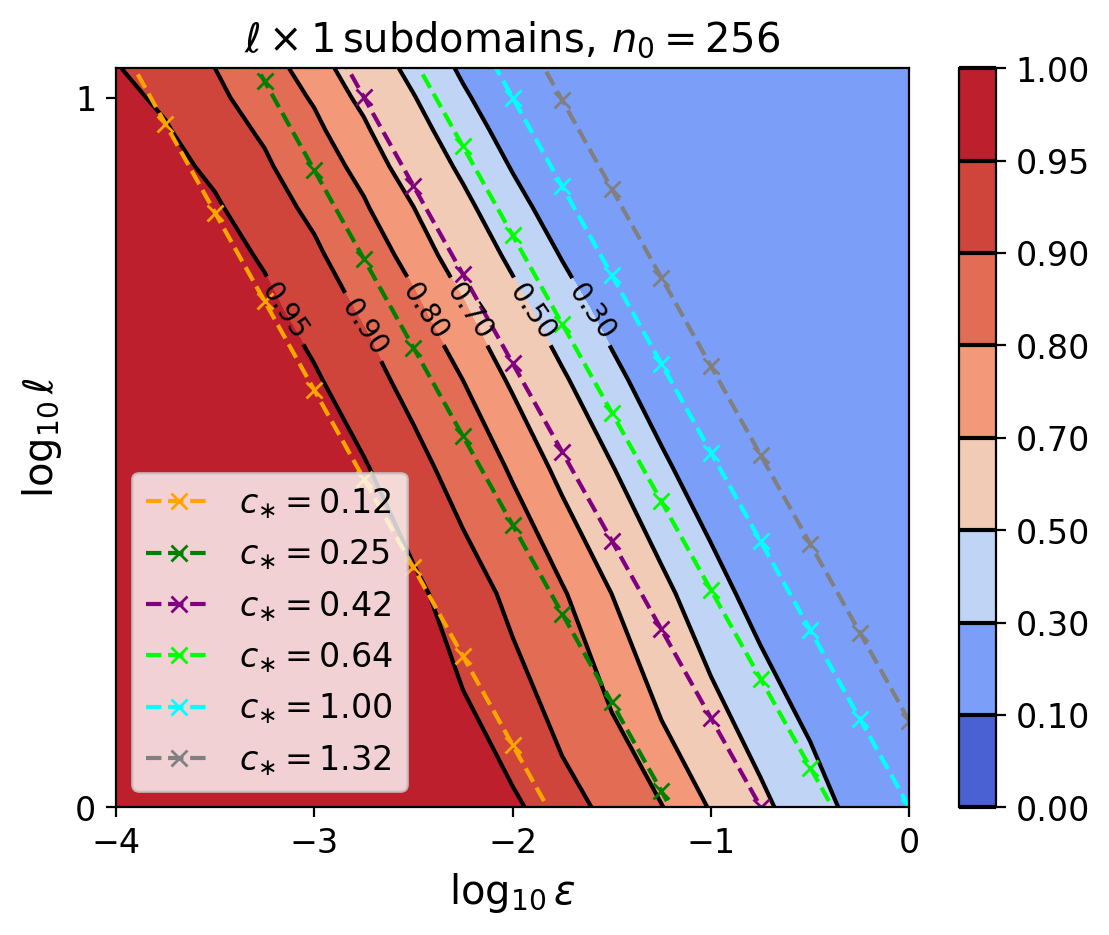}
    \includegraphics[width=0.4\linewidth]{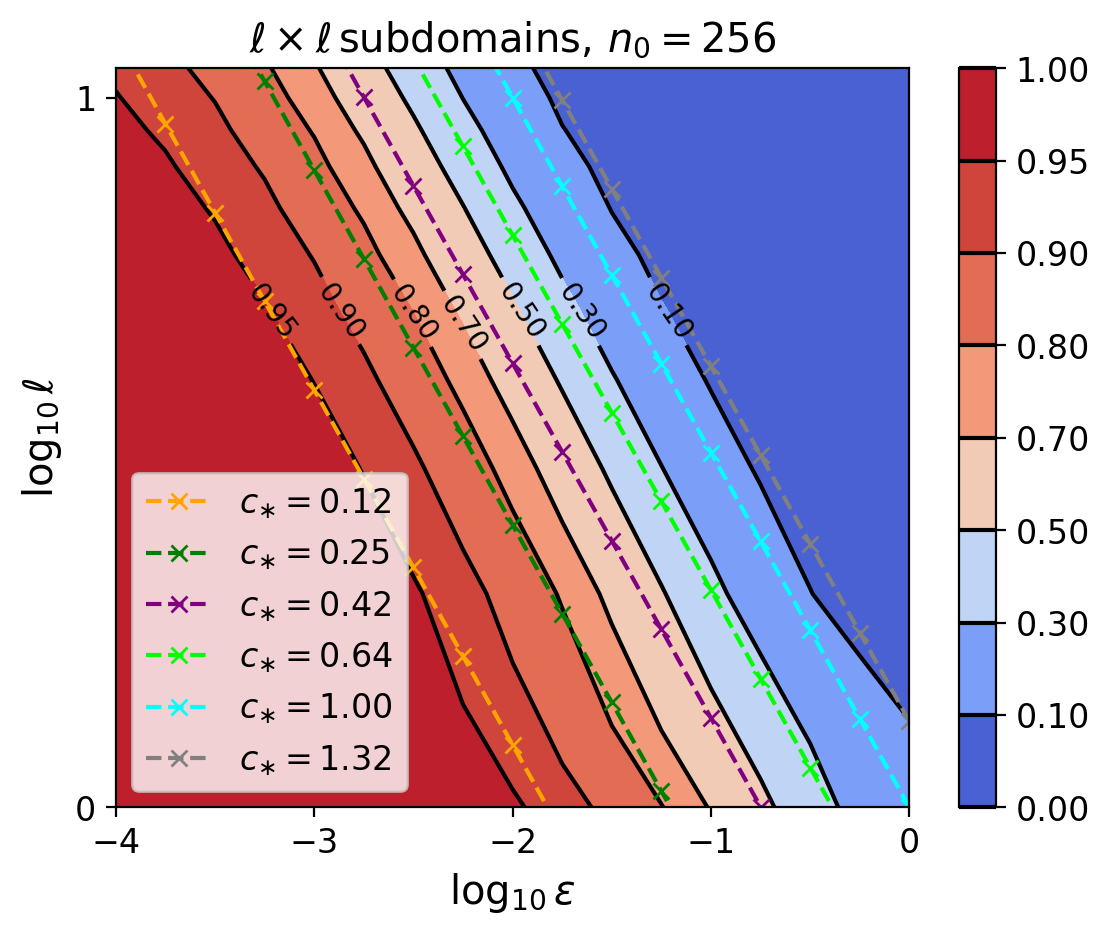}
    \caption{FD discretization of grid-aligned problem \eqref{eq:aligned}.
    Numerically measured V-cycle convergence factors on $n_0 \times n_0$ fine grids using maximally overlapped Schwarz subdomains of size $\ell \times 1$ (\textbf{left} column) or $\ell \times \ell$ (\textbf{right} column). 
    Colored dashed lines are curves $\ell = c_* \epsilon^{-1/2}$ for $c_*$ values shown in legends.
    Note that $\ell$ is sampled at $1, 2, \ldots, 12$ in these plots.
    \label{fig:NUM-max-overlap}
    }
\end{figure}

Our next set of test results is shown in \cref{fig:NUM-max-overlap}. Here we contour V-cycle convergence factors using maximally overlapping subdomains of sizes $\ell \times 1$ (left column) and $\ell \times \ell$ (right column) for several different fine-grid sizes $n_0$.
Observe that convergence factors for the solver using $\ell \times \ell$ subdomains are only very slightly smaller than those corresponding to $\ell \times 1$ subdomains---for example, for a fixed $n_0$, contour lines for both subdomain types are placed almost identically relative to the  $c_* \in \{0.12, 0.42, 1.00\}$ dashed lines.
On one hand this is surprising since the smoother on square subdomains does significantly more work. That is, ignoring differences that occur at boundaries, one smoothing sweep using $\ell \times 1$ subdomains updates every DOF in the domain $\ell$ times via solving $\ell$ residual correction problems of dimension $\ell \times \ell$, while a smoothing sweep using $\ell \times \ell$ subdomains updates every DOF in the domain $\ell^2$ times via solving $\ell^2$ residual correction problems of dimension $\ell^2 \times \ell^2$.
On the other hand, connections in the $y$-direction are weak (they are ${\cal O}(\epsilon)$), so that residual corrections using more information in the $y$-direction, as square subdomains do, are not likely to be significantly more accurate than residual corrections taking no $y$-directional information into account, as rectangular subdomains do not.

Further considering \cref{fig:NUM-max-overlap}, for sufficiently large $n_0$, it appears that level curves of each solver's convergence factor are given by $\ell \propto \epsilon^{-1/2}$. 
Recall that the smoothing factor for $\ell \times 1$ subdomains takes the form $\mu_{\ell, 1}(\epsilon) = 1 - \delta \ell (\ell + 1) \epsilon + {\cal O}(\epsilon^2)$ for some constant $\delta$ (see \cref{sec:LFA:ellx1-theory}).
Thus, for sufficiently small $\epsilon$, where the two-grid convergence factor is essentially set by the smoothing factor, we should see $\rho_{\ell, 1} \approx \mu_{\ell, 1}^2 \approx 1 - 2 \delta \ell(\ell+1)\epsilon$ (see \cref{sec:LFA:2x2-align} for analogous arguments).
The level curves of $\rho_{\ell, 1}$ can thus be expected to be of the form $\textrm{constant} \approx \ell \epsilon^{1/2}$, which is exactly what we see in the right column of \cref{fig:NUM-max-overlap}.
These functions appear to be the level curves of the convergence factors in the $\ell \times \ell$ case too, which is perhaps unsurprising given our argument above regarding connections in the $y$-direction being weak.

Finally, \cref{fig:NUM-max-overlap} illustrates how the LFA results and heuristics are only applicable in settings where the Schwarz subdomain size is bounded sufficiently far from the global domain size. 
This is particularly apparent by considering the top left quadrant of each plot, where the difference between LFA-predicted and observed convergence factors is significant for small $n_0 = 64$, but by $n_0=256$ the LFA-predicted convergence factors provide excellent estimates of observed convergence. 
This is unsurprising: LFA predictions are based on assumptions of infinite domains and periodic boundaries, while Schwarz on $\ell \times 1$ subdomains is a direct solver for these Dirichlet boundary problems when $\epsilon = 0$ and $\ell = n_0$, so that the LFA predictions are expected to break down as $\epsilon \to 0$ and $\ell \to n_0$.
In any event, the LFA heuristics are generally quite accurate for a wide range of mesh sizes and anisotropies.

For our final test we investigate the effects of using non-maximally overlapped $\ell \times 1$ subdomains, allowing for an overlap of $0 \leq \textrm{overlap}(\ell) \leq \ell-1$ in the $x$-direction.
V-cycle convergence factors on meshes with $n_0 = 256$ are contoured in \cref{fig:NUM-1d-overlap} for anisotropy ratios $\epsilon \in \{10^{-2}, 10^{-3}\}$. 
The results are qualitatively similar for both values of $\epsilon$, with the convergence factor for fixed $\ell$ being a decreasing function of $\textrm{overlap}(\ell)$.
It is interesting to note that the convergence factor does not uniformly decrease with increasing overlap, as evidenced by the fact that, for fixed $\ell$, contours tend to be bunched relatively closer for $\textrm{overlap}(\ell) \approx 0$ than $\textrm{overlap}(\ell) \approx \ell-1$.
In other words, the gain in convergence speed from adding a small amount of overlap is relatively large compared to the increase from adding a large amount of overlap. For example, on the left at $\ell=9$ with overlap of $\ell-1=8$, we have convergence factor $\approx 0.225$, whereas for overlap of two we have convergence factor of $\approx 0.475$. Here, for roughly a quarter of the work, convergence degrades by only a factor of two in iteration count. 
Note that the smoother is not $\epsilon$-robust \textit{independent} of its overlap, because, clearly, the smoother is strongest when maximally overlapped and all of our previous results demonstrate that the maximally overlapped smoother is not $\epsilon$-robust.
However, these results do indicate a potentially practical method with partially overlapping anisotropy-aligned blocks, with block size a function of the anisotropy ratio. In many practical applications $\epsilon$ is large (i.e. not small) enough that block smoothing with block size $1/\sqrt{\epsilon}\times 1$ may be a tractable computational cost for good multilevel convergence, especially if one does not use maximally overlapping blocks. 

\begin{figure}[t!]
    \centering
    \includegraphics[width=0.425\linewidth]{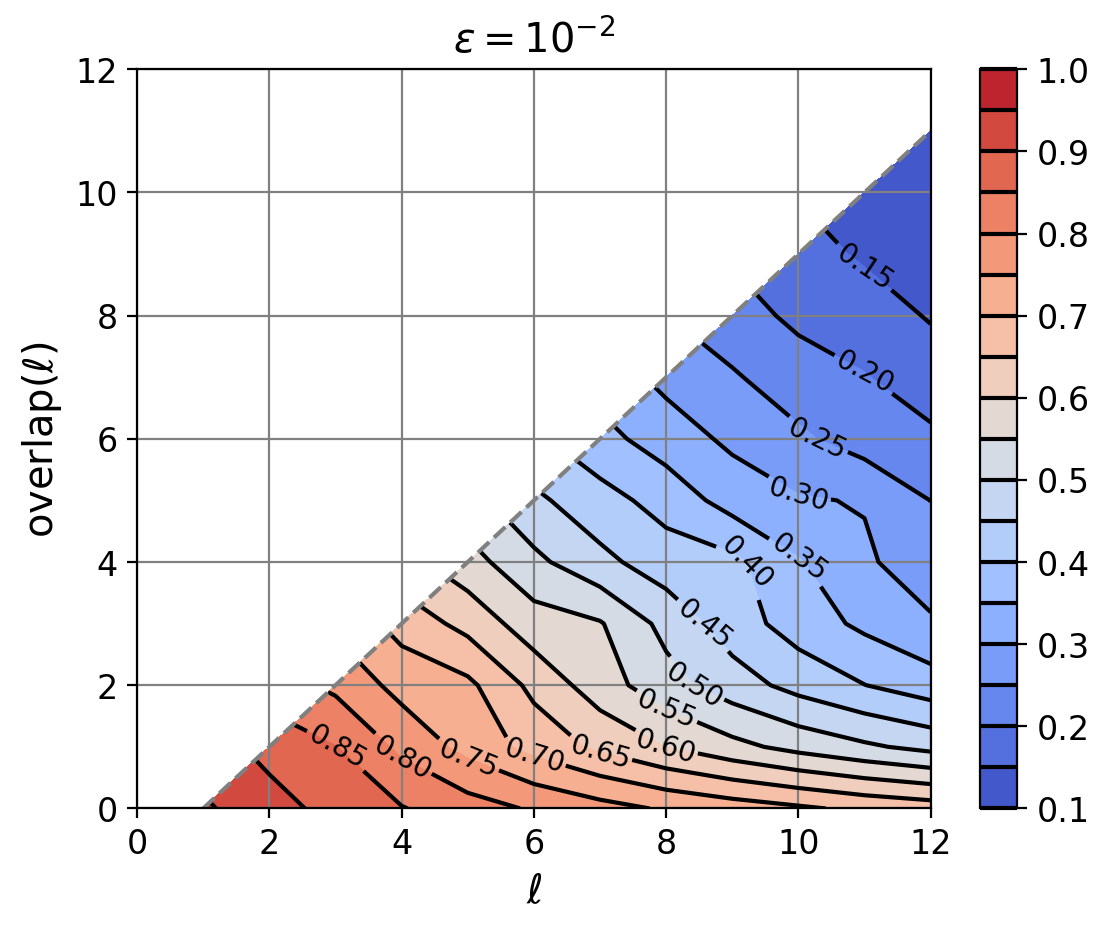}
    \includegraphics[width=0.425\linewidth]{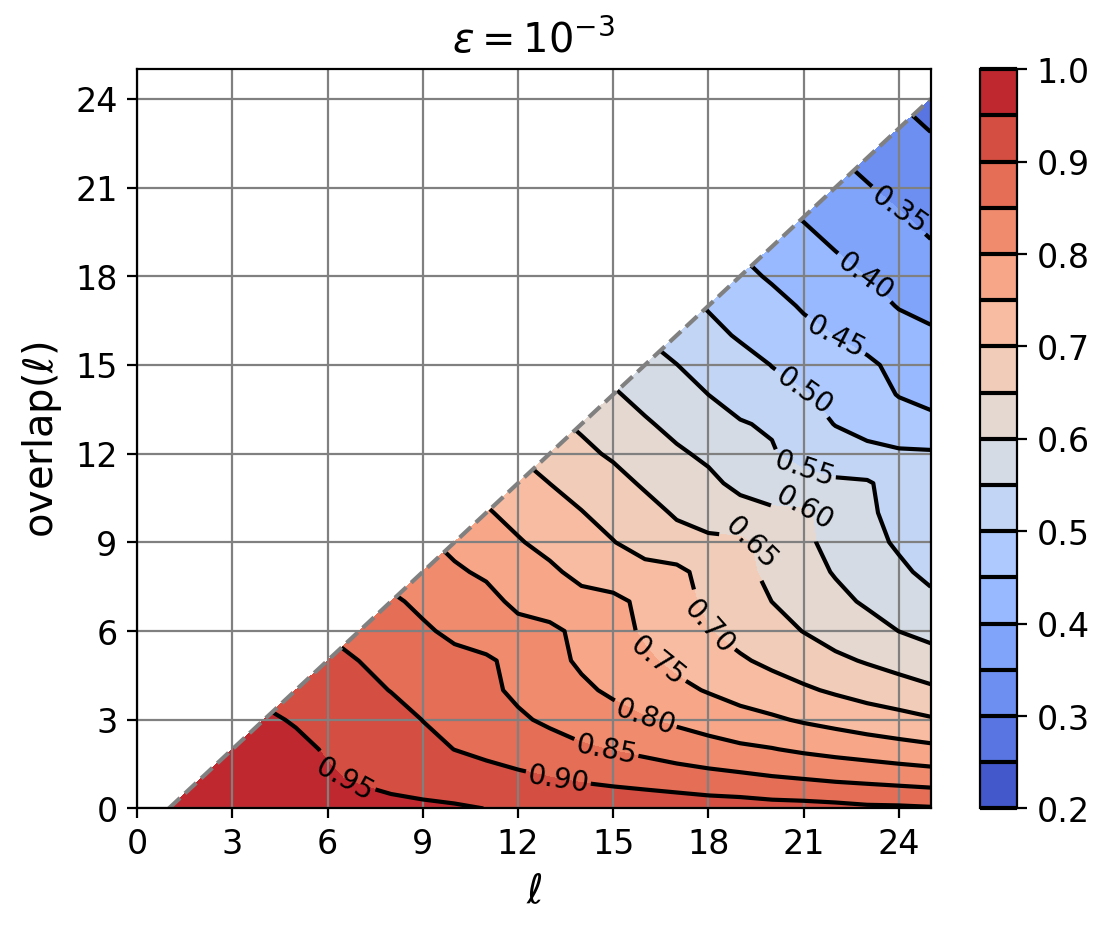}
    \caption{V-cycle convergence factors for the FD discretization of the grid-aligned problem \eqref{eq:aligned} for anisotropy ratio $\epsilon = 10^{-2}$ (\textbf{left}) and $\epsilon = 10^{-3}$ (\textbf{right}). 
    The smoother is overlapping Schwarz on $\ell \times 1$ subdomains with an overlap of ${0 \leq \textrm{overlap}(\ell) \leq \ell-1}$ in the $x$-direction.
    \label{fig:NUM-1d-overlap}
    }
\end{figure}

\begin{remark}[Alternating Schwarz smoothing]
Inspired by the commonplace use of alternating line smoothing for non-grid-aligned anisotropies, we have also tested maximally overlapping Schwarz using alternating one-dimensional subdomains. That is, first applying maximally overlapped Schwarz on $\ell \times 1$ subdomains, sweeping east to west then south to north, then taking a second pass of maximally overlapped Schwarz on $1 \times \ell$ subdomains sweeping south to north then east to west.
We do not show the details of these tests here, however, because we find this smoother is less efficient than using two passes of maximally overlapping Schwarz on $\ell \times 1$ subdomains, at least provided the non-alignment angle in \eqref{eq:rot-fd} satisfies $\theta \in (0, \pi/4)$.
This is perhaps unsurprising because we also
find that analogous conclusions hold for alternating line smoothers based on first  smoothing all $x$-lines and then all $y$-lines.
In particular, if $\theta \in [0, \pi/4)$ in \eqref{eq:rot-fd}, then it is more cost effective to use two passes of $x$-line smoothing (or two passes of $y$-line smoothing if $\theta \in (\pi/4 \pi/2]$)  rather than passing first over $x$-lines and then $y$-lines.
As discussed previously, alternating line smoothing is not always $\epsilon$-robust, e.g., failing on \eqref{eq:rot-fd} for $\theta = \pi/4 = 45^{\circ}$ (see \cref{sec:LFA:2x2-rot} and \cite[Table 7.7.1]{Wesseling-1992}).
Similarly, we find that the above-described smoother based on alternating local subdomains also fails to be $\epsilon$-robust for \eqref{eq:rot-fd} when $\theta = \pi/4 = 45^{\circ}$.
\end{remark}

To conclude we note an interesting connection. 
\begin{remark}[Robustness on anisotropic meshes]
Consider the isotropic diffusion operator $-\partial_{xx} - \partial_{yy}$ discretized with standard second-order FDs on a mesh with spacing $h_x = h$ and $h_y = h_x \epsilon^{-1/2}$ in the $x$- and $y$-directions, respectively, so that $h_x \ll h_y$ for $0 < \epsilon \ll 1$.
Then, it is well-known that the resulting stencil is identical to that in \eqref{eq:rot-fd}, corresponding to the FD discretization of the anisotropic operator $-\partial_{xx} - \epsilon \partial_{yy}$ discretized on an isotropic grid with spacing $h$, e.g., \cite{Briggs-etal-2000}.
However, the robustness requirement that the Schwarz blocks scale in size as $\ell_*(\epsilon) \times 1 = {\cal O}(\epsilon^{-1/2}) \times 1$ now has a different interpretation.
Specifically, the since the physical distance between adjacent grid nodes in the $x$- and $y$-directions is $h_x$ and $h_y$, respectively, the physical size of the required blocks is $h_x {\cal O}(\epsilon^{-1/2}) \approx h_y$ and $h_y$ in the $x$- and $y$-directions, respectively. 
That is, for the discretization of the isotropic PDE on an anisotropic grid, the robustness requirement is that the blocks are approximately square in physical space with side length on the order of largest grid spacing.
\end{remark}

%
\section{Conclusions}
\label{sec:con}

For the multigrid solution of anisotropic diffusion problems we have analyzed certain classes of multiplicative overlapping Schwarz smoothers with general subdomain sizes. This general class of smoother includes point-wise Gauss--Seidel and line smoothing as special cases. 
Despite recent successful results applying overlapping Schwarz smoothers with geometric multigrid methods to directional (transport) problems, e.g. \cite{farrell2019augmented,abu2023monolithic}, the efficacy of these methods does not extend to anisotropic diffusion. 
For fixed subdomain size, we prove that the smoothing properties of these methods deteriorate without bound as the level of anisotropy increases, provided that the subdomain size is bounded far enough away from the global domain size.
Moreover, we find that this class of methods can yield anisotropy-robust smoothing so long as the diameter of the subdomains used scales as ${\cal O}(\epsilon^{-1/2})$ for anisotropy ratio $\epsilon \in (0, 1]$ in the underlying diffusion equation.
Our theoretical findings based on LFA are supported by numerical experiments.
As such, the definitive conclusion is that one should exercise extreme caution when using non-global block overlapping smoothers for anisotropic problems. Instead, one should opt to use global line/plane smoothing methods since these are known to be anisotropy robust, or specialized methods for handling large anisotropies such as \cite{Wimmer-etal-2024,chacon2024asymptotic}.

Future work could extend our analysis to study the efficacy of combining overlapping Schwarz smoothers with semi-coarsening strategies, which are typically coupled with either point-wise smoothers or line smoothers, or utilize overlapping smoothers in an algebraic multigrid context.

\section*{Acknowledgements}
This research has been funded by the Los Alamos National Laboratory (LANL) Advanced Simulation and Computation (ASC) and Directed Research and Development (LDRD) programs (LDRD project number 20240261ER). The research was performed under the auspices of the National Nuclear Security Administration of the U.S. Department of Energy at Los Alamos National Laboratory, managed by Triad National Security, LLC under contract 89233218CNA000001. LA-UR-24-33282.
%







\bibliography{schwarz-refs}

\appendix

%
\section{Derivations of $2 \times 2$ theoretical results}
\label{app:2x2}

Before proving the lemma below (a copy of \cref{lem:mu-2x2} from \cref{sec:LFA:2x2-align}) we introduce some simplified notation.
Since this result deals with the grid-aligned problem \eqref{eq:aligned}, the respective FD and FE stencils can be simplified from those in \eqref{eq:rot-fd} and \eqref{eq:rot-fe} for the more general rotated problem \eqref{eq:rot}.
Specifically, the general 9-point stencil \eqref{eq:9-point} and its analogue in \eqref{eq:9-point-wc} may be written as 
\begin{subequations} \label{eq:9-point-simple-both}
\begin{align} \label{eq:9-point-simple}
    [A] &= 
    \begin{bmatrix}
        \snw & \sn & \sne \\
        \sw & \scc & \se \\
        \ssw & \ss & \sse
    \end{bmatrix}
    =
    \begin{bmatrix}
        \snw & \sn & \snw \\
        \sw & \scc & \sw \\
        \snw & \sn & \snw
    \end{bmatrix},
    \\
    \label{eq:9-point-simple-wc}
    [\wc{A}] &= 
    \begin{bmatrix}
        \csnw & \csn & \csne \\
        \csw & \cscc & \cse \\
        \cssw & \css & \csse
    \end{bmatrix}
    =
    \begin{bmatrix}
        \snw e^{\i(-\omega_1 + \omega_2)} & \sn e^{+ \i\omega_2} & \snw e^{\i(+\omega_1 + \omega_2)} \\
        \sw e^{-\i \omega_1} & \scc & \sw e^{+ \i\omega_1} \\
        \sn e^{\i(-\omega_1 - \omega_2)} & \sn e^{-\i \omega_2} & \ssw e^{\i(+\omega_1 - \omega_2)}
    \end{bmatrix}.
\end{align}
\end{subequations}
Setting $\theta = 0$ in the stencils \eqref{eq:rot-fd} and \eqref{eq:rot-fe} gives stencil coefficients listed above as
\begin{subequations}
\label{eq:aligned-simple-stencil}
\begin{align} 
    \label{eq:aligned-simple-stencil-FD}
    \textrm{FD:} \quad 
    &\snw = 0,  
    \quad
    &&\sn = -\epsilon, 
    \quad
    &&\sw = -1,  
    \quad
    &&\scc = 2(1 + \epsilon),
    \\
    \label{eq:aligned-simple-stencil-FE}
    \textrm{FE:} \quad 
    &\snw = - \frac{1}{6} - \frac{1}{6} \epsilon, 
    \quad 
    &&\sn = \frac{1}{3} - \frac{2}{3}\epsilon, 
    \quad
    &&\sw = -\frac{2}{3} + \frac{1}{3} \epsilon, 
    \quad
    &&\scc = \frac{4}{3} + \frac{4}{3} \epsilon.
\end{align}
\end{subequations}

\begin{lemma}[Copy of \cref{lem:mu-2x2}] \label{lem:mu-2x2-copy}
    Consider the anisotropic diffusion equation \eqref{eq:aligned} with anisotropy ratio ${\epsilon \in [0,1]}$ discretized with either FDs \eqref{eq:rot-fd} or FEs \eqref{eq:rot-fe}.
    Let $\mu_{2,2}(\epsilon)$ be the smoothing factor \eqref{eq:mu-def} of maximally overlapping multiplicative Schwarz with $2 \times 2$ subdomains (see \cref{sec:LFA:2x2}).
    Then, 
    \begin{align} \label{eq:mu-2x2-copy}
        1 - c \epsilon + {\cal O}(\epsilon^2)
        \leq 
        \mu_{2,2}(\epsilon)
        \leq 1,
    \end{align}
    with $c = 12$ and $c = 19.2$ for the FD and FE discretizations, respectively.
    As such, this smoother is not $\epsilon$-robust in the sense of \cref{def:robust} for either discretization, since ${\lim_{\epsilon \to 0^+} \mu_{2,2}(\epsilon) = 1}$.
\end{lemma}

\begin{proof}
    From \eqref{eq:mu-def}, we have the lower bound $\mu_{2,2} \geq |\wt{s}(\omega_1, \omega_2)|$ for any high frequency pair. 
    In particular, the lower bound in \eqref{eq:mu-2x2-copy} corresponds to the pair $(\omega_1, \omega_2) = (0, 3 \pi/2)$.
    Now we work through evaluating the terms in the $4 \times 4$ linear system ${\cal A} \bm{\alpha} = [A_{ij} D (U - I) - D C]
    \bm{\alpha}
    =
    [D \bm{c} - A_{ij} D \bm{e}_4] \alpha_0
    =
    \bm{b} \alpha_0$, as given in \eqref{eq:SCH-2x2-system}, at $(\omega_1, \omega_2) = (0, 3 \pi/2)$ since this governs the symbol $\wt{s}(\omega_1, \omega_2) = ({\cal A}^{-1} \bm{b})_1$ of the smoother, see \eqref{eq:SCH-2x2-symbol-def}.

    Evaluating the matrix $D$ in \eqref{eq:SCH-2x2-D} at $(\omega_1, \omega_2) = (0, 3 \pi/2)$ we get $D = \diag
     \left(
        1,
        1, 
        -\i,
        -\i
    \right)$.
    Simplifying the projected matrix $A_{ij}$ in \eqref{eq:SCH-2x2-Aij} by using the simplified 9-point stencil \eqref{eq:9-point-simple} yields
    \begin{align} \label{eq:AijD-2d-simplified}
        A_{ij}
    =
    \begin{bmatrix}
        \scc & \se & \sn & \sne \\
        \sw & \scc & \snw & \sn \\
        \ss & \sse & \scc & \se \\
        \ssw & \ss & \sw & \scc 
    \end{bmatrix}
    =
    \begin{bmatrix}
        \scc & \sw & \sn & \snw \\
        \sw & \scc & \snw & \sn \\
        \sn & \snw & \scc & \sw \\
        \snw & \sn & \sw & \scc 
    \end{bmatrix}
    \quad
    \longrightarrow
    \quad
    A_{ij} D
    = 
 \begin{bmatrix}
        \scc & \sw & -\i \sn & -\i\snw \\
        \sw & \scc & -\i \snw & -\i\sn \\
        \sn & \snw & -\i \scc & -\i\sw \\
        \snw & \sn & -\i \sw & -\i\scc 
    \end{bmatrix}.
\end{align}
Recalling the difference matrix $U - I$ from \eqref{eq:SCH-2x2-system} we have
\begin{align} \nonumber
    U - I = 
    \begin{bmatrix}
        -1 & \hphantom{-}1 \\
        & -1 & \hphantom{-}1 \\
        & & -1 & \hphantom{-}1 \\
        & & & -1
    \end{bmatrix},
    \quad
    A_{ij} D (U - I)
    = 
 \begin{bmatrix}
        -\scc & \scc-\sw & \sw+\i \sn & \i(\snw - \sn) \\
        -\sw & \sw-\scc & \scc+\i \snw & \i(\sn - \snw) \\
        -\sn & \sn-\snw & \snw+\i \scc & \i(\sw - \scc) \\
        -\snw & \snw-\sn & \sn+\i \sw & \i(\scc - \sw) 
    \end{bmatrix}.
\end{align}
Now consider the other matrix $DC$ appearing in ${\cal A}$.
Evaluating the symmetrically simplified stencil in \eqref{eq:9-point-simple-wc} at $(\omega_1, \omega_2) = (0, 3 \pi/2)$ gives
\begin{align} \nonumber
    [\wc{A}] &= 
    \begin{bmatrix}
        \csnw & \csn & \csne \\
        \csw & \cscc & \cse \\
        \cssw & \css & \csse
    \end{bmatrix}
    =
    \begin{bmatrix}
        -\i \snw  & -\i \sn  & -\i \snw  \\
        \sw  & \scc & \sw  \\
        \i \snw & \i \sn  & \i \snw 
    \end{bmatrix}.
\end{align}
Next let us recall the matrix $D C$ from \eqref{eq:SCH-2x2-Rold} and then plug in the above  simplification to get
\begin{align} 
\nonumber
D C 
&=
D
\begin{bmatrix}
    \cssw + \css + \csse + \csw & \cscc & \cse + \csnw  &  \csn  \\ 
    \cssw + \css + \csse        &  \csw & \cscc + \cse  & \csnw  \\ 
    \cssw                       & \css  & \csse + \csw  & \cscc  \\ 
         0                      & \cssw & \css + \csse  & \csw  
\end{bmatrix}
=
D
\begin{bmatrix}
    \i(2 \snw + \sn) + \sw & \scc   & \sw -\i \snw    &  -\i \sn  \\ 
    \i(2 \snw + \sn)        &  \sw   & \scc + \sw      & -\i\snw  \\ 
    \i\snw                          & \i \sn   & \i\snw + \sw    & \scc  \\ 
         0                          & \i\snw & \i (\sn + \snw) & \sw  
\end{bmatrix},
\\
\nonumber
&=
\begin{bmatrix}
    \i(2 \snw + \sn) + \sw & \scc   & \sw -\i \snw    &  -\i \sn  \\ 
    \i(2 \snw + \sn)         &  \sw   & \scc + \sw      & -\i\snw  \\ 
      \snw                          & \sn    &  \snw -\i\sw  &-\i \scc  \\ 
         0                          & \snw  &  \sn + \snw     & -\i\sw  
\end{bmatrix}.
\end{align}
Combining with the expression for $A_{ij} D (U - I)$ above, we have
\begin{align} \label{eq:calA-simplified-2x2}
    {\cal A} 
    = 
    A_{ij} D (U - I) - DC
    =
    \begin{bmatrix}
        -(\scc + \sw) - \i (2 \snw + \sn) & -\sw  & \i (\sn + \snw) & \i \snw \\
        -\sw - \i (2 \snw + \sn)         & -\scc  & -\sw + \i \snw  & \i \sn \\
        -\sn - \snw                      & -\snw & \i (\scc + \sw)  & \i \sw \\
        -\snw                            & -\sn  & -\snw + \i \sw  & \i \scc
    \end{bmatrix}.
\end{align}
Now consider the right-hand side vector $\bm{b} = D \bm{c} - A_{ij} D \bm{e}_4$.
Simplifying $D \bm{c}$ from \eqref{eq:SCH-2x2-Rold} we find
\begin{align} \nonumber
    D \bm{c}
    =
    D
    \begin{bmatrix}
    \csne \\
     \csn + \csne \\
       \cse + \csnw + \csn + \csne \\
       \cscc + \cse + \csnw + \csn + \csne
\end{bmatrix}
=
D
\begin{bmatrix}
    -\i\snw \\
     -\i(\sn + \snw) \\
       \sw - \i (2\snw + \sn) \\
       \scc + \sw -\i (2 \snw + \sn)
\end{bmatrix}
=
\begin{bmatrix}
    -\i\snw \\
     -\i(\sn + \snw) \\
       -\i \sw - (2\snw + \sn) \\
       -\i(\scc + \sw) - (2 \snw + \sn)
\end{bmatrix}.
\end{align}
Noting that $A_{ij} D \bm{e}_4$ is the last column of $A_{ij} D$, combining \eqref{eq:AijD-2d-simplified} with the above gives
\begin{align} \label{eq:b-simplified-2x2}
    \bm{b} = D \bm{c} - A_{ij} D \bm{e}_4
    =
    \begin{bmatrix}
    -\i\snw \\
     -\i(\sn + \snw) \\
       -\i \sw - (2\snw + \sn) \\
       -\i(\scc + \sw) - (2 \snw + \sn)
    \end{bmatrix}
    + \i
     \begin{bmatrix}
        \snw \\
        \sn \\
        \sw \\
        \scc 
    \end{bmatrix}
    =
    -
    \begin{bmatrix}
        0 \\
        \i \snw \\
        2 \snw + \sn \\
        (2 \snw + \sn) + \i \sw
    \end{bmatrix}.
\end{align}

Plugging the FD stencil entries $\snw = 0$,  $\sn = -\epsilon$, $\sw = -1$, and $\scc = 2(1 + \epsilon)$ from \eqref{eq:aligned-simple-stencil-FD} into \eqref{eq:calA-simplified-2x2} and \eqref{eq:b-simplified-2x2} we get
\begin{align}
    {\cal A}^{\rm FD} 
            =
    \underbrace{
    \begin{bmatrix}
        -1 & 1 & 0 & 0 \\
        1  & -2 & 1 & 0 \\
        0 & 0 & \i  & -\i \\
        0 & 0 & -\i & 2 \i 
    \end{bmatrix}
    }_{{\cal A}^{\rm FD}_0}
    +
    \epsilon
    \underbrace{
    \begin{bmatrix}
         -2 + \i  & 0 & -\i  & 0 \\
        \i & -2 & 0 & -\i \\
        1 & 0 & 2\i  & 0 \\
        0 & 1 & 0 & 2 \i 
    \end{bmatrix}
    }_{{\cal A}^{\rm FD}_1},
    \quad
    \bm{b}^{\rm FD}
    =
    \underbrace{
    \begin{bmatrix}
        0 \\
        0 \\
        0 \\
        \i 
    \end{bmatrix}
    }_{\bm{b}_0^{\rm FD}}
    +
    \epsilon
    \underbrace{
    \begin{bmatrix}
        0 \\
        0 \\
        1 \\
        1
    \end{bmatrix}
    }_{\bm{b}_1^{\rm FD}}.
\end{align}
Since the system matrix and right-hand side in the linear system ${\cal A} \bm{\alpha} = \bm{b} \alpha_0$ are linear functions in $\epsilon$, we seek a power series solution in the form of $\bm{\alpha} = \bm{\alpha}_0 + \epsilon \bm{\alpha}_1 + {\cal O}(\epsilon^2)$.
Plugging in this power series ansatz into the system ${\cal A} \bm{\alpha} = \bm{b} \alpha_0$ and equating terms with equal powers of $\epsilon$ we find ${\cal A}_0^{\rm FD} \bm{\alpha}_0^{\rm FD} = \bm{b}_0^{\rm FD}$ and ${\cal A}_0^{\rm FD} \bm{\alpha}_1^{\rm FD} = \bm{b}_1^{\rm FD} - {\cal A}_1^{\rm FD} \bm{\alpha}_0^{\rm FD}$.
Solving these two linear systems in sequence we obtain the solutions
\begin{align}
    \bm{\alpha}_0^{\rm FD} = 
    \begin{bmatrix}
        1 \\
        1 \\
        1 \\
        1
    \end{bmatrix}, 
    \quad
    \bm{\alpha}_1^{\rm FD} = 
    \begin{bmatrix}
        -12 \\
        -10 \\
        -6 \\
        -4
    \end{bmatrix}.
\end{align}
Recalling \eqref{eq:SCH-2x2-symbol-def}, the symbol is $\wt{s}(0, 3\pi/2) = ({\cal A}^{-1} \bm{b})_1 = [\bm{\alpha}_0^{\rm FD} + \epsilon \bm{\alpha}_1^{\rm FD} + {\cal O}(\epsilon^2)]_1 = 1 - 12 \epsilon + {\cal O}(\epsilon^2)$.
This completes the proof for the FD case.

Now consider the FE case. Plugging the FE stencil entries $\snw = - \frac{1}{6} - \frac{1}{6} \epsilon$, $\sn = \frac{1}{3} - \frac{2}{3}\epsilon$, $\sw = -\frac{2}{3} + \frac{1}{3} \epsilon$, and $\scc = \frac{4}{3} + \frac{4}{3} \epsilon$ from \eqref{eq:aligned-simple-stencil-FE} into \eqref{eq:calA-simplified-2x2} and \eqref{eq:b-simplified-2x2} we find
\begin{align}
    {\cal A}^{\rm FE}
        =
    \underbrace{
    \frac{1}{6}
    \begin{bmatrix}
        -4 & 4 & \i  & -\i \\
         4 & -8 & 4-\i & 2 \i \\
        -1 & 1  & 4 \i & -4 \i \\
        1  & -2 & 1 - 4 \i  & 8 \i
    \end{bmatrix}
    }_{{\cal A}^{\rm FE}_1}
    +
    \epsilon
    \underbrace{
    \frac{1}{6}
    \begin{bmatrix}
        2(-5 + 3 \i) & -2 & -5\i  & -\i  \\
        2 (-1 + 3 \i) & -8 & - (2 + \i) &  - 4\i  \\
        5  & 1 & 10 \i  & 2\i \\
        1  &  4 &  1 + 2 \i &  8 \i
    \end{bmatrix}
    }_{{\cal A}^{\rm FE}_0},
    \quad
    \bm{b}^{\rm FE}
    =
    \underbrace{
    \frac{1}{6}
    \begin{bmatrix}
        0 \\
        \i \\
        0 \\
        4 \i
    \end{bmatrix}
    }_{\bm{b}^{\rm FE}_0}
    +
    \epsilon
    \underbrace{
    \frac{1}{6}
    \begin{bmatrix}
        0 \\
        \i  \\
        6  \\
        2(3 - \i)
    \end{bmatrix}.
    }_{\bm{b}^{\rm FE}_1}
\end{align}
Since the system matrix and right-hand side are again linear functions in $\epsilon$, we seek a power series solution to ${\cal A} \bm{\alpha} = \bm{b} \alpha_0$ exactly as above. Doing so, we find,
    \begin{align}
        \bm{\alpha}_0^{\rm FE} = 
        \begin{bmatrix}
            1 \\
            1 \\
            1 \\
            1
        \end{bmatrix},
        \quad
        \bm{\alpha}_1^{\rm FE} = 
        \begin{bmatrix}
            -19.2 \\
            -16 + 0.8 \i \\
            -9.6 + 2.4 \i \\
            -6.4 + 1.6 \i
        \end{bmatrix}.
    \end{align}
    Following the same logic as above we have $\wt{s}(0, 3\pi/2) = ({\cal A}^{-1} \bm{b})_1 = 1 - 19.2 \epsilon + {\cal O}(\epsilon^2)$.
    
\end{proof}

%
\section{Derivations of $\ell \times 1$ theoretical results}
\label{app:ellx1}

The theoretical results presented in \cref{sec:LFA:ellx1-theory} for $\ell \times 1$ subdomains are rather lengthy to derive, so to aid with readability we split the calculations across several sections.
First in \cref{sec:ellx1-prob-setup} we formulate the problems to be solved, then key intermediate results are given in \cref{sec:ellx1-aux}, and then the main results from  \cref{sec:LFA:ellx1-theory} are derived in \cref{sec:ellx1-main-theory}.
%

%
\subsection{Problem simplifications and formulation}
\label{sec:ellx1-prob-setup}

We begin by writing out more concretely the linear system ${\cal A} \bm{\alpha} = [A_{ij} D (U - I) - D C]
    \bm{\alpha}
    =
    [D \bm{c} - A_{ij} D \bm{e}_{\ell}] \alpha_0
    =
    \bm{b} \alpha_0$ from \eqref{eq:SCH-1d-LFA-sys} whose first solution component is the symbol of the smoother in question, see \eqref{eq:SCH-ellx1-symbol-def}. 
In doing so, we also apply the simplifications arising from using the simplified 9-point stencils in \eqref{eq:9-point-simple-both}.
Throughout this section we use the shorthand $a := e^{\i \omega_1}$, and denote its complex conjugate as $\bar{a} = a^{-1}$.
Let us begin with the system matrix ${\cal A} = A_{ij} D (U - I) - D C$.
Recalling the matrix $D$ from \eqref{eq:SCH-1d-Aij-D}, notice that $D = \diag(1, a, a^2, \ldots, a^{\ell-1})$. 
Applying the simplified stencils in \eqref{eq:9-point} to the matrix $C$ given in \eqref{eq:SCH-1d-Rold} therefore yields $DC$ as
\begin{align}
    D C
    =
    D
    \left(
    \begin{bmatrix}
        \csw & {\cscc} & {\cse} \\
        & {\csw} & {\cscc} & {\cse} \\
        & & {\csw} & {\cscc} & {\cse} \\
        & & & \ddots & \ddots & \ddots
    \end{bmatrix}
    +
    \left( \cssw + \css + \csse \right) \bm{1} \bm{e}_1^\top
    \right) 
    =
    \begin{bmatrix}
        \sw \bar{a} & \scc & \sw a \\
                   & \sw & \scc a & \sw a^2  \\
                   &     & \sw a & \scc a^2 & \sw a^3  \\
                   &     &       & \ddots   & \ddots & \ddots 
    \end{bmatrix}
    +
    e^{-\i \omega_2} \left( 2 \snw \cos \omega_1 + \sn \right) D \bm{1} \bm{e}_1^\top.
\end{align}
Now consider the matrix $A_{ij} D (U - I)$. Multiplying a matrix from the right by the diagonal matrix $D$ scales its $k$th column by $a^{k-1}$.
Hence, recalling the matrix $A_{ij}$ given in \eqref{eq:SCH-1d-Aij-D}, we have
\begin{align} \label{eq:SCH-1d-AijD}
    A_{ij}
        =
        \begin{bmatrix}
            \scc & \se \\
            \sw & \scc & \se \\
            & \sw & \scc & \se \\
            & & \sw & \scc & \se \\
            & & & \ddots & \ddots & \ddots 
        \end{bmatrix}
        \quad
        \rightarrow
        \quad
    A_{ij} D
    &=
    \begin{bmatrix}
        \scc & \sw a \\
        \sw & \scc a & \sw a^2  \\
        & \sw a & \scc a^2 & \sw a^3  \\
        & & \sw a^2 & \scc a^3 & \sw a^4  \\
        & & & \ddots & \ddots & \ddots & 
    \end{bmatrix}.
\end{align}
Multiplying $A_{ij} D$ on the right by the difference matrix $U - I$ (see \eqref{eq:SCH-1d-LFA-sys}) scales the $k$th column of $A_{ij} D$ by $-1$ and adds to it the $(k-1)$st column, so that
\begin{align}
A_{ij} D (U - I)
&=
\begin{bmatrix}
    -\scc \, & (\scc - \sw a) \, & \sw a \hphantom{-\sw a^2}  \\
    -\sw \, & (\sw - \scc a) \, & (\scc a - \sw a^2) \, & \sw a^2  \hphantom{-\sw a^3}  \\
          & \hphantom{\sw} -\sw a \, & (\sw a - \scc a^2) \, & (\scc a^2 - \sw a^3) \, & \sw a^3  \hphantom{-\sw a^4}  \\
          &              & \hphantom{\sw a} -\sw a^2      \,   & (\sw a^2 - \scc a^3) \, & (\scc a^3 - \sw a^4) \, & \sw a^4  \\
          & & & \ddots & \ddots & \ddots & \ddots
\end{bmatrix}.
\end{align}
Further simplifying, the matrix ${\cal A} = A_{ij} D (U - I) - D C$ can now be written as
\begin{align} 
    \nonumber
    {\cal A} &=
    \begin{bmatrix}
    -\scc \, & (\scc - \sw a) \, & \sw a \hphantom{-\sw a^2}  \\
    -\sw  \, & (\sw - \scc a) \, & (\scc a - \sw a^2) \, & \sw a^2  \hphantom{-\sw a^3}  \\
          & \hphantom{\sw} -\sw a \, & (\sw a - \scc a^2) \, & (\scc a^2 - \sw a^3) \, & \sw a^3  \hphantom{-\sw a^4}  \\
          & & \ddots & \ddots & \ddots & \ddots
    \end{bmatrix}
    -
    \begin{bmatrix}
        \sw \bar{a} & \scc & \sw a \\
                   & \sw & \scc a & \sw a^2  \\
                   &     & \sw a & \scc a^2 & \sw a^3  \\
                   &     &       & \ddots   & \ddots & \ddots 
    \end{bmatrix}
    -
    e^{-\i \omega_2} \left( 2 \snw \cos \omega_1 + \sn \right) D \bm{1} \bm{e}_1^\top,
    \\
    \nonumber
    &=
    \begin{bmatrix}
    -(\scc + \sw \bar{a})  & - \sw a  \\
    -\sw  & - \scc a & - \sw a^2  \\
          & -\sw a  & - \scc a^2 & - \sw a^3 \\
          & & \ddots & \ddots & \ddots 
    \end{bmatrix}
    -
    e^{-\i \omega_2} \left( 2 \snw \cos \omega_1 + \sn \right) D \bm{1} \bm{e}_1^\top,
    \\
    \label{eq:W-and-delta}
    &=
    (W - \delta D \bm{1} \bm{e}_1^\top) D,
    \quad
    W := -
    \begin{bmatrix}
    \scc + \bar{a} \sw & \sw \\
    \sw  & \scc & \sw  &  \\
          & \sw       & \scc  & \sw  \\
          & & \ddots & \ddots & \ddots
    \end{bmatrix},
    \quad
    \delta := e^{-\i \omega_2} (2 \snw \cos \omega_1 + \sn),
\end{align}
where, notice that we have factored $D$ out from the right of ${\cal A}$, exploiting that $\bm{1} \bm{e}_1^\top = \bm{1} \bm{e}_1^\top D$, which is true since $\bm{1} \bm{e}_1^\top$ is a matrix with zeros everywhere except its first column which is $\bm{1}$ and because $D_{1,1} = 1$.
Now we consider the right-hand side vector $\bm{b} = D \bm{c} - A_{ij} D \bm{e}_{\ell}$ from \eqref{eq:SCH-1d-LFA-sys}.
The vector $D \bm{c}$ is given in \eqref{eq:SCH-1d-Rold} and upon applying the simplified stencil \eqref{eq:9-point-simple} becomes
\begin{align}
    D \bm{c}
    =
    D
    \left[  
    (\csnw + \csn + \csne) \bm{1} 
    + \cse \bm{e}_{\ell-1} 
    + (\cscc + \cse) \bm{e}_{\ell} 
    \right]
    =
    D
    \left[
    e^{\i \omega_2}(2\snw \cos \omega_1 + \sn) \bm{1} 
    + a \sw \bm{e}_{\ell-1} 
    + (\scc + a \sw) \bm{e}_{\ell} 
    \right].
\end{align}
Next, note that $A_{ij} D \bm{e}_{\ell}$ is the last column of matrix $A_{ij} D$ from \eqref{eq:SCH-1d-AijD}, which can be judiciously written as $\sw a^{\ell-1} \bm{e}_{\ell-1} + \scc a^{\ell-1} \bm{e}_{\ell} = D ( a \sw \bm{e}_{\ell-1} + \scc \bm{e}_{\ell} )$. 
As such, the right-hand side vector $\bm{b}$ can be written as
\begin{align}
    \bm{b} 
    =
    D
    \left[
    e^{\i \omega_2}(2\snw \cos \omega_1 + \sn) \bm{1} 
    + a \sw \bm{e}_{\ell-1} 
    + (\scc + a \sw) \bm{e}_{\ell} 
    \right]
    -
    D ( a \sw \bm{e}_{\ell-1} + \scc \bm{e}_{\ell} )
    =
    D
    \left[
    e^{\i \omega_2}(2\snw \cos \omega_1 + \sn) \bm{1} 
    + a \sw \bm{e}_{\ell} 
    \right].
\end{align}

Now that we a simplified representation of the linear system ${\cal A} \bm{\alpha} = \bm{b} \alpha_0$ from \eqref{eq:SCH-1d-LFA-sys}, we work toward developing an approximate solution to it. 
Our first step is to simplify the calculation by a change of variable; noting from \eqref{eq:W-and-delta} that ${\cal A}$ has a common factor of $D$, we change variables and instead solve the linear system
\begin{align} \label{eq:beta-system}
    {\cal B} \bm{\beta} = \bm{b}, 
    \quad \bm{\beta} = D \bm{\alpha},
    \quad {\cal B} = W - \delta D \bm{1} \bm{e}_1^\top.
\end{align}
Because $D$ is diagonal with $D_{1,1} = 1$ (see \eqref{eq:SCH-1d-Aij-D}), the symbol that we ultimately care about is, trivially, $\wt{s} = \alpha_1 = \beta_1$.
Next we expand the solution $\bm{\beta}$ of the linear system in \eqref{eq:beta-system} as a power series in $\epsilon$, motivated by the fact that ${\cal B}$ and $\bm{b}$ are linear functions of $\epsilon$.
Specifically, we write
$
{\cal B} = {\cal B}_0 + \epsilon {\cal B}_1,
$
$
\bm{b} = \bm{b}_0 + \epsilon \bm{b}_1,
$
and
$
\bm{\beta} = \bm{\beta}_0 + \epsilon \bm{\beta}_1 + {\cal O}(\epsilon^2).
$
Our goal is then to approximately solve for the symbol as $\wt{s} = (\bm{\beta}_0)_1 + \epsilon (\bm{\beta}_1)_1 + {\cal O}(\epsilon^2) \equiv \wt{s}_0 + \wt{s}_1 \epsilon + {\cal O}(\epsilon^2)$.
Plugging the power series solution and equating terms of like powers in $\epsilon$ we arrive at the pair of linear systems:
\begin{align} \label{eq:beta-pair-systems}
    {\cal B} \bm{\beta} = \bm{b}
    \quad
    \Longrightarrow
    \quad
    {\cal B}_0 \bm{\beta}_0 = \bm{b}_0, 
    \quad
    \textrm{and}
    \quad
    {\cal B}_0 \bm{\beta}_1 = \bm{b}_1 - {\cal B}_1 \bm{\beta}_0.
\end{align}

To better understand the structure of the linear systems in \eqref{eq:beta-pair-systems}, we now plug in the stencil entries \eqref{eq:aligned-simple-stencil}.
To this end, it is useful to first define the matrices
\begin{align} \label{eq:wt-B-wc-B}
    \wt{B} := 
    \begin{bmatrix}
    -2 + \bar{a} & 1 \\
     1  & -2 & 1  &  \\
          & 1       & -2  & 1 \\
          & & \ddots & \ddots & \ddots
    \end{bmatrix},
    \quad
    \wc{B} := 
    \begin{bmatrix}
    4 + \bar{a} & 1 \\
     1  & 4 & 1  &  \\
          & 1       & 4  & 1 \\
          & & \ddots & \ddots & \ddots
    \end{bmatrix}.
\end{align}
Now, recall from \eqref{eq:aligned-simple-stencil-FD} the FD stencil $\snw = 0$,  $\sn = -\epsilon$, $\sw = -1$, and $\scc = 2(1 + \epsilon)$, and from \eqref{eq:aligned-simple-stencil-FE} the FE stencil  $\snw = - \frac{1}{6} - \frac{1}{6} \epsilon$, $\sn = \frac{1}{3} - \frac{2}{3}\epsilon$, $\sw = -\frac{2}{3} + \frac{1}{3} \epsilon$, and $\scc = \frac{4}{3} + \frac{4}{3} \epsilon$.
Plugging in we find
\begin{align}
    \label{eq:beta-components-FD}
    &{\cal B}_0^{\rm FD} 
    = \wt{B},
    \quad
    &&{\cal B}_1^{\rm FD}
    =
    e^{-\i \omega_2} D \bm{1} \bm{e}_1^\top - 2 I,
    \quad
    &&\bm{b}_0^{\rm FD} = - a^{\ell} \bm{e}_{\ell},
    \quad
    &&\bm{b}_1^{\rm FD} = - e^{\i \omega_2} D \bm{1},
    \\
    \label{eq:beta-components-FE}
    &{\cal B}_0^{\rm FE} 
    = \frac{2}{3} [ \wt{B} + \delta_0 D \bm{1} \bm{e}_1^\top ],
    \quad
    &&{\cal B}_1^{\rm FE}
    =
    - \frac{1}{3}
    [
    \wc{B}
    + \delta_1 D \bm{1} \bm{e}_1^\top
    ],
    \quad
    &&\bm{b}_0^{\rm FE} = - \frac{2}{3} ( \bar{\delta}_0 D \bm{1} + a^{\ell} \bm{e}_{\ell} ),
    \quad
    &&\bm{b}_1^{\rm FE} = \frac{1}{3} ( \bar{\delta}_1 D \bm{1} + a^{\ell} \bm{e}_{\ell} ),
\end{align}
with $\delta_0 := \frac{1}{2} e^{-\i \omega_2} (\cos \omega_1 - 1)$, and $\delta_1 := -e^{-\i \omega_2} (\cos \omega_1 + 2)$.

Considering the linear systems in \eqref{eq:beta-pair-systems}, the key challenge in solving them is inverting the matrices ${\cal B}_0$.
In the FD case \eqref{eq:beta-components-FD}, this means inverting $\wt{B}$ in \eqref{eq:wt-B-wc-B} and in the FE case \eqref{eq:beta-components-FE}, this means inverting a rank-1 perturbation of $\wt{B}$. 
Notice also from \eqref{eq:wt-B-wc-B} that $\wt{B}$ itself is a rank-1 perturbation of a Toeplitz matrix.
In fact, the forthcoming theoretical results in \cref{sec:ellx1-main-theory} are derived by exploiting the fact that, for both the FD and FE discretization, ${\cal B}_0$ is a rank-1 perturbation of a matrix with a particularly simple Cholesky decomposition, which allows us to invert ${\cal B}_0$ in closed form via the Sherman--Morrison formula.  
Specifically, the key matrix here and its Cholesky decomposition are
\begin{align} \label{eq:B0}
    B_0 := 
    \begin{bmatrix}
    -1 & \hphantom{-}1 \\
     \hphantom{-}1  & -2 & \hphantom{-}1  &  \\
          & \hphantom{-}1       & -2  & \hphantom{-}1 \\
          & & \ddots & \ddots & \ddots
    \end{bmatrix}
    =
    -
    \underbrace{
    \begin{bmatrix}
        \hphantom{-}1 \\
        -1 & \hphantom{-}1 & \\
        & \ddots & \ddots \\
        & & -1 & \hphantom{-}1
    \end{bmatrix}
    }_{{\cal L}_0}
    \underbrace{
    \begin{bmatrix}
        \hphantom{-}1 & -1 \\
        & \ddots & \ddots \\
        & & \hphantom{-}1 & -1 \\
        & & & \hphantom{-}1
    \end{bmatrix}}_{{\cal L}_0^\top}.
\end{align}
From \eqref{eq:beta-components-FD} observe that ${\cal B}_0^{\rm FD} = B_0 + \bm{u} \bm{e}_1^\top$, with $\bm{u} = f \bm{e}_1$ and $f := -1 + \bar{a}$. 
Moreover, from \eqref{eq:beta-components-FE} observe that ${\cal B}_0^{\rm FE} = \frac{2}{3}[B_0 + \bm{u} \bm{e}_1^\top]$ with $\bm{u} = \delta_0 D \bm{1} + f \bm{e}_1$.

%
\subsection{Auxiliary theoretical results}
\label{sec:ellx1-aux}

In this section we present several key intermediate results that are required to develop the main theoretical results next in \cref{sec:ellx1-main-theory}.

\begin{lemma}[Solutions of non perturbed system] \label{lem:B0-solves} \label{lem:Bx=b-solutions}
Let $\bm{k}^p \in \mathbb{R}^{\ell}$ be a vector with elements $(\bm{k}^p)_j = j^p$, $j = 1, \ldots, \ell$.
Let $a := e^{\i \omega_1}$ and $D$ be as in \eqref{eq:SCH-1d-Aij-D}. 
Then, the following pairs of vectors $(\bm{x}, \bm{b})$ satisfy the linear system $B_0 \bm{x} = \bm{b}$ with $B_0$ from \eqref{eq:B0}:\\
\begin{subequations}
\begin{align}
    \label{eq:B0x=b-1}
    \bm{b} &= \bm{e}_1, 
    \quad
    &&\bm{x} = \bm{w} := \bm{k} - (\ell + 1) \bm{1},
    \quad
    &&w_1 = -\ell
    \\
    \label{eq:B0x=b-2}
    \bm{b} &= \bm{e}_{\ell}, 
    \quad
    &&\bm{x} = - \bm{1},
    \quad
    &&x_1 = - 1,
    \\
    \label{eq:B0x=b-3}
    \bm{b} &= \bm{1}, 
    \quad
    &&\bm{x} = c^{3,2} \bm{k}^2 + c^{3, 1} \bm{k} + c^{3, 0} \bm{1}, 
    \quad
    &&x_1 =  - \frac{1}{2} \ell (\ell + 1), 
    \\
    \label{eq:B0x=b-4}
    \bm{b} &= \bm{k}, 
    \quad
    &&\bm{x} = c^{4, 3} \bm{k}^3 + c^{4, 1} \bm{k} + c^{4, 0} \bm{1},
    \quad
    &&x_1 =  - \frac{1}{6} \ell (\ell + 1) (\ell + 2),        
    \\
    \label{eq:B0x=b-5}
    \bm{b} &= D \bm{1}, 
    \quad
    &&\bm{x} = \bm{z} := 
    \begin{cases}
        c^{5, e} D \bm{1} + c^{5, 1} \bm{k} + c^{5, 0} \bm{1}, \, &a \neq 1, \\
        c^{3,2} \bm{k}^2 + c^{3, 1} \bm{k} + c^{3, 0} \bm{1}, \, &a = 1,
    \end{cases}
    \quad
    &&z_1
    =
    \begin{cases}
        \frac{\ell(a - 1) + a(1 - a^\ell)}{(a - 1)^2}, \, &a \neq 1, \\
        - \frac{1}{2} \ell (\ell + 1), \, &a = 1,
    \end{cases}
    \end{align}
\end{subequations}
where the constants involved are 
\begin{subequations}
\begin{align}
    &c^{3,2} = \frac{1}{2}, 
    \quad
    &&c^{3, 1} = - \frac{1}{2}, 
    \quad
    &&c^{3, 0} = - \frac{1}{2} \ell (\ell + 1),
    \\
    &c^{4, 3} =  \frac{1}{6},
    \quad
    &&c^{4, 1} =  -\frac{1}{6},
    \quad
    &&c^{4, 0} = -\frac{1}{6} \ell (\ell + 1) (\ell + 2)
    \\
    &c^{5, e} = \frac{a}{(a - 1)^2},
    \quad
    &&c^{5, 1} =  - \frac{1}{a - 1},
    \quad
    &&c^{5, 0} = \frac{1}{a - 1} \left( 1 + \ell - \frac{a^{\ell+1}}{a - 1} \right).
\end{align}    
\end{subequations}

\end{lemma}

\begin{proof}

    First consider \eqref{eq:B0x=b-2}. Notice that the row sum of $B_0$ in \eqref{eq:B0} is zero in every row except the last where it is $-1$. Hence $B_0 \bm{1} = -\bm{e}_{\ell} = -\bm{b}$, so that $\bm{x} = - \bm{1}$. 

    Now we prove all remaining cases using the Cholesky decomposition of $B_0$ given in \eqref{eq:B0}.
    Consider that $\bm{b} = B_0 \bm{x} = -{\cal L}_0 {\cal L}_0^\top \bm{x} = {\cal L}_0 \bm{y}$, with $\bm{y} = -{\cal L}_0^\top \bm{x}$. Hence, we can solve $B_0 \bm{x} = \bm{b}$ in two steps: 1. Solve ${\cal L}_0 \bm{y} = \bm{b}$ for $\bm{y}$, 2. Solve ${\cal L}_0^\top \bm{x} = -\bm{y}$ for $\bm{x}$.
    Considering the structure of ${\cal L}_0$ and ${\cal L}_0^\top$ in \eqref{eq:B0}, the two systems we must solve are equivalent to the following one-step recurrence relations 
    \begin{subequations}
    \begin{align}
        \label{eq:rec-1}
        {\cal L}_0 \bm{y} = \bm{b}
        \quad
        \Longrightarrow
        \quad
        &y_1 = b_1, \quad 
        y_k - y_{k-1} = b_k, \quad 
        k = 2, \ldots, \ell 
        \\
        \label{eq:rec-2}
        {\cal L}_0^\top \bm{x} = -\bm{y}
        \quad
        \Longrightarrow
        \quad
        &x_{\ell} = -y_{\ell}, \quad 
        x_{k} - x_{k+1} = -y_k, \quad 
        k = \ell - 1, \ldots, 1.
    \end{align}
    \end{subequations}

    \underline{Case \eqref{eq:B0x=b-1} $\bm{b} = \bm{e}_1$:}
    The first system is ${\cal L}_0 \bm{y} = \bm{e}_1$; notice from \eqref{eq:B0} that $\bm{e}_1$ is just the row sums of ${\cal L}_0$, and hence $\bm{y} = \bm{1}$. 

    Now consider the second system.
    Plugging into \eqref{eq:rec-2} gives the recurrence
    $x_{\ell} = -1$, and $x_{k} - x_{k+1} = -1$, $k = \ell - 1, \ldots, 1$.
    Plug in the ansatz that $x_k = \alpha k + \beta$ for some constants $\alpha, \beta$:
    $x_{k} - x_{k+1} = [\alpha k + \beta] - [\alpha (k+1) + \beta] = - \alpha = -1$. Hence, $\alpha = 1$, so that $x_k = k + \beta$. To determine $\beta$ use the final condition: $x_{\ell} = -1 = \ell + \beta$, so that $\beta = -(\ell + 1)$. Hence, $x_k = k - (\ell + 1)$.

    \underline{Case \eqref{eq:B0x=b-3} $\bm{b} = \bm{1}$:}
    The first recurrence \eqref{eq:rec-1} is $y_1 = 1$, and $y_k - y_{k-1} = 1$, $k = 2, \ldots, \ell$.
    Plug in the ansatz that $x_k = \alpha k + \beta$ for some constants $\alpha, \beta$:
    $y_{k} - y_{k-1} = [\alpha k + \beta] - [\alpha (k-1) + \beta] = \alpha = 1$. Hence, $y_k = k + \beta$. To determine $\beta$ use the initial condition: $1 = y_{1} = 1 + \beta$, so that $\beta = 0$. Hence, $y_k = k$.

    Now consider the second system.
    Plugging into \eqref{eq:rec-2} gives the recurrence
    $x_{\ell} = -\ell$, and $x_{k} - x_{k+1} = -k$, $k = \ell - 1, \ldots, 1$.
    Plug in the ansatz that $x_k = \alpha k^2 + \beta k + \gamma$ for some constants $\alpha, \beta, \gamma$:
    $x_{k} - x_{k+1} = [\alpha k^2 + \beta k + \gamma] - [\alpha (k+1)^2 + \beta (k+1) + \gamma] = -\alpha (2k + 1) - \beta = -k$. Hence, $\alpha = \tfrac{1}{2}$, and $\beta = -\tfrac{1}{2}$, so that $x_k = \tfrac{1}{2} k^2 - \tfrac{1}{2} k + \gamma$. To determine $\gamma$ use the final condition: $-\ell = x_{\ell} = \tfrac{1}{2} \ell^2 - \tfrac{1}{2} \ell + \gamma$, so that $\gamma = -\tfrac{1}{2} \ell - \tfrac{1}{2} \ell^2 = -\tfrac{1}{2} \ell (\ell + 1)$. Hence, $x_k = \tfrac{1}{2}k^2 - \tfrac{1}{2} k -\tfrac{1}{2} \ell (\ell + 1)$.
    
    \underline{Case \eqref{eq:B0x=b-4} $\bm{b} = \bm{k}$:}
    The first recurrence \eqref{eq:rec-1} is $y_1 = 1$, and $y_k - y_{k-1} = k$, $k = 2, \ldots, \ell$.
    Plug in the ansatz that $x_k = \alpha k^2 + \beta k + \gamma$ for some constants $\alpha, \beta, \gamma$:
    $y_{k} - y_{k-1} = [\alpha k^2 + \beta k + \gamma] - [\alpha (k-1)^2 + \beta (k-1) + \gamma] = (2k - 1) \alpha + \beta = k$. Hence, $\alpha = \tfrac{1}{2}$ and $\beta = \tfrac{1}{2}$, so that $y_k = \tfrac{1}{2} k^2 + \tfrac{1}{2} k + \gamma$. To determine $\gamma$ use the initial condition: $1 = y_{1} = 1 + \gamma$, so that $\gamma = 0$. Hence, $y_k = \tfrac{1}{2} k (k + 1)$.

    Now consider the second system.
    Plugging into \eqref{eq:rec-2} gives the recurrence
    $x_{\ell} = -\tfrac{1}{2} \ell (\ell + 1)$, and $x_{k} - x_{k+1} = - \tfrac{1}{2} k (k + 1)$, $k = \ell - 1, \ldots, 1$.
    Plug in the ansatz that $x_k = \alpha k^3 + \beta k^2 + \gamma k + \delta$ for some constants $\alpha, \beta, \gamma, \delta$:
    $x_{k} - x_{k+1} = [\alpha k^3 + \beta k^2 + \gamma k + \delta] - [\alpha (k+1)^3 + \beta (k+1)^2 + \gamma (k+1) + \delta] 
    =
    -3 \alpha k^2 - (3 \alpha + 2 \beta) k - (\alpha + \beta + \gamma)
    =
    - \tfrac{1}{2} k^2 - \tfrac{1}{2} k$.
    Hence, $\alpha = \frac{1}{6}$, $\beta = 0$, and $\gamma = -\frac{1}{6}$, so that $x_k = \tfrac{1}{6} k (k^2 - 1) + \delta =  \tfrac{1}{6} k (k + 1) (k - 1) + \delta$.
    To determine $\delta$ use the final condition: $-\tfrac{1}{2} \ell (\ell + 1) = x_{\ell} = \tfrac{1}{6} \ell (\ell + 1) (\ell - 1) + \delta$, so that $\delta = \ell(\ell + 1) [ -\frac{1}{2} - \tfrac{1}{6} (\ell - 1 )  ] = -\frac{1}{6} \ell(\ell + 1)( \ell + 2 )$.
    Hence, $x_k = \tfrac{1}{6} k^3 - \tfrac{1}{6} k  -\frac{1}{6} \ell(\ell + 1)( \ell + 2 )$.

    \underline{Case \eqref{eq:B0x=b-5} $\bm{b} = D \bm{1}$:}
    Notice from \eqref{eq:SCH-1d-Aij-D} that when $a = 1$ we have $D = I$ so that $\bm{b} = D \bm{1} = \bm{1}$, so that the system is the same as case \eqref{eq:B0x=b-3}.
    Hence, moving forward, assume $a \neq 1$.

    Notice that the right-hand side is $b_k = (D \bm{1})_k = D_{kk} = a^{k-1}$. 
    The first recurrence \eqref{eq:rec-1} is $y_1 = 1$, and $y_k - y_{k-1} = a^{k-1}$, $k = 2, \ldots, \ell$.
    Plug in the ansatz that $y_k = \alpha a^{k - \beta} + \gamma$ for some constants $\alpha, \beta, \gamma$:
    $y_{k} - y_{k-1} = [\alpha a^{k - \beta} + \gamma] - [\alpha a^{k-1 - \beta} + \gamma] = \alpha (1 - a^{-1}) a^{k-\beta} =  a^{k-1}$. 
    Hence $\alpha = \tfrac{1}{1 - a^{-1}} = \tfrac{a}{a - 1}$, and $\beta = -1$, so that $y_k = \tfrac{a^k}{a - 1} + \gamma$.
    To determine $\gamma$ use the initial condition: $1 = y_1 = \tfrac{a}{1 - a} + \gamma$, so that $\gamma = 1 - \tfrac{a}{a - 1} = -\tfrac{1}{a - 1}$. Hence, $y_k = \tfrac{a^k - 1}{a - 1}$.

    Now consider the second system, ${\cal L}_0^\top \bm{x} = - \bm{y} = -\tfrac{a^{\bm{k}} - \bm{1}}{a - 1} = -\tfrac{a}{a - 1} (D \bm{1}) +\tfrac{1}{a - 1} \bm{1}$.
    Since this system is linear, let us decompose the solution as $\bm{x} = \tfrac{a}{a - 1} \wh{\bm{x}} - \tfrac{1}{a - 1}  \wc{\bm{x}}$ where ${\cal L}_0^\top \wh{\bm{x}} = -D \bm{1}$, and ${\cal L}_0^\top \wc{\bm{x}} = -\bm{1}$.
    Notice that ${\cal L}_0^\top \wc{\bm{x}} = -\bm{1}$ is the same as the second system in considered in case \eqref{eq:B0x=b-1} $\bm{b} = \bm{e}_1$ above, so that $\wc{x}_k = k - (\ell + 1)$.

    Now consider ${\cal L}_0^\top \wh{\bm{x}} = -D \bm{1}$. Plugging into the second recurrence \eqref{eq:rec-2} gives
    $\wh{x}_{\ell} = -a^{\ell - 1}$, and $\wh{x}_k - \wh{x}_{k+1} = - a^{k - 1}$, $k = \ell - 1, \ldots, 1$.
    Plug in the ansatz that $\wh{x}_k = \alpha a^{k - \beta} + \gamma$ for some constants $\alpha, \beta, \gamma$:
    $\wh{x}_k - \wh{x}_{k+1} = [\alpha a^{k - \beta} + \gamma] - [\alpha a^{k+1 - \beta} + \gamma] = \alpha (1 - a) a^{k-\beta}  = -a^{k-1}$. 
    Hence $\alpha = \tfrac{1}{a - 1}$, and $\beta = 1$, so that $\wh{x}_k = \tfrac{a^{k-1}}{a - 1} + \gamma$.
    To determine $\gamma$ use the final condition: $-a^{\ell - 1} = \wh{x}_{\ell} = \tfrac{a^{\ell-1}}{a - 1} + \gamma$, so that
    $\gamma = -a^{\ell - 1} - \tfrac{a^{\ell-1}}{a - 1}
    =
    -\frac{a^{\ell}}{a - 1}
    $.
    Hence, $\wh{x}_k = \tfrac{1}{a - 1}[a^{k-1} - a^{\ell}] = \frac{1}{a - 1}[ (D \bm{1})_k - a^{\ell} ]$.
    Plugging into the above we have
    $x_k 
    = 
    \tfrac{a}{a - 1} \wh{x}_k - \tfrac{1}{a - 1} \wc{x}_k 
    =
    \tfrac{a}{a - 1} \frac{1}{a - 1}[ (D \bm{1})_k - a^{\ell} ]
    -
    \tfrac{1}{a - 1} [k - (\ell + 1)]
    =
    \tfrac{a}{(a - 1)^2} (D \bm{1})_k 
    - 
    \tfrac{1}{a - 1} k
    +
    \tfrac{1}{a - 1} [ 1 + \ell - \tfrac{a^{\ell+1}}{a - 1} ]
    $.

\end{proof}

\begin{lemma}[Solution of perturbed system] \label{lem:B0-pert-solve}
    Consider the linear system
    \begin{align} \label{eq:pert-lin-sys}
        \left[B_0 + (f \bm{e}_1 + \delta_0 D \bm{1}) \bm{e}_1^\top \right] 
        \bm{\eta} = \bm{\xi},
    \end{align}
    with matrix $B_0$ given in \eqref{eq:B0}, matrix $D$ given in \eqref{eq:SCH-1d-Aij-D}, and $f, \delta_0 \in \mathbb{C}$ constants such that $1 - \ell f + \delta_0 z_1 \neq 0$.
    Then, the solution of this system satisfies
    \begin{align}
        \bm{\eta}
        =
        B_0^{-1} \bm{\xi}
        + 
        \Delta ({\bm{\xi}}) 
        [
        f \bm{w} + \delta_0 \bm{z}
        ],
        \quad
        \Delta ({\bm{\xi}})
        =
        -
        \frac{(B_0^{-1} \bm{\xi})_1}{1 - \ell f + \delta_0 z_1},
    \end{align}
    with the vectors $\bm{w}$ and $\bm{z}$ defined in \cref{lem:Bx=b-solutions}.
\end{lemma}

\begin{proof}
    Recall the Sherman--Morrison formula (see e.g., \cite[Proposition 7.3.1]{Kelley1995}) tells us that the inverse of a non-singular matrix $\wt{Z} = Z + \bm{u} \bm{v}^\top$ can be expressed in terms of that of $Z$ (assumed non-singular) according to
    \begin{align*}
        \wt{Z}^{-1} = Z^{-1} \left( I - \frac{\bm{u} \bm{v}^\top Z^{-1}}{1 + \bm{v}^\top Z^{-1} \bm{u}} \right),
    \end{align*}
    provided that $1 + \bm{v}^\top Z^{-1} \bm{u} \neq 0$.
    Applying this to the matrix in \eqref{eq:pert-lin-sys} we may express its inverse as
    \begin{align*}
        \left[B_0 + (f \bm{e}_1 + \delta_0 D \bm{1}) \bm{e}_1^\top \right]^{-1} 
        = 
        B_0^{-1} \left( I - \frac{ (f \bm{e}_1 + \delta_0 D \bm{1}) \bm{e}_1^\top B_0^{-1}}{1 + \bm{e}_1^\top B_0^{-1} (f \bm{e}_1 + \delta_0 D \bm{1})} \right).
    \end{align*}
    Applying cases 1 and 5 from \cref{lem:Bx=b-solutions} we have that $\bm{e}_1^\top B_0^{-1} (f \bm{e}_1 + \delta_0 D \bm{1}) = f [\bm{e}_1^\top (B_0^{-1} \bm{e}_1)] + \delta_0 [\bm{e}_1^\top (B_0^{-1} D \bm{1})] = f [B_0^{-1} \bm{e}_1]_1 + \delta_0 [B_0^{-1} D \bm{1}]_1 = f[\bm{w}]_1 + \delta_0 [\bm{z}]_1 =  -\ell f + \delta_0 z_1$.
    Plugging into the above formula for the inverse and applying it to a vector $\bm{\xi}$ gives
    \begin{align*}
        \left[B_0 + (f \bm{e}_1 + \delta_0 D \bm{1}) \bm{e}_1^\top \right]^{-1} 
        \bm{\xi}
        &= 
        B_0^{-1} \left( I - \frac{ (f \bm{e}_1 + \delta D \bm{1})  \bm{e}_1^\top B_0^{-1} }{1 -\ell f + \delta_0 z_1} \right) \bm{\xi},
        \\
        &=
        B_0^{-1} \bm{\xi} - \frac{ (B_0^{-1} \bm{\xi})_1 }{1 -\ell f + \delta_0 z_1} B_0^{-1} (f \bm{e}_1 + \delta_0 D \bm{1}),
        \\
        &=B_0^{-1} \bm{\xi} - \frac{ (B_0^{-1} \bm{\xi})_1 }{1 -\ell f + \delta_0 z_1}  (f \bm{w} + \delta_0 \bm{z}),
    \end{align*}
    where we note that $\bm{e}_1^\top B_0^{-1} \bm{\xi} = (B_0^{-1} \bm{\xi})_1$, and the last equality  following from cases 1 and 5 of \cref{lem:Bx=b-solutions}.

\end{proof}

\begin{lemma}[Critical point of asymptotic expansion] \label{lem:crit-point-asymptotic}
    Let $\epsilon$ be a small parameter independent of $x \in \wc{\mathbb{R}} \subseteq \mathbb{R}$. Consider functions $g, g_0, g_1 \colon \wc{\mathbb{R}} \to \mathbb{R}$, and suppose that their derivatives $g', g_0', g_0'', g_0''', g_1', g_1''$ all exist and are bounded on $\wc{\mathbb{R}}$.
    Let these functions be related as
    \begin{align}
        g(x) = g_0(x) + \epsilon g_1 (x) + {\cal O}(\epsilon^2).
    \end{align}
    Let $x^* = x_0^* + \epsilon x_1^* + {\cal O}(\epsilon^2)$ be an arbitrary critical point of $g$, i.e., $g'(x^*) = 0$. Then 
    \begin{align} \label{eq:g-crit-eval}
        g( x^* ) = g_0(x_0^*) + \epsilon g_1 (x_0^*) + {\cal O}(\epsilon^2),
    \end{align}
    and $x_0^*$ is a critical point of $g_0$, i.e., $g_0'(x_0^*) = 0$.
\end{lemma}

\begin{proof}
Evaluating the derivative of $g$ at the critical point gives
\begin{align*}
    0 = g'(x^*)
    &=
    g_0'(x^*) + \epsilon g_1' (x^*) + {\cal O}(\epsilon^2),
    \\
    &=
    g_0'(x_0^* + \epsilon x_1^* + {\cal O}(\epsilon^2)) + \epsilon g_1' (x_0^* + \epsilon x_1^* + {\cal O}(\epsilon^2)) + {\cal O}(\epsilon^2),
    \\
    &=
    g_0'(x_0^*) + \epsilon 
    \left[
        x_1^* g_0''(x_0^*) + g_1' (x_0^*)
    \right]
    + 
    {\cal O}(\epsilon^2),
\end{align*}
with the last equality following from the Taylor expansion $q(x + \delta) = q(x) + \delta q'(x) + {\cal O}(\delta^2)$ for small $\delta$ and continuously differentiable $q$ with bounded derivative.
Setting this last expression equal to zero and equating powers of $\epsilon$ we get the system of equations:
    $g_0'(x_0^*) = 0$ and
    $x_1^* g_0''(x_0^*) + g_1' (x_0^*) = 0$.
Evidently, $x_0^*$ is a critical point of $g_0$ as claimed.
Evaluating $g$ at the critical point gives
\begin{align*}
    g(x^*)
    &=
    g_0(x^*) + \epsilon g_1 (x^*) + {\cal O}(\epsilon^2),
    \\
    &=
    g_0(x_0^* + \epsilon x_1^* + {\cal O}(\epsilon^2)) + \epsilon g_1 (x_0^* + \epsilon x_1^* + {\cal O}(\epsilon^2)) + {\cal O}(\epsilon^2),
    \\
    &=
    g_0(x_0^*) + \epsilon 
    \left[
        x_1^* g_0'(x_0^*) + g_1 (x_0^*)
    \right]
    + 
    {\cal O}(\epsilon^2),
\end{align*}
with the last equality again following from Taylor expansion.
Substituting $g'(x_0^*) = 0$ gives \eqref{eq:g-crit-eval}.

\end{proof}

%
\subsection{Main theoretical results}
\label{sec:ellx1-main-theory}

Now we prove our main theoretical results presented in \cref{sec:LFA:ellx1-theory}.
To help with readability we repeat the theorem statements from \cref{sec:LFA:ellx1-theory}.
Since the results are rather technical to derive, we provide some supporting numerical evidence for their correctness in \cref{fig:LFA-theory-support}.

\begin{theorem}[Copy of \cref{thm:1d-symbol-FD}. Linearized symbol: FD] \label{thm:1d-symbol-FD-copy}
    Consider the anisotropic diffusion equation \eqref{eq:aligned} with anisotropy ratio $\epsilon \in [0,1]$ discretized with FDs \eqref{eq:rot-fd}.
    Let $\wt{s}_{\ell,1}$ be the symbol for maximally overlapping multiplicative Schwarz with $\ell \times 1$ subdomains, $\ell \in \mathbb{N}$ (see \cref{sec:LFA:ellx1}).
    Then,
    \begin{align}
        \wt{s}_{\ell,1}(\omega_1, \omega_2) = \wt{s}_0(\omega_1) + \wt{s}_{1} (\omega_1, \omega_2) \epsilon + {\cal O}(\epsilon^2),
    \end{align}
    where 
    \begin{subequations}
    \begin{align}
        \wt{s}_0(\omega_1) 
            &= 
            \frac{a^{\ell}}{1 + \ell - \ell \bar{a}}, 
            \\
        \quad
        \wt{s}_1(\omega_1, \omega_2)
            &=
        -
        a^{-\ell} 
        \wt{s}_0
        \left[
            \frac{1}{3} \ell (\ell + 1) [3 \bar{a} + (1 - \bar{a}) (\ell + 2)] \wt{s}_0
            +
            (e^{\i \omega_2} + \wt{s}_0 e^{-\i \omega_2}) z_1(a)
        \right],
    \end{align}
    \end{subequations}
    with $a = e^{\i \omega_1}$ and $z_1(a)$ defined in \cref{lem:Bx=b-solutions}.
\end{theorem}

\begin{proof}
    Recall from \cref{sec:ellx1-prob-setup} that the symbol is given by $\wt{s}_{\ell,1} = (\bm{\beta}_0)_1 + \epsilon (\bm{\beta}_1)_1 + {\cal O}(\epsilon^2) = \wt{s}_0 + \wt{s}_1 \epsilon + {\cal O}(\epsilon^2)$ with vectors $\bm{\beta}_0$ and $\bm{\beta}_1$ satisfying the pair of linear systems in \eqref{eq:beta-pair-systems}.
    Recalling the components from \eqref{eq:beta-components-FD}, the linear system governing $\bm{\beta}_0$, i.e.,
    ${\cal B}_0 \bm{\beta}_0 = \bm{b}_0$, 
    is
    $[B_0 + f \bm{e}_1 \bm{e}_1^\top] \bm{\beta}_0 = -a^{\ell} \bm{e}_{\ell}$, with $f = -1 + \bar{a}$. 
    The solution of this system may be expressed as
    \begin{align*}
        \bm{\beta}_0
        =
        -
        a^{\ell} \left( B_0 + f \bm{e}_1 \bm{e}_1^\top \right)^{-1} \bm{e}_{\ell}.
    \end{align*}
    Applying \cref{lem:B0-pert-solve} to evaluate the inverse here (setting constant $\delta_0 = 0$ in \cref{lem:B0-pert-solve}) yields
    \begin{align*}
        \bm{\beta}_0 
        = 
        -a^{\ell} \left( B_0^{-1} \bm{e}_{\ell} + f \Delta(\bm{e}_{\ell}) \bm{w} \right)
        = 
        -a^{\ell} \left( -\bm{1} + f \Delta(\bm{e}_{\ell}) \bm{w} \right) ,
    \end{align*}
    with $B_0^{-1} \bm{e}_{\ell} = -\bm{1}$ being case 2 of \cref{lem:Bx=b-solutions}, and the constant $\Delta ({\bm{\xi}})
        =
        -
        \frac{(B_0^{-1} \bm{\xi})_1}{1 - \ell f}$ for vector $\bm{\xi}$ being introduced in \cref{lem:B0-pert-solve}.
    Applying again case 2 of \cref{lem:Bx=b-solutions} this constant evaluates to
    $\Delta(\bm{e}_{\ell}) = \frac{-(B_0^{-1} \bm{e}_{\ell})_1}{1 - \ell f} = \frac{1}{1 - \ell f}$.
    Plugging in this constant and that $\bm{w} = \bm{k} - (\ell + 1) \bm{1}$ from case 1 of \cref{lem:Bx=b-solutions} we have
    \begin{align} \label{eq:beta0-FD}
        \bm{\beta}_0 
        = -a^{\ell} \left( -\bm{1} + \frac{f}{1 - \ell f} [\bm{k} - (\ell + 1) \bm{1}] \right)
        =
        \frac{a^{\ell}}{1 - \ell f} [(1 + f) \bm{1} - f\bm{k}].
    \end{align}
    Plugging in $f = -1 + \bar{a}$, and recalling from \cref{lem:B0-solves} that $(\bm{k}_1) = 1$ gives $\wt{s}_0 = (\bm{\beta}_0)_1 = \frac{a^{\ell}}{1 - \ell f}$ as claimed.

    Having solved for $\bm{\beta}_0$, we now solve for $\bm{\beta}_1$ using the second linear system in \eqref{eq:beta-pair-systems}: ${\cal B}_0 \bm{\beta}_1 = \bm{b}_1 - {\cal B}_1 \bm{\beta}_0$.
    First let us simplify the right-hand side. Recall from \eqref{eq:beta-components-FD} that 
    $
    {\cal B}_1
    =
    e^{-\i \omega_2} D \bm{1} \bm{e}_1^\top - 2 I$,
    and
    $
    \bm{b}_1 = - e^{\i \omega_2} D \bm{1}
    $.
    So, multiplying $\bm{\beta}_0$ from \eqref{eq:beta0-FD} we have the right-hand side vector
    \begin{align*}
        \bm{b}_1 - {\cal B}_1 \bm{\beta}_0
        &=
        - e^{\i \omega_2} D \bm{1}
        -
        \frac{a^{\ell}}{1 - \ell f} 
        (e^{-\i \omega_2} D \bm{1} \bm{e}_1^\top - 2 I)
        [(1 + f) \bm{1} - f\bm{k}],
        \\
        &=
        - e^{\i \omega_2} D \bm{1}
        -
        \wt{s}_0
        (\bm{e}_1^\top  [(1 + f) \bm{1} - f\bm{k}]) e^{-\i \omega_2} D \bm{1}
        -
        2 \wt{s}_0
        [ f \bm{k} - (1 + f) \bm{1} ]
        ,
        \\
        &=
        -(e^{\i \omega_2} + \wt{s}_0 e^{-\i \omega_2}) D \bm{1}
        -
        2 \wt{s}_0 f \bm{k}
        +
        2 \wt{s}_0 (1 + f) \bm{1},
        \\
        &
        \equiv
        c_e D \bm{1} + c_k \bm{k} + c_1 \bm{1},
    \end{align*}
    where we have used that $\bm{e}_1^\top [(1 + f) \bm{1} - f\bm{k}] = [(1 + f) \bm{1} - f\bm{k}]_1 = 1$, and we have introduced the constants $c_e$, $c_k$, $c_1$.
    Now applying \cref{lem:B0-pert-solve} to the solve the system ${\cal B}_0 \bm{\beta}_1 = \left( B_0 + f \bm{e}_1 \bm{e}_1^\top \right) \bm{\beta}_1 = \bm{b}_1 - {\cal B}_1 \bm{\beta}_0$ for $\bm{\beta}_1$ we find
    \begin{align*}
        \bm{\beta}_1
        &=
        c_e \left(B_0 + f \bm{e}_1 \bm{e}_1^\top \right)^{-1} D \bm{1}
        +
        c_k \left(B_0 + f \bm{e}_1 \bm{e}_1^\top \right)^{-1} \bm{k}
        +
        c_1 \left(B_0 + f \bm{e}_1 \bm{e}_1^\top \right)^{-1} \bm{1},
        \\
        &=
        c_e \left(B_0^{-1} (D \bm{1}) + f \Delta(D \bm{1}) \bm{w} \right)
        +
        c_k \left(B_0^{-1} \bm{k} + f \Delta(\bm{k}) \bm{w} \right)
        +
        c_1 \left(B_0^{-1} \bm{1} + f \Delta(\bm{1}) \bm{w} \right).
    \end{align*}
    Now we just consider the first element in the vector $\bm{\beta}_1$ because $\wt{s}_1 = (\bm{\beta}_1)_1$. 
    Recall from \cref{lem:Bx=b-solutions} that $(\bm{w})_1 = -\ell$ and from \cref{lem:B0-pert-solve} the constant $\Delta ({\bm{\xi}})
        =
        -
        \frac{(B_0^{-1} \bm{\xi})_1}{1 - \ell f}$.
    Hence, for any vector $\bm{\xi}$ we have that $\left(B_0^{-1} \bm{\xi} + f \Delta( \bm{\xi} ) \bm{w} \right)_1 = (B_0^{-1} \bm{\xi})_1 + \frac{\ell f}{1 - \ell f} (B_0^{-1} \bm{\xi})_1 = \frac{(B_0^{-1} \bm{\xi})_1}{1 - \ell f}$.
    Plugging this result into the above expression for $\bm{\beta}_1$ yields
    \begin{align*}
        (\bm{\beta}_1)_1
        &=
        \frac{1}{1 - \ell f}
        \left[
        c_e (B_0^{-1} D \bm{1})_1 
        +
        c_k (B_0^{-1} \bm{k})_1 
        +
        c_1 (B_0^{-1} \bm{1})_1 \right],
        \\
        &=
        \frac{1}{1 - \ell f}
        \left[
            c_e z_1
            -
            \frac{1}{6} \ell (\ell + 1) (\ell + 2) c_k
            -
            \frac{1}{2} \ell (\ell + 1) c_1
        \right],
        \\
        &=
        -
        a^{-\ell} 
        \wt{s}_0
        \left[
            (e^{\i \omega_2} + \wt{s}_0 e^{-\i \omega_2}) z_1
            +
            \frac{1}{6} \ell (\ell + 1) [ 3 c_1 + (\ell + 2) c_k] 
        \right],
        \\
        &=
        -
        a^{-\ell} 
        \wt{s}_0
        \left[
            (e^{\i \omega_2} + \wt{s}_0 e^{-\i \omega_2}) z_1
            +
            \frac{1}{3} \ell (\ell + 1) [3(1+f) -f(\ell + 2)] \wt{s}_0
        \right],
    \end{align*}
    with the second equality following from applying \cref{lem:Bx=b-solutions} to evaluate $(B_0^{-1} D \bm{1})_1$, $(B_0^{-1} \bm{k})_1$, $(B_0^{-1} \bm{1})_1$.
    Plugging in $f = -1 + \bar{a}$ completes the proof. 

\end{proof}

\begin{theorem}[Copy of \cref{thm:1d-symbol-FE}. Linearized symbol: FE] \label{thm:1d-symbol-FE-copy}
    Consider the anisotropic diffusion equation \eqref{eq:aligned} with anisotropy ratio $\epsilon \in [0,1]$ discretized with FEs \eqref{eq:rot-fe}.
    Let $\wt{s}_{\ell,1}$ be the symbol for maximally overlapping multiplicative Schwarz with $\ell \times 1$ subdomains, $\ell \in \mathbb{N}$ (see \cref{sec:LFA:ellx1}).
    Then,
    \begin{align}
        \wt{s}_{\ell,1}(\omega_1, \omega_2) = \wt{s}_0(\omega_1, \omega_2) + \wt{s}_{1}(\omega_1, \omega_2) \epsilon  + {\cal O}(\epsilon^2),
    \end{align}
    where 
    \begin{align} \label{eq:FE-s0-copy}
         \wt{s}_0(\omega_1, \omega_2) = \frac{a^{\ell} - \bar{\delta}_0  z_1(a)}{1 + \ell - \ell \bar{a} + \delta_0 z_1(a)}, 
    \end{align}
    with $a = e^{\i \omega_1}$, $\delta_0 = \frac{1}{2} e^{-\i \omega_2} (\cos \omega_1 - 1)$, $z_1(a)$ defined in \cref{lem:B0-solves},
    and
    $\wt{s}_1(\omega_1, \omega_2)$ a function of $\omega_1$ and $\omega_2$ satisfying $\wt{s}_{1}(0, \omega_2) = - \frac{3}{2} \ell (\ell + 1)(1 - e^{-i \omega_2})$.
\end{theorem}

\begin{proof}
    This proof proceeds similarly to that of \cref{thm:1d-symbol-FD}.
    Recall from \cref{sec:ellx1-prob-setup} that the symbol is given by $\wt{s} = (\bm{\beta}_0)_1 + \epsilon (\bm{\beta}_1)_1 + {\cal O}(\epsilon^2) = \wt{s}_0 + \wt{s}_1 \epsilon + {\cal O}(\epsilon^2)$ with vectors $\bm{\beta}_0$ and $\bm{\beta}_1$ satisfying the pair of linear systems in \eqref{eq:beta-pair-systems}.
    Recalling the components from \eqref{eq:beta-components-FE}, the linear system governing $\bm{\beta}_0$, i.e.,
    ${\cal B}_0 \bm{\beta}_0 = \bm{b}_0$, 
    is
    $\frac{2}{3} \left( B_0 + \bm{u} \bm{e}_1^\top \right) \bm{\beta}_0 = -\frac{2}{3} ( \bar{\delta}_0 D \bm{1} + a^{\ell} \bm{e}_{\ell} )$, where $\bm{u} = f \bm{e}_1 + \delta_0 D \bm{1}$.
    The solution of this system can be expressed as
    \begin{align} \label{eq:beta0-FE}
        \bm{\beta}_0
        =
        -
        a^{\ell} \left( B_0 + \bm{u} \bm{e}_1^\top \right)^{-1} \bm{e}_{\ell}
        -
        \bar{\delta}_0 \left(B_0 + \bm{u} \bm{e}_1^\top \right)^{-1} D \bm{1}.
    \end{align}
    Applying \cref{lem:B0-pert-solve} to evaluate these two inverses we have
    \begin{align*}
        \bm{\beta}_0
        = 
        -a^{\ell} 
        \left[ (B_0^{-1}\bm{e}_{\ell}) + \Delta(\bm{e}_{\ell})(f \bm{w} + \delta_0 \bm{z}) \right]
        -\bar{\delta}_0 
        \left[ (B_0^{-1} D\bm{1}) + \Delta(D\bm{1})(f \bm{w} + \delta_0 \bm{z}) \right],
    \end{align*}
    where recall the constant $\Delta(\bm{\xi}) = -\frac{(B_0^{-1} \bm{\xi})_1}{1 - \ell f + \delta_0 z_1}$ for vector $\bm{\xi}$.
    Now consider the first element of $\bm{\beta}_0$ because $\wt{s}_0 = (\bm{\beta}_0)_1$.
    Recalling from \cref{lem:Bx=b-solutions} that $(\bm{w})_1 = - \ell$, we find
    \begin{align*}
        (\bm{\beta}_0)_1
            &= 
        -a^{\ell} (B_0^{-1}\bm{e}_{\ell})_1 
        \left[ 1 - \frac{-\ell f + \delta_0 z_1 }{1 - \ell f + \delta_0 z_1}
        \right]
        -\bar{\delta}_0
        (B_0^{-1} D\bm{1})_1
        \left[ 1 - \frac{-\ell f + \delta_0 z_1 }{1 - \ell f + \delta_0 z_1}
        \right],
        \\
        &= 
        \frac{-1}{1 - \ell f + \delta_0 z_1}
        \left[
        a^{\ell} (B_0^{-1}\bm{e}_{\ell})_1 
        +
        \bar{\delta}_0
        (B_0^{-1} D\bm{1})_1 
        \right],
        \\
        &=
        \frac{a^{\ell} -\bar{\delta}_0 z_1}{1 - \ell f + \delta_0 z_1},
    \end{align*}
    where the last equality follows by using \cref{lem:Bx=b-solutions} to evaluate $(B_0^{-1}\bm{e}_{\ell})_1 = -1$ and $(B_0^{-1} D\bm{1})_1 = z_1$.
    This completes the first part of the proof. 

    Now we evaluate $\wt{s}_{1}(0, \omega_2) = ( \bm{\beta}_1 )_1(0, \omega_2)$, where recall from \eqref{eq:beta-pair-systems} that $\bm{\beta}_1$ solves the linear system ${\cal B}_0 \bm{\beta}_1 = \bm{b}_1 - {\cal B}_1 \bm{\beta}_0$, with the terms in this system given in \eqref{eq:beta-components-FE}. 
    In the remainder of the proof we take $\omega_1 = 0$, but for notional simplicity we do not write this explicitly. 
    When $\omega_1 = 0$, we have: $a = e^{-i \omega_1} = 1$, $f = -1 + \bar{a} = 0$, $D \bm{1} = \bm{1}$, and from \eqref{eq:beta-components-FE} $\delta_0 = 0$, $\delta_1 = -3 e^{-i \omega_2}$. 
    Plugging these into \eqref{eq:beta-components-FE} we find, when $\omega_1 = 0$,
    ${\cal B}_0 = \frac{2}{3} B_0$, $\bm{b}_0 = - \frac{2}{3} \bm{e}_{\ell}$, ${\cal B}_1 = - \frac{1}{3}[\wc{B} - 3 e^{-i \omega_2} \bm{1} \bm{e}_1^\top ]$, $\bm{b}_1 = \frac{1}{3}( - 3 e^{-i \omega_2} \bm{1} + \bm{e}_{\ell} )$, with $\wc{B}$ defined in \eqref{eq:wt-B-wc-B}.
    As such, when $\omega_1 = 0$ the linear system governing $\bm{\beta}_0$ is simply $\frac{2}{3} B_0 \bm{\beta}_0 = -\frac{2}{3} \bm{e}_{\ell}$, so that from case 2 of \cref{lem:B0-solves} we have $\bm{\beta}_0 = \bm{1}$.
    Therefore, the right-hand side of the linear system ${\cal B}_0 \bm{\beta}_1 = \bm{b}_1 - {\cal B}_1 \bm{\beta}_0$ simplifies as 
    \begin{align*}
        \bm{b}_1 - {\cal B}_1 \bm{\beta}_0
        &=
        \frac{1}{3}( - 3 e^{-i \omega_2} \bm{1} + \bm{e}_{\ell} )
        +
        \frac{1}{3}(\wc{B} - 3 e^{-i \omega_2} \bm{1} \bm{e}_1^\top ) \bm{1},
        \\
        &=
        \frac{1}{3}\left[ - 6 e^{-i \omega_2} \bm{1} + \bm{e}_{\ell}
        +
        (6\bm{1} - \bm{e}_{\ell})
        \right]
        =
        2 (1 - e^{-i \omega_2}) \bm{1},
    \end{align*}
    noting that $\wc{B} \bm{1} = 6\bm{1} - \bm{e}_{\ell}$, because this is just the row sums of $\wc{B}$ in \eqref{eq:wt-B-wc-B} which, when $a = 1$, is 6 in all but the last row, where it is 5.
    As such, the linear system ${\cal B}_0 \bm{\beta}_1 = \bm{b}_1 - {\cal B}_1 \bm{\beta}_0$ becomes $\frac{2}{3} B_0 \bm{\beta}_1 = 2 (1 - e^{-i \omega_2}) \bm{1}$.
    Hence, the first component of its solution is given by
    \begin{align*}
        \wt{s}_{1} 
        = 
        (\bm{\beta}_1)_1 
        = 
        3 (1 - e^{-i \omega_2}) (B_0^{-1} \bm{1})_1
        =
        -\frac{3}{2} \ell (\ell + 1) (1 - e^{-i \omega_2}),
    \end{align*}
    with the last equality following from case 3 of \cref{lem:B0-solves}.
    
\end{proof}

\begin{theorem}[Copy of \cref{thm:1d-smooth}. Linearized smoothing factor] \label{thm:1d-smooth-copy}
    Consider the anisotropic diffusion equation \eqref{eq:aligned} with anisotropy ratio $\epsilon \in [0,1]$ discretized with either FDs \eqref{eq:rot-fd} or FEs \eqref{eq:rot-fe}.
    Let $\mu_{\ell,1}(\epsilon)$ be the smoothing factor (see \cref{def:mu}) for maximally overlapping multiplicative Schwarz with $\ell \times 1$ subdomains, $\ell \in \mathbb{N}$ (see \cref{sec:LFA:ellx1}).
    Further, suppose that the symbols in \cref{thm:1d-symbol-FD,thm:1d-symbol-FE} are sufficiently smooth functions of $\omega_1$ and $\omega_2$.\footnote{We have no reason to believe that these functions do not meet the smoothness conditions, but we do not explicitly verify due the complexity or non-analytic form of the functions.}
    Then, the smoothing factor for the FD discretization is
    \begin{align} \label{eq:1d-smooth-FD-copy}
        \mu_{\ell, 1}(\epsilon) 
        = 
        1 - \ell (\ell + 1) \epsilon
    + {\cal O}(\epsilon^2),
    \end{align}
    and the smoothing factor for the FE discretization is bounded from below as\footnote{We believe that this lower bound in fact holds with equality; that is, $\mu_{\ell, 1}(\epsilon)
        =
        1 - \frac{3}{2} \ell (\ell + 1) \epsilon
    + {\cal O}(\epsilon^2)$. However we have not been able to prove that this is the case. See discussion in \cref{rem:smooth-FE-equality}.}
    \begin{align} \label{eq:1d-smooth-FE-copy}
    \mu_{\ell, 1}(\epsilon)
    \geq
        1 - \frac{3}{2} \ell (\ell + 1) \epsilon
    + {\cal O}(\epsilon^2).
    \end{align} 
\end{theorem}

\begin{proof}

While we have previously used a bar to denote a complex conjugate, for readability, in this proof we use also a superscript dagger to denote the complex conjugate where a bar is inconvenient to use.
We prove the results separately for each discretization.

\underline{\textbf{FD case:}} Recalling the symbol from \cref{thm:1d-symbol-FD-copy}, the squared magnitude of $\wt{s}$ can be written as
\begin{align} \label{eq:s-mag-sq-FD}
    |\wt{s}|^2 
    =
    \wt{s} \wt{s}^{\dagger}
    = 
    \left[ 
    \wt{s}_0 + \wt{s}_1 \epsilon + {\cal O}(\epsilon^2) 
    \right]
    \left[ 
    \wt{s}_0^{\dagger} + \wt{s}_1^{\dagger} \epsilon + {\cal O}(\epsilon^2) 
    \right]
    =
    g_0
    +
    g_1
    \epsilon 
    + {\cal O}(\epsilon^2),
\end{align}
with $g_1(\omega_1, \omega_2) := \wt{s}_0 \wt{s}_1^{\dagger} + \wt{s}_0^{\dagger} \wt{s}_1$, and 
\begin{align} \label{eq:r-def}
    g_0(\omega_1) 
    :=
    \wt{s}_0 \wt{s}_0^{\dagger}
    =
    \frac{a^{\ell}}{(1 + \ell) - \ell \bar{a}} \frac{a^{-\ell}}{(1 + \ell) - \ell a}
    =
    \frac{1}{2\ell (\ell +1) (1 - \cos \omega_1) + 1}.
\end{align}
Recalling \cref{def:mu}, we compute the smoothing factor by maximizing the function \eqref{eq:s-mag-sq-FD} on the high-frequency domain, i.e., $(\omega_1, \omega_2) \in [-\pi/2, 3\pi/2)^2 \setminus [-\pi/2, \pi/2)^2$.
On this domain, the function $|\wt{s}(\omega_1, \omega_2)|^2 = g_0(\omega_1) + \epsilon g_1(\omega_1, \omega_2) + {\cal O}(\epsilon^2)$ is maximized either on a boundary or at a critical point in the interior of the domain. 

For the moment, let us analyze the behavior of $|\wt{s}|^2$ with respect to $\omega_1$ only, ignoring any $\omega_2$ dependence by writing $|\wt{s}(\omega_1; \omega_2)|^2 = g_0(\omega_1) + \epsilon g_1(\omega_1;\omega_2) + {\cal O}(\epsilon^2)$.
The $\omega_1$ boundary segments here are $\omega_1^{\rm bndry} \in { \{ -\pi/2, \pi/2, 3 \pi/2 \}} $; observe that $g_0(-\pi/2) = g_0(\pi/2) = g_0(3 \pi/2) = \frac{1}{2 \ell(\ell + 1) + 1} \leq \frac{1}{5}$ for $\ell \geq 1$.
As such, on the $\omega_1$ boundary segments we have $|\wt{s}(\omega_1^{\rm bndry}; \omega_2)|^2 \leq \frac{1}{5} + \epsilon g_1(\omega_1^{\rm bndry};\omega_2) + {\cal O}(\epsilon^2)$.
On the other hand, observe that $g_0(0) = 1$, so that $|\wt{s}(0; \omega_2)|^2 = 1 + \epsilon g_1(0; \omega_2) + {\cal O}(\epsilon^2)$.
Hence, for $\epsilon$ sufficiently small it holds that $|\wt{s}(0; \omega_2)|^2 > |\wt{s}(\omega_1^{\rm bndry}; \omega_2)|^2$, and therefore, with respect to $\omega_1$, there is at least one interior point at which the function is larger than all points on its boundary. 
That is, for $\epsilon$ sufficiently small, the maximum of $|\wt{s}(\omega_1; \omega_2)|^2$ must occur at a critical point in the interior of the domain.
To this end, we analyze critical points of $|\wt{s}(\omega_1; \omega_2)|^2$ with respect to $\omega_1$.

Now we invoke \cref{lem:crit-point-asymptotic} which tells us that the critical points of $|\wt{s}(\omega_1; \omega_2)|^2$ are ${\cal O}(\epsilon)$ away from those of $g_0(\omega_1)$.
From \eqref{eq:r-def}, observe that $g_0(\omega_1)$ has two critical points $\omega_1 = \omega_1^*$ on $\omega_1 \in [-\pi/2, 3 \pi/2)$ with $\omega_1^* = 0$ a global maximum and $\omega_1^* = \pi$ a global minimum.
We are interested in the maximum at $\omega_1 = 0$, since this must be ${\cal O}(\epsilon)$ away from the global maximum of $|\wt{s}(\omega_1; \omega_2)|^2$.
Applying \cref{lem:crit-point-asymptotic} we have that the derivative of $|\wt{s}(\omega_1; \omega_2)|^2$ vanishes at some point $\omega_1^* = 0 + {\cal O}(\epsilon)$, with the function taking on the value $|\wt{s}(\omega_1^*; \omega_2)|^2 = 1 + \epsilon g_1(0; \omega_2) + {\cal O}(\epsilon^2)$.
We therefore have the following global maximum:
\begin{align} \label{eq:smooth-FD-max-temp}
    \max \limits_{ \omega_1 \in [-\pi/2, 3 \pi/2) } 
    |\wt{s}(\omega_1; \omega_2)|^2
    =
    1 +  \epsilon g_1(0; \omega_2) + {\cal O}(\epsilon^2).
\end{align}
We now compute $g_1(0; \omega_2) = \wt{s}_0(0) \wt{s}_1^{\dagger}(0; \omega_2) + \wt{s}_0^{\dagger}(0) \wt{s}_1(0; \omega_2)$.
Recall from \cref{thm:1d-symbol-FD-copy} that $\wt{s}_0(\omega_1) = \frac{a^{\ell}}{1 + \ell - \ell \bar{a}}$, with $a(\omega_1) = e^{\i \omega_1}$.
Clearly $\wt{s}_0(0) = 1$ because $a(0) = 1$.
Further recall 
$\wt{s}_1(\omega_1; \omega_2)
=
-
a^{-\ell} 
\wt{s}_0
\left[
    \frac{1}{3} \ell (\ell + 1) [3 \bar{a} + (1 - \bar{a}) (\ell + 2)] \wt{s}_0
    +
    (e^{\i \omega_2} + \wt{s}_0 e^{-\i \omega_2}) z_1(a) 
\right]$,
with $z_1(1) = -\frac{1}{2} \ell(\ell+1)$ given in \cref{lem:Bx=b-solutions}.
Simplifying we have 
$\wt{s}_1(0; \omega_2) = - 1 \cdot 1 \cdot \left[
\frac{1}{3} \ell (\ell + 1) [3  + 0] \cdot 1
+
(e^{\i \omega_2} + 1 \cdot e^{-\i \omega_2}) z_1(0) \right]
=
- \left[ \ell (\ell + 1) - \ell(\ell+1) \cos \omega_2  \right]
=
- \ell (\ell + 1) (1 - \cos \omega_2) 
$.
Plugging into \eqref{eq:smooth-FD-max-temp} we have
\begin{align}
    \max \limits_{ \omega_1 \in [-\pi/2, 3 \pi/2) } 
    |\wt{s}(\omega_1; \omega_2)|^2
    =
    1   - 2 \ell (\ell + 1) (1 - \cos \omega_2) \epsilon + {\cal O}(\epsilon^2).
\end{align}

Now we maximize over $\omega_2$. Because the global maximum of $|\wt{s}(\omega_1, \omega_2)|^2$ occurs on the line $\omega_1 = \omega_1^* = 0 + {\cal O}(\epsilon)$, the domain for the global maximum with respect to $\omega_2$ is $\omega_2 \in [\pi/2, 3 \pi/2]$, recalling that the high frequencies are $(\omega_1, \omega_2) \in [-\pi/2, 3\pi/2)^2 \setminus [-\pi/2, \pi/2)^2$.
Notice that on this domain the function $- 2\ell (\ell + 1) (1 - \cos \omega_2)$ is maximized at either end point where $\cos \omega_2 = 0$. 
Plugging in this result and then taking the square root and applying that $\sqrt{1 + g \epsilon + {\cal O}(\epsilon^2)} = 1 + \frac{1}{2} g \epsilon + {\cal O}(\epsilon^2)$ concludes the proof for the FD case.

\underline{\textbf{FE case:}} 
Recall from \cref{thm:1d-symbol-FE-copy} that the symbol for the FE discretization is
$\wt{s}_{\ell,1}(\omega_1, \omega_2) = \wt{s}_0(\omega_1, \omega_2) + \wt{s}_{1}(\omega_1, \omega_2) \epsilon  + {\cal O}(\epsilon^2)$,
where $\wt{s}_0(\omega_1, \omega_2) = \frac{a^{\ell} - \bar{\delta}_0  z_1(a)}{1 + \ell - \ell \bar{a} + \delta_0 z_1(a)}$, 
with $a(\omega_1) = e^{\i \omega_1}$, $\delta_0(\omega_1, \omega_2) = \frac{1}{2} e^{-\i \omega_2} (\cos \omega_1 - 1)$, $z_1(a)$ defined in \cref{lem:B0-solves},
and
$\wt{s}_1(\omega_1, \omega_2)$ a function of $\omega_1$ and $\omega_2$ satisfying $\wt{s}_{1}(0, \omega_2) = - \frac{3}{2} \ell (\ell + 1)(1 - e^{-i \omega_2})$.

Observe that $\delta_0(0, \omega_2) = 0$ and $a(0) = 1$, such that $\wt{s}_0(0, \omega_2) = \frac{1 - 0 \cdot z_1(1)}{1 + \ell - \ell \cdot 1 + 0 \cdot z_1(1)} = 1$ for all $\omega_2$.
As such, $\wt{s}_{\ell,1}(0, 3 \pi/2) = 1 + \wt{s}_{1}(0, 3 \pi/2) \epsilon  + {\cal O}(\epsilon^2) = 1 - \frac{3}{2} \ell (\ell + 1) \epsilon + {\cal O}(\epsilon^2)$.

Recall from \cref{def:mu} the smoothing factor 
    $\mu_{\ell, 1} := \max_{(\omega_1, \omega_2) \in [-\pi/2, 3\pi/2)^2 \setminus [-\pi/2, \pi/2)^2} | \wt{s}_{\ell, 1}(\omega_1, \omega_2) |$. 
Thus, for any $(\omega_1, \omega_2) \in [-\pi/2, 3\pi/2)^2 \setminus [-\pi/2, \pi/2)^2)$, $| \wt{s}(\omega_1, \omega_2) |$ gives a lower bound on the smoothing factor. In this case we use the high-frequency pair $(\omega_1, \omega_2) = (0, 3 \pi/2)$.     
This concludes the proof for the FE case.
    
\end{proof}

\begin{remark}[On the lower bound \eqref{eq:1d-smooth-FE-copy} holding with equality] \label{rem:smooth-FE-equality}
    Extensive numerical tests indicate that the lower bound $\mu_{\ell, 1}(\epsilon)
    \geq
        1 - \frac{3}{2} \ell (\ell + 1) \epsilon
    + {\cal O}(\epsilon^2)$ from \cref{thm:1d-smooth-copy} for the smoothing factor of the FE discretization holds with equality; see e.g., the right-hand side of \cref{fig:LFA-theory-support}. 
    However, we have not been able to prove this. 
    One line of proof follows the steps used to prove the FD case in \cref{thm:1d-smooth-copy}.
    Specifically, the result easily follows provided we could establish that the magnitude of the zeroth-order term \eqref{eq:FE-s0-copy} in the FE symbol has a maximum of unity and that this maximum is uniquely achieved with respect to $\omega_1$ at $\omega_1 = 0$.
    Recall that the zeroth-order term in question is $\wt{s}_0(\omega_1, \omega_2) = \frac{a^{\ell} - \bar{\delta}_0  z_1(a)}{1 + \ell - \ell \bar{a} + \delta_0 z_1(a)}$, and from the proof of \cref{thm:1d-smooth-copy} that $\wt{s}_0(0, \omega_2) = 1$ for all $\omega_2$. Hence, we would need to show this point is a global maximum and that the maximum is unique.
    Observe the following upper bound on the magnitude squared of $\wt{s}_0$ due to the triangle inequality:
    \begin{align} \label{eq:FE-s0-mag-sq-bound}
        |\wt{s}_0(\omega_1, \omega_2)|^2
        =
        \left(
        \frac{a^{\ell} - \bar{\delta}_0  z_1(a)}{1 + \ell - \ell \bar{a} + \delta_0 z_1(a)}
        \right)
        \left(
        \frac{a^{\ell} - \bar{\delta}_0  z_1(a)}{1 + \ell - \ell \bar{a} + \delta_0 z_1(a)}
        \right)^{\dagger}
        \leq
        \frac{(1 + |\delta_0  z_1(a)|)^2}{(|1 + \ell - \ell \bar{a}| - |\delta_0 z_1(a)|)^2},
    \end{align}
    with $\dagger$ denoting the complex conjugate.
    Plots of this upper bound are shown in \cref{fig:FE-upper-bound}.
   For all $\ell$ shown in the plots, this upper bound appears to have have a global maximum of unity at $\omega_1 = 0$, and for $\ell > 1$ this maximum is unique.
    For $\ell = 1$ there is also a maximum of one at $\omega_1 = \pi$, but it is straightforward to show that the upper bound exceeds $|\wt{s}_0(\omega_1, \omega_2)|^2$ at this particular point.
    While the desired behavior seems to be true, proving that it holds is seems not straightforward due to the complicated nature of the functions in the upper bound \eqref{eq:FE-s0-mag-sq-bound}.
\end{remark}

\begin{figure}[t!]
    \centering
    \includegraphics[width=0.4\linewidth]{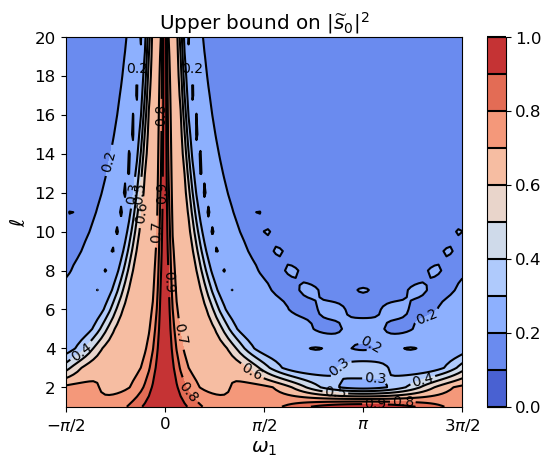}
    \includegraphics[width=0.4\linewidth]{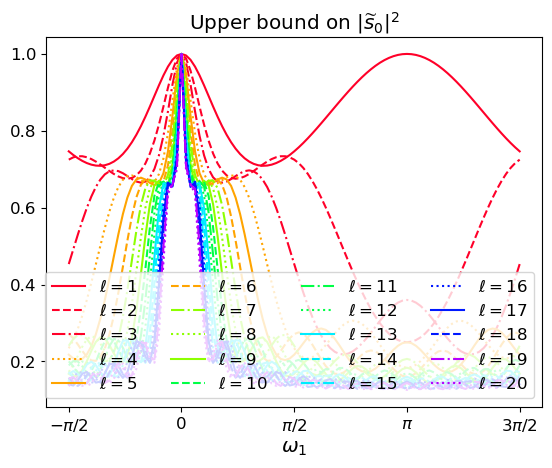}
    \caption{The upper bound \eqref{eq:FE-s0-mag-sq-bound} as a function of $\omega_1 \in [-\pi/2, 3\pi/2]$ and of $\ell \in \{1, 2, \ldots, 20\}$.
    \textbf{Left:} Contour over $\ell$ and $\omega_1$.
    \textbf{Right:} Cross-sections for integer $\ell$.
    \label{fig:FE-upper-bound}
    }
\end{figure}

\begin{corollary}[Copy of \cref{cor:robust-1d}. $\epsilon$-robustness] \label{cor:robust-1d-copy}
    Consider the class of multiplicative Schwarz smoothers with maximally overlapping subdomains of size $\ell \times 1$, $\ell \in \mathbb{N}$, and the notion of $\epsilon$-robustness in \cref{def:robust}.
    Then, the following statements hold for the FD \eqref{eq:rot-fd} and FE \eqref{eq:rot-fe} discretizations of grid-aligned anisotropic diffusion \eqref{eq:aligned}:

    \begin{enumerate}
        \item For any fixed $\ell \in \mathbb{N}$ the smoother is not $\epsilon$-robust, with $\lim_{\epsilon \to 0^+} \mu_{\ell,1} = 1$ in the FD case, and $\lim_{\epsilon \to 0^+} \mu_{\ell,1} \geq 1$ in the FE case.
        \item A necessary condition for $\epsilon$-robustness of this class is that $\ell$ increases at least as fast as ${\cal O}(\epsilon^{-1/2})$ as $\epsilon \to 0^+$.

        \item Consider the FD case only. Let $\mu_* \in (0, 1)$ be a smoothing factor for which the ${\cal O}(\epsilon^2)$ terms in \eqref{eq:1d-smooth-FD-copy} are uniformly small with respect to $\ell$, i.e., $| \mu_* - [1 - \ell(\ell + 1) \epsilon ] | < 1$, for all $\ell \in \mathbb{N}$.
    Then, for the FD discretization this class of smoothers has an $\epsilon$-independent smoothing factor of $\mu_{\ell, 1} = \mu_*$ if the subdomain size satisfies
    \begin{align} \label{eq:ell-star-copy}
        \ell 
        = 
        \ell_*(\epsilon) 
        = 
        \left\lceil
        \sqrt{1 - \mu_*} \epsilon^{-1/2} + {\cal O}(\epsilon^0)
        \right\rceil. 
    \end{align}
    \end{enumerate}
\end{corollary}

\begin{figure}[b!]
    \centering
    \includegraphics[width=0.4\linewidth]{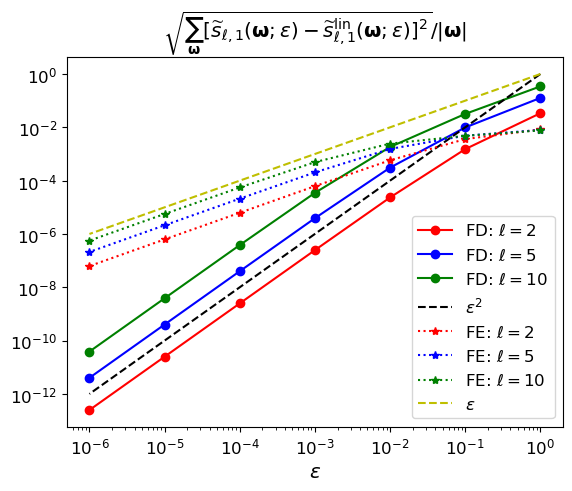}
    \includegraphics[width=0.4\linewidth]{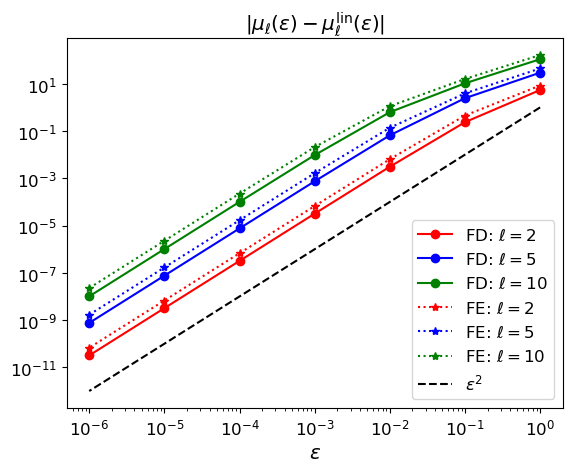}
    \caption{\textbf{Left:} Supporting numerical evidence for \cref{thm:1d-symbol-FD-copy,thm:1d-symbol-FE-copy}. 
    For fixed $\ell \in \{2, 5, 10\}$, this plot shows the difference between the ``true'' symbols $\wt{s}_{\ell,1}$ obtained via numerically computing \eqref{eq:SCH-1d-LFA-sys} and those of $\wt{s}_{\ell,1}^{\rm lin}$ posited in \cref{thm:1d-symbol-FD,thm:1d-symbol-FE} as a function of $\epsilon$. 
    That is, in the FD case $\wt{s}_{\ell,1}^{\rm lin}$ is $\wt{s}_0 + \epsilon \wt{s}_1$ from \cref{thm:1d-symbol-FD-copy}, and in the FE case $\wt{s}_{\ell,1}^{\rm lin}$ is $\wt{s}_0$ from \cref{thm:1d-symbol-FE-copy}.
    The difference between these functions appears to scale as ${\cal O}(\epsilon^2)$ and ${\cal O}(\epsilon)$ for the FD and FE cases, respectively, which is consistent with the estimates in \cref{thm:1d-symbol-FD,thm:1d-symbol-FE} being correct.
    For fixed $\ell$, and fixed $\epsilon$, the symbols are sampled on an $\omega_1 \times \omega_2$ grid with 64 equispaced points in each direction.
    \textbf{Right:} Supporting numerical evidence for \cref{thm:1d-smooth-copy}.
    For fixed $\ell \in \{2, 5, 10\}$, this plot shows the difference between the ``true'' smoothing factor $\mu_{\ell,1}$ obtained by numerically solving \eqref{eq:mu-def} with symbol in \eqref{eq:SCH-ellx1-symbol-def}, and the estimates $\mu_{\ell,1}^{\rm lin} = 1 - c \ell (\ell + 1) \epsilon$ from \cref{thm:1d-smooth}, with $c = 1$ and $c = 3/2$ for the FD and FE cases, respectively.
    This difference appears to scale as ${\cal O}(\epsilon^2)$, which is consistent with \cref{thm:1d-smooth-copy} being correct, and the lower bound in \eqref{eq:1d-smooth-FE-copy} holding with equality.
    \label{fig:LFA-theory-support}
    }
\end{figure}

\begin{proof}
    Statement 1 follows trivially by taking $\epsilon \to 0^+$ in \cref{thm:1d-smooth-copy}.

    Statements 2 and 3 can be proved simultaneously. Since $\epsilon$-robustness means that the smoothing factor is uniformly bounded above by one, to prove statement 2 in the FD case we set the smoothing factor in \cref{thm:1d-smooth-copy} to some $\epsilon$-independent constant smaller than one and then rearrange for the $\ell = \ell(\epsilon)$ for which this holds.
    The proof for the FE case works analogously, except that we set the lower bound of the smoothing factor equal to an $\epsilon$-independent constant.
    To prove statement 3 we do the same thing where $\mu_*$ is this $\epsilon$-independent constant.
    Recalling that the smoothing factor expression and bound from \cref{thm:1d-smooth-copy} are of the form $\mu_{\ell}(\epsilon) = 1 - c \ell (\ell + 1) \epsilon + {\cal O}(\epsilon^2)$ with constant $c$ we have
    \begin{align}
        1 - c \ell_* (\ell_* + 1)\epsilon + {\cal O}(\epsilon^2) = \mu_*
        \quad
        \rightarrow
        \quad
        \ell_*^2 + \ell_* + \left( \frac{\mu_* - 1 + {\cal O}(\epsilon^2)}{c \epsilon} \right) = 0.
    \end{align}
    Solving for $\ell_*$ using the quadratic formula gives
    \begin{align}
        \ell_* 
        = 
        \frac{1}{2} 
        \left[
        -1 \pm \sqrt{1 - 4 \left( \frac{\mu_* - 1 + {\cal O}(\epsilon^2)}{c \epsilon} \right)} 
        \right]
        = \frac{1}{2} 
        \left[
        -1 \pm \sqrt{
        4\frac{1 - \mu_*}{c \epsilon}
        + 1 + {\cal O}(\epsilon)} 
        \right]
        = \frac{1}{2} 
        \left[
        -1 \pm 2\sqrt{\frac{1 - \mu_*}{c \epsilon}} \sqrt{
        1
        + 
        {\cal O}(\epsilon) 
        }
        \right].
    \end{align}
    Applying Taylor series to the square root gives $\sqrt{1 + {\cal O}(\epsilon)} = 1 + {\cal O}(\epsilon)$. Discarding the negative root (since $\ell_* > 0$) and taking the ceiling of the result (since $\ell_*$ is an integer) gives the claim.
    
\end{proof}

Finally, since the theoretical results in this section have been rather tedious to derive, we offer some supporting numerical evidence for their correctness in \cref{fig:LFA-theory-support}.

\end{document}